\documentclass{siamart251216}

\usepackage{amssymb,mathtools,bm}
\usepackage{enumitem}
\usepackage{algpseudocode}
\usepackage{booktabs}
\usepackage{graphicx}
\usepackage{placeins}
\usepackage{subcaption}
\usepackage{tikz}
\usetikzlibrary{arrows.meta,calc,positioning}

\definecolor{expGreen}{HTML}{2CA02C}
\definecolor{expBlue}{HTML}{1F77B4}
\definecolor{expRed}{HTML}{D62728}
\definecolor{expYellow}{HTML}{BCBD22}

\newsiamthm{assumption}{Assumption}
\newsiamthm{claim}{Claim}
\newsiamremark{remark}{Remark}

\AddToHook{env/remark/begin}{\crefalias{theorem}{remark}}
\newcommand{\suppref}[1]{Appendix~\ref{#1}}

\newcommand{\mtx}[1]{\mathbf{#1}}
\newcommand{\msf}[1]{\mathsf{#1}}
\newcommand{\norm}[1]{\left\lVert #1 \right\rVert}
\newcommand{\normtwo}[1]{\norm{#1}_{2}}
\newcommand{\normfro}[1]{\norm{#1}_{F}}
\newcommand{\R}{\mathbb{R}}
\newcommand{\C}{\mathbb{C}}
\newcommand{\field}{\mathbb{F}}

\newcommand{\rank}{\mathrm{rank}}

\newcommand{\range}{\operatorname{range}}

\newcommand{\pinv}[1]{#1^\dagger}
\newcommand{\diag}{\operatorname{diag}}

\newcommand{\E}{\mathbb{E}}

\newcommand{\near}{\mathcal{N}}
\newcommand{\blockrestrict}[2]{\left[#1\right]_{#2}}

\title{Quasi-Optimal Randomized Recovery of \texorpdfstring{$\mathcal{H}^2$}{H2} Matrices}
\author{Anna Yesypenko\thanks{The Department of Mathematics, The Ohio State
University (\email{yesypenko.1@osu.edu}).}}
\headers{Quasi-Optimal Randomized Recovery of \texorpdfstring{$\mathcal{H}^2$}{H2} Matrices}{A. Yesypenko}

\begin{document}

\maketitle

\begin{abstract}
An $\mathcal H^2$ matrix represents well-separated interactions using nested
low-rank bases while retaining nearby interactions explicitly.  We consider
the problem of constructing an explicit, strongly admissible $\mathcal H^2$
approximation using only products with the matrix and its adjoint, and develop
two randomized recovery algorithms on an adaptive geometric
hierarchy.

Carried recovery begins with a single pair of dense Gaussian test
matrices, without hierarchy-specific probing, and updates the resulting
forward--adjoint sketches across levels. All operator samples can
therefore be acquired in one batch. Resketched recovery applies the
same recovery procedure using fresh Gaussian test matrices at each
level. For this variant, we prove a global quasi-optimal bound on the
expected Frobenius error relative to the best approximation with the
prescribed hierarchy, rank bounds, and retained near-neighbor pattern.

Carried recovery requires only $O(k+p)$ applications of the matrix
and its adjoint, independent of tree depth, whereas resketched recovery
requires $O((k+p)\log(N/k))$ on trees of logarithmic depth.
Here $N$ is the matrix dimension, $k$ the maximum prescribed
box rank, and $p$ the oversampling parameter. These bounds assume
bounded near-neighbor counts and $O(k)$ local dimensions.  Experiments on Laplace, Dirichlet-to-Neumann, and Helmholtz operators in two and three dimensions
with up to $1.28$ million unknowns show that, with sufficient
oversampling, carried recovery achieves nearly refinement-independent
error and accuracy comparable to resketched recovery while using
many fewer operator applications. Its compressed application
cost scales linearly for a fixed rank.
\end{abstract} 

\section{Introduction}
\label{sec:intro}

Large dense matrices arise throughout scientific computing, particularly from discretizations of nonlocal operators. Although dense, many such matrices inherit  structure that admits efficient approximate representations. Partitioning the degrees of freedom geometrically exposes numerically low-rank interactions between well-separated groups, while the comparatively few nearby interactions are stored densely. The fast multipole method \cite{greengard1987fast,greengard1997new} and hierarchical matrix representations \cite{borm2003introduction,hackbusch1999sparse,hackbusch2015hierarchical} exploit this distinction. Hierarchical matrix formats can reduce storage and
matrix--vector application costs to nearly linear complexity. For suitable problems, the
$\mathcal H^2$ format further exploits nested low-rank structure across scales, enabling linear storage and matrix--vector product cost \cite{ballani2016hierarchical,borm2010efficient,hackbusch2002data}.

In many applications, however, the matrix is not explicitly available. Instead, one has access only to routines for applying $\mtx A$ and $\mtx A^*$, for example through partial differential equation (PDE) or integral-equation solvers, or other implicit operator evaluations. Individual matrix entries may be expensive to obtain or unavailable altogether. This motivates the matrix-free recovery problem considered here: construct an explicit $\mathcal H^2$ approximation using only matrix--vector products with $\mtx A$ and $\mtx A^*$, while minimizing the number of such products.

The difficulty is that hierarchical structure is local and multiscale,
whereas a matrix--vector product is global.  Multiplication by a dense test matrix
simultaneously mixes interactions across multiple spatial scales.
Multilevel recovery methods can separate these interactions by
processing the hierarchy one level at a time. However, forming
fresh sketches at successive levels makes the required number
of matrix--vector products grow with the depth of the hierarchy.  This cost is particularly important when each application of
$\mtx A$ or $\mtx A^*$ requires an expensive numerical solve.

\pagebreak{}

We develop two closely related randomized algorithms for matrix-free
$\mathcal H^2$ recovery on an adaptive geometric hierarchy. Our contributions are:

\begin{itemize}

\item[\textbf{C1.}]
\textbf{Carried randomized recovery.}
We introduce carried recovery, which constructs an explicit
$\mathcal H^2$ representation from one forward--adjoint sketch pair.
The initial test matrices are fully dense Gaussian matrices and require no
hierarchy-specific probing pattern.  Under the geometric assumptions and rank conditions in
\Cref{subsec:recovery-complexity}, the method requires
$O(k+p)$ queries to each of $\mtx A$ and $\mtx A^*$,
independent of tree depth. Here $k$ is the maximum prescribed box rank and $p$ is
the oversampling parameter.

\item[\textbf{C2.}]
\textbf{Resketched recovery with a global quasi-optimal guarantee.}
Resketched recovery draws fresh Gaussian test matrices at each level.
For a prescribed hierarchy, rank bounds, and nearby interactions,
we prove
\[
    \mathbb E \normfro{\mtx A-\widetilde{\mtx A}}
    \le
    \sqrt{2(L-\ell_0+1)}\,\beta\,\mathrm{OPT},
\]
where $\mathrm{OPT}$ is the best Frobenius-norm approximation error in the
prescribed $\mathcal H^2$ class and $\beta$ is an explicit local recovery
factor.  Under fixed geometric bounds and oversampling proportional to rank,
$\beta$ remains bounded.  On trees of logarithmic depth, resketched recovery uses
$O((k+p)\log(N/k))$ queries of $\mtx A$ and $\mtx A^*$. Here, $L-\ell_0 + 1$ is the number of processed levels and $N$ is the matrix dimension.

\item[\textbf{C3.}]
\textbf{A controlled numerical comparison of sample reuse.}
We evaluate the two algorithms on Laplace, Dirichlet-to-Neumann, and Helmholtz operators with up to \(1.28\) million unknowns. The methods share the same multilevel recovery procedure and differ primarily
in how samples are supplied across levels, so the experiments directly compare
fresh and reused sampling. With sufficient oversampling, carried recovery attains accuracy comparable to resketched recovery while using fewer operator applications.

\item[\textbf{C4.}]
\textbf{A telescoping multilevel approximation on adaptive trees.}
We reorganize the hierarchical approximation into a telescoping form in 
which the same one-level
recovery problem appears at every scale, including on adaptive trees with
leaves at different depths.  With nested orthonormal bases, the resulting
levelwise contributions are mutually orthogonal in the Frobenius inner product.
At each level, only neighboring interactions are retained explicitly, while long-range interactions are represented on coarser levels through the
nested bases.  For fixed box bases, we prove that this
representation is an exact regrouping of the conventional adaptive $\mathcal H^2$ representation.

\begin{figure}[H]
\centering
\resizebox{0.92\linewidth}{!}{
\begin{tikzpicture}[
    every node/.style={align=center,font=\large},
    process/.style={-{Stealth[length=2.2mm]},line width=0.8pt},
    page/.style={
        draw=black,
        fill=white,
        line width=0.9pt,
        inner sep=0.3mm,
    },
]
    \node (geom)
        {\includegraphics[width=34mm]
        {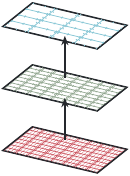}};
    \node[anchor=east]
        at ($(geom.south west)!0.14!(geom.north west)+(-1mm,0)$)
        {$L$};
    \node[anchor=east]
        at ($(geom.south west)!0.50!(geom.north west)+(-1mm,0)$)
        {$L-1$};
    \node[anchor=east]
        at ($(geom.south west)!0.68!(geom.north west)+(-1mm,0)$)
        {$\vdots$};
    \node[anchor=east]
        at ($(geom.south west)!0.86!(geom.north west)+(-1mm,0)$)
        {$\ell_0$};
    \node[fill=white,inner sep=0.55pt,anchor=east]
        at ($(geom.south west)!0.32!(geom.north west)+(-1mm,0)$)
        {Geometric\\ hierarchy};

    \node[page,right=6mm of geom] (physical)
        {\includegraphics[width=42mm]
        {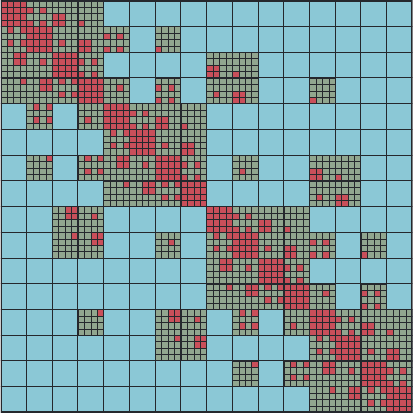}};
    \node[anchor=south,inner sep=0pt]
        at ([yshift=1.5mm]physical.north)
        {Hierarchical block partition};
    \node[font=\LARGE,fill=white,inner sep=0.8pt]
        at (physical.center) {$\mtx A$};

    \node[page,right=10mm of physical,yshift=14mm] (t-pattern)
        {\includegraphics[width=14mm]{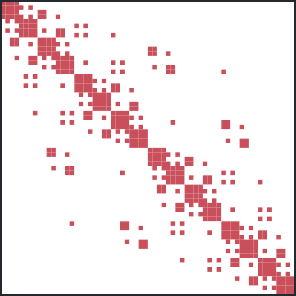}};
    \node at ($(physical.east)!0.5!(t-pattern.west)+(0,8.5mm)$)
        {$\approx$};

    \node[page,right=9mm of t-pattern] (d-pattern)
        {\includegraphics[width=14mm]{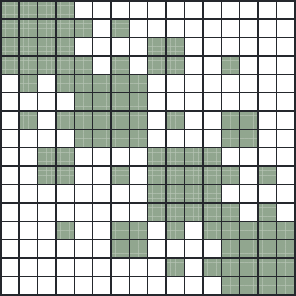}};
    \node at ($(t-pattern.east)!0.5!(d-pattern.west)$)
        {$+$};

    \node[page,right=20mm of d-pattern] (w-pattern)
        {\includegraphics[width=14mm]{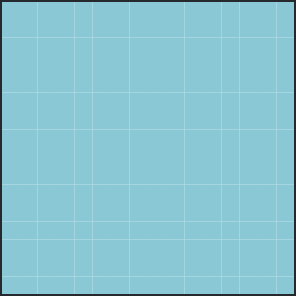}};
    \node at ($(d-pattern.east)!0.2!(w-pattern.west)$)
        {$+$};
    \node at ($(d-pattern.east)!0.5!(w-pattern.west)$)
        {$\cdots$};
    \node at ($(d-pattern.east)!0.8!(w-pattern.west)$)
        {$+$};
    \node[anchor=south,inner sep=0pt]
        at ($(t-pattern.north)!0.5!(w-pattern.north)+(0,1.5mm)$)
        {Multilevel approximation};

\begin{scope}[every node/.style={font=\large,inner sep=0.4pt}]
        \node[anchor=west]
            at ([yshift=-4.5mm]t-pattern.south west)
            {\textbf{Resketched:}};
        \node (resketched-fine)
            at ([yshift=-10.5mm]t-pattern.south)
            {$(\mtx\Omega_L,\mtx\Psi_L)$};
        \node
            at ([yshift=-10.5mm]d-pattern.south)
            {$(\mtx\Omega_{L-1},\mtx\Psi_{L-1})$};
        \node
            at ($(d-pattern.south)!0.5!(w-pattern.south)+(0,-10.5mm)$)
            {$\cdots$};
        \node (resketched-coarse)
            at ([yshift=-10.5mm]w-pattern.south)
            {$(\mtx\Omega_{\ell_0},\mtx\Psi_{\ell_0})$};

        \node[anchor=west]
            at ([yshift=-18mm]t-pattern.south west)
            {\textbf{Carried:}};
        \node (carried-tests)
            at ([yshift=-13.5mm]resketched-fine.center)
            {$(\mtx\Omega,\mtx\Psi)$};
        \draw[process]
            ([xshift=1.5mm]carried-tests.east) --
            ([yshift=-13.5mm]resketched-coarse.east);
    \end{scope}
\end{tikzpicture}
 }
\caption{The geometric hierarchy induces the recovered hierarchical approximation. Resketched recovery uses new Gaussian sketches at successive levels, while carried recovery updates one initial forward–adjoint sketch pair throughout the hierarchy.}
\label{fig:multilevel-telescope}
\end{figure}
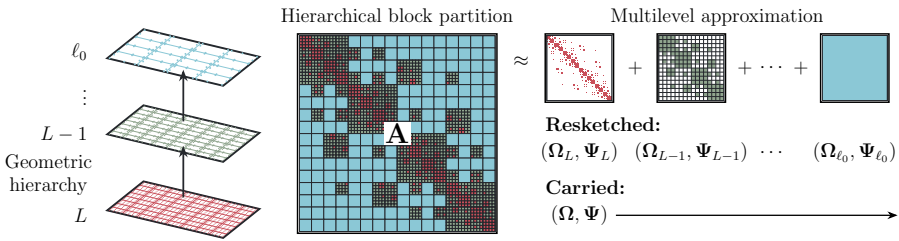

\end{itemize}
\subsection{Related work}

Hierarchical low-rank formats differ in both basis structure and admissibility.
An $\mathcal H$-matrix approximates admissible blocks independently, whereas
an $\mathcal H^2$-matrix uses bases shared across interactions and
nested across levels.  Weak admissibility treats off-diagonal interactions as
compressible, while strong admissibility compresses only well-separated pairs
and stores nearby interactions explicitly
\cite{ballani2016hierarchical,borm2010efficient,hackbusch1999sparse,hackbusch2002data}.  In two and three dimensions, weakly admissible
ranks typically grow under refinement, whereas strongly admissible
well-separated ranks can remain bounded for many nonoscillatory problems and
fixed-frequency Helmholtz problems
\cite{malhotra2015pvfmm,ying2004kernel}.  We focus on strong admissibility, with weak admissibility as a special case.

Randomized hierarchical reconstruction using matrix products includes
levelwise peeling for strongly admissible matrices \cite{lin2011fast} and its
graph-colored extension \cite{levitt2024coloring}, methods for HSS/HODLR
representations
\cite{liu2016parallel,martinsson2011fast,martinsson2016compressing}, and
adaptive or GPU-accelerated hierarchical constructions
\cite{boukaram2025adaptive,boukaram2019randomized,gorman2019matrixfree}.
Other operator-recovery methods exploit additional PDE structure, including
ellipticity and localization
\cite{boulle2023elliptic,schafer2024sparse}.

Strongly admissible peeling methods have the same
$O(k\log(N/k))$-type query complexity as resketched recovery under comparable
rank assumptions.  In these methods, interactions recovered at earlier levels
are subtracted before later levels are sampled, so earlier approximation errors
can enter subsequent sketches.  Resketched recovery instead samples each
current projected problem directly, producing an exact sketch of that problem
conditional on the previously recovered bases.

Two lines of work are particularly close to
the present work: randomized recovery for hierarchically semiseparable matrices (HSS) and randomized strong recursive skeletonization (RSRS). 
For weakly admissible HSS matrices, Levitt and Martinsson carry a global forward--adjoint sketch pair through
the hierarchy \cite{levitt2024linear}.  Amsel et al.\ analyze a levelwise-resketched variation of this
method and prove a global quasi-optimal guarantee in the
Frobenius norm \cite{amsel2026quasioptimal}. More broadly, quasi-optimal
approximation guarantees have been established for HODLR recovery and
for recovery with a prescribed sparsity pattern
\cite{amsel2026fixedsparsity,chen2025nearoptimal}.  To our knowledge, the
present work provides the first global quasi-optimal expected Frobenius-error
guarantee for matrix-free recovery of a strongly admissible
$\mathcal H^2$ approximation.

RSRS \cite{yesypenko2026randomized} is a black-box
extension of recursive skeletonization for 
strongly admissible $\mathcal H^2$ matrices \cite{minden2017recursive}.
Like the carried algorithm, it uses one dense Gaussian
forward-adjoint sketch pair throughout the hierarchy. From these sketches, 
RSRS constructs an approximate 
invertible factorization, whereas
the present work  constructs an explicit $\mathcal H^2$
approximation in orthonormal box coordinates.

Related telescoping decompositions appear in dual-space kernel splitting
\cite{jiang2025dmk}, are implicit in the simplified kernel-independent FMM
\cite{yesypenko2025skelfmm}, and are made explicit for weakly admissible HSS
matrices in \cite{casulli2024functions}.  Our formulation combines near-neighbor coefficient storage at each level with Frobenius orthogonality across levels.

\subsection{Organization}

\Cref{sec:one-level-prd} develops the one-level decomposition and recovery
procedure.  \Cref{sec:multilevel-prd} constructs the multilevel
representation, and \Cref{sec:algorithms} describes the two recovery algorithms,
their query complexity, and the resketched error guarantee.
\Cref{sec:experiments} presents the numerical experiments.  Appendix~\ref{app:proofs_and_supporting_estimates} provides supporting estimates and recovery proofs. 
\suppref{app:pseudocode} provides the $\mathcal H^2$ equivalence proof, matrix--vector application details, and pseudocode.

\paragraph{Notation}
Bold capitals denote matrices, and sans-serif capitals denote local box or
coordinate index sets.  For index sets $\msf I,\msf J$, $\mtx A_{\msf I,\msf J}$
denotes the submatrix of $\mtx A$ with rows in $\msf I$
and columns in $\msf J$. We write $\norm{\cdot}$ for the Euclidean norm of vectors,
and $\normtwo{\cdot}$ and $\normfro{\cdot}$ for the spectral
and Frobenius norms of matrices, respectively. We use 
$\range(\mtx M)$, $\operatorname{null}(\mtx M)$, and $\mtx M^\dagger$
for the column space, nullspace, and Moore--Penrose pseudoinverse of
$\mtx M$.
{}
All Gaussian test matrices introduced below have independent entries.
Over $\R$, the entries are standard real Gaussians; over $\C$, each entry
is $(g_1+\mathrm{i}g_2)/\sqrt{2}$, where $g_1$ and $g_2$ are independent
standard real Gaussians.
 \section{One-Level Decomposition and Recovery}
\label{sec:one-level-prd}

We begin with a fixed partition of the degrees of freedom into boxes.
The aim is to compress interactions between well-separated boxes
while retaining nearby interactions explicitly. We first describe
this approximation for given local bases, then show how to recover
the bases and the required coefficients using only products with
the matrix and its adjoint.

\subsection{One-level coordinate split}
\label{sec:one-level-basis-split}

To motivate the construction, consider a discretization of the two-dimensional
logarithmic kernel
\[
    \mtx A_{ij}=-{(2\pi)}^{-1}\ \log\norm{x_i-x_j},\qquad i\ne j,
\]
with the diagonal supplied by a local quadrature rule that treats the
singularity at $i=j$.  This is the free-space Green's function for the Laplace
equation in two dimensions.  Partitioning the physical points into boxes
induces a block partition of the matrix.  For sufficiently well-separated
boxes, the kernel is smooth and the corresponding matrix subblocks are
numerically low rank. 

To make this approximation precise, consider a general matrix
$\mtx A\in\field^{N\times N}$, with $\field\in\{\R,\C\}$, whose row and
column indices are partitioned into boxes.  For each box $\tau$, let
$\near(\tau)$ contain $\tau$ and its nearby boxes.  For an axis-aligned
spatial partition with strong admissibility, these are the boxes that touch
$\tau$ along a face, edge, or corner, as shown in
\cref{fig:one-level-recursion}.
Fix a box $\tau$.  Relative to $\tau$, write
\[
    \msf B := \tau,
    \qquad
    \msf N := \near(\tau)\setminus\{\tau\},
    \qquad
    \msf F := \text{all remaining boxes}.
\]
We refer to $\msf B$, $\msf N$, and $\msf F$ as the box, its near neighbors,
and its far field, respectively, and use their
corresponding index sets in matrix subscripts.

For each box $\tau$, fix a rank $1\le k_\tau\le m_\tau:=|\tau|$,
and choose orthonormal target and source bases
\(\mtx U_{\msf S_\tau},\mtx V_{\msf S_\tau} \in \field^{m_\tau\times k_\tau}
\)
such that
\begin{equation}
\label{eq:far-field-basis-capture}
    \mtx U_{\msf S_\tau} \mtx U_{\msf S_\tau}^* \mtx A_{\msf B,\msf F}
    \approx
    \mtx A_{\msf B,\msf F},
    \qquad
    \mtx A_{\msf F,\msf B} \mtx V_{\msf S_\tau} \mtx V_{\msf S_\tau}^*
    \approx
    \mtx A_{\msf F,\msf B}.
\end{equation}
Thus the \(\msf S_\tau\) coordinates capture the far-field block row and column of \(\tau\).

To separate the coordinates that capture far-field interactions
from the remaining coordinates, complete the target and source
bases with orthonormal complements
$\mtx U_{\msf R_\tau}$ and $\mtx V_{\msf R_\tau}$:
\[
    \mtx U_\tau
    =
    \left[
        \mtx U_{\msf S_\tau}\ \ \mtx U_{\msf R_\tau}
    \right],
    \qquad
    \mtx V_\tau
    =
    \left[
        \mtx V_{\msf S_\tau}\ \ \mtx V_{\msf R_\tau}
    \right].
\]
Here $\msf S_\tau$ and $\msf R_\tau$ index the coordinates
in the chosen bases and their orthogonal complements, respectively.
For a pair of boxes $\tau,\sigma$, the corresponding
interaction in these local coordinates is
\begin{equation}
\begin{aligned}
\label{eq:one-level-coordinate-block}
    \mtx C_{\tau,\sigma}
    :=
    \mtx U_\tau^* \mtx A_{\tau,\sigma} \mtx V_\sigma
    &=
    \begin{bmatrix}
        \mtx C_{\msf S_\tau,\msf S_\sigma}
        &
        \mtx C_{\msf S_\tau,\msf R_\sigma}
        \\
        \mtx C_{\msf R_\tau,\msf S_\sigma}
        &
        \mtx C_{\msf R_\tau,\msf R_\sigma}
    \end{bmatrix}\\
    &\approx
        \begin{bmatrix}
        \mtx C_{\msf S_\tau,\msf S_\sigma}
        &
        \mathmakebox[\widthof{$\mtx C_{\msf S_\tau,\msf R_\sigma}$}]
            {\mtx 0\vphantom{\mtx C_{\msf S_\tau,\msf R_\sigma}}}
        \\
        \mathmakebox[\widthof{$\mtx C_{\msf R_\tau,\msf S_\sigma}$}]
            {\mtx 0\vphantom{\mtx C_{\msf R_\tau,\msf S_\sigma}}}
        &
        \mathmakebox[\widthof{$\mtx C_{\msf R_\tau,\msf R_\sigma}$}]
            {\mtx 0\vphantom{\mtx C_{\msf R_\tau,\msf R_\sigma}}}
    \end{bmatrix},\ \text{if $\tau$ and $\sigma$ are well separated}.
    \end{aligned}
\end{equation}
If $\tau$ and $\sigma$ are neighbors, all four blocks are
retained.  If they are well separated, only
$\mtx C_{\msf S_\tau,\msf S_\sigma}$ is retained.  Accurate far-field bases make the discarded \(\msf{SR}\), \(\msf{RS}\), and \(\msf{RR}\) blocks small.
Collecting the $\msf{SS}$ blocks over all current box pairs gives a smaller
matrix which is resolved at the next coarser level.  The retained \(\msf{R}\)-containing blocks record the nearby interactions resolved at the current level; see
\Cref{fig:one-level-recursion}.

\begin{figure}[t]
    \centering
    \resizebox{0.88\linewidth}{!}{\begin{tikzpicture}[
    every node/.style={align=center},
    process/.style={-{Stealth[length=2.1mm]},line width=0.8pt},
    matrix page/.style={
        draw=black,
        fill=white,
        line width=1.0pt,
        inner sep=0.2mm,
    },
]
    \definecolor{levelThreeRed}{HTML}{C94F5A}
    \definecolor{levelTwoGreen}{HTML}{91A68F}
    \definecolor{coarseCyan}{HTML}{8BC8D6}
    \definecolor{matrixGrid}{HTML}{25282D}

\node[matrix page] (physical)
        {\includegraphics[width=31mm]
        {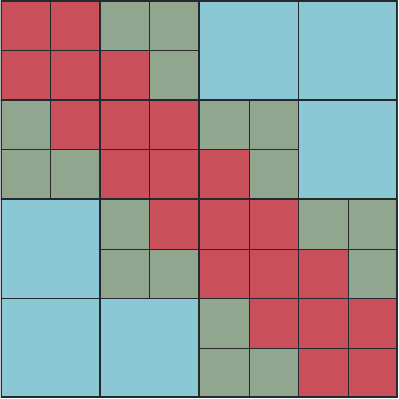}};
    \foreach \i in {1,...,15} {
        \pgfmathsetmacro{\t}{\i/16}
        \draw[white,opacity=0.52,line width=0.18pt]
            ($(physical.south west)!\t!(physical.south east)$) --
            ($(physical.north west)!\t!(physical.north east)$);
        \draw[white,opacity=0.52,line width=0.18pt]
            ($(physical.south west)!\t!(physical.north west)$) --
            ($(physical.south east)!\t!(physical.north east)$);
    }
    \node[font=\normalsize,fill=white,inner sep=0.8pt]
        at (physical.center) {$\mtx A$};
\node[matrix page,right=44mm of physical] (transformed)
        {\begin{tikzpicture}[x=1.94mm,y=1.94mm]
            \fill[white] (0,0) rectangle (16,16);
            \foreach \i in {0,...,7} {
                \foreach \j in {0,...,7} {
                    \pgfmathtruncatemacro{\childnear}{abs(\i-\j)<=1}
                    \pgfmathtruncatemacro{\parentfar}{
                        abs(floor(\i/2)-floor(\j/2))>1
                    }
                    \ifnum\parentfar=1
                        \def\corecolor{coarseCyan}
                    \else
                        \def\corecolor{levelTwoGreen}
                    \fi
                    \fill[\corecolor]
                        (\j,{15-\i}) rectangle ({\j+1},{16-\i});
                    \ifnum\childnear=1
                        \fill[levelThreeRed]
                            ({8+\j},{15-\i})
                            rectangle ({9+\j},{16-\i});
                        \fill[levelThreeRed]
                            (\j,{7-\i}) rectangle ({\j+1},{8-\i});
                        \fill[levelThreeRed]
                            ({8+\j},{7-\i})
                            rectangle ({9+\j},{8-\i});
                    \fi
                }
            }
            \foreach \g in {0,...,16} {
                \draw[matrixGrid,line width=0.18pt]
                    (\g,0) -- (\g,16);
                \draw[matrixGrid,line width=0.18pt]
                    (0,\g) -- (16,\g);
            }
            \draw[matrixGrid,line width=0.9pt] (8,0) -- (8,16);
            \draw[matrixGrid,line width=0.9pt] (0,8) -- (16,8);
            \draw[matrixGrid,line width=1.0pt]
                (0,0) rectangle (16,16);
        \end{tikzpicture}};

    \node[font=\scriptsize,anchor=south]
        at ([xshift=-7.75mm,yshift=0.8mm]transformed.north)
        {$\msf S$};
    \node[font=\scriptsize,anchor=south]
        at ([xshift=7.75mm,yshift=0.8mm]transformed.north)
        {$\msf R$};
    \node[font=\scriptsize,anchor=east]
        at ([xshift=-0.8mm,yshift=7.75mm]transformed.west)
        {$\msf S$};
    \node[font=\scriptsize,anchor=east]
        at ([xshift=-0.8mm,yshift=-7.75mm]transformed.west)
        {$\msf R$};

    \node[font=\scriptsize,fill=white,fill opacity=0.90,
        text opacity=1,inner sep=0.5pt]
        at ([xshift=-7.75mm,yshift=7.75mm]transformed.center)
        {$\mtx C_{\msf S,\msf S}$};
    \node[font=\scriptsize,fill=white,fill opacity=0.90,
        text opacity=1,inner sep=0.5pt]
        at ([xshift=7.75mm,yshift=7.75mm]transformed.center)
        {$\mtx C_{\msf S,\msf R}$};
    \node[font=\scriptsize,fill=white,fill opacity=0.90,
        text opacity=1,inner sep=0.5pt]
        at ([xshift=-7.75mm,yshift=-7.75mm]transformed.center)
        {$\mtx C_{\msf R,\msf S}$};
    \node[font=\scriptsize,fill=white,fill opacity=0.90,
        text opacity=1,inner sep=0.5pt]
        at ([xshift=7.75mm,yshift=-7.75mm]transformed.center)
        {$\mtx C_{\msf R,\msf R}$};
    \node[font=\scriptsize,anchor=north]
        at ($(transformed.south)+(0,-1.2mm)$)
        {$\mtx C=\mtx U^*\mtx A\mtx V$};
    \coordinate (change) at ($(physical.east)!0.5!(transformed.west)$);
    \node[font=\scriptsize]
        at ($(change)+(0,11mm)$)
        {$\mtx U=\operatorname{diag}_\tau\mtx U_\tau$\\[1pt]
         $\mtx V=\operatorname{diag}_\tau\mtx V_\tau$};
    \draw[process] (physical.east) --
        node[above=1mm,font=\scriptsize] {$\mtx U^*(\,\cdot\,)\mtx V$}
        (transformed.west);

\coordinate (core-key) at ($(change)+(-17mm,-5.5mm)$);
    \fill[levelTwoGreen] (core-key) rectangle ++(1.3mm,2.6mm);
    \fill[coarseCyan] ($(core-key)+(1.3mm,0)$)
        rectangle ++(1.3mm,2.6mm);
    \draw[matrixGrid,line width=0.3pt]
        (core-key) rectangle ++(2.6mm,2.6mm);
    \node[font=\scriptsize,anchor=west,align=left,inner sep=0pt]
        at ($(core-key)+(4mm,1.3mm)$)
        {$\msf{SS}$ blocks\\pass to next level};

    \coordinate (near-key) at ($(change)+(-17mm,-13mm)$);
    \filldraw[fill=levelThreeRed,draw=matrixGrid,line width=0.3pt]
        (near-key) rectangle ++(2.6mm,2.6mm);
    \node[font=\scriptsize,anchor=west,align=left,inner sep=0pt]
        at ($(near-key)+(4mm,1.3mm)$)
        {Near field: retain\\$\msf{SR},\msf{RS},\msf{RR}$};

    \coordinate (far-key) at ($(change)+(-17mm,-20.5mm)$);
    \filldraw[fill=white,draw=matrixGrid,line width=0.3pt]
        (far-key) rectangle ++(2.6mm,2.6mm);
    \node[font=\scriptsize,anchor=west,align=left,inner sep=0pt]
        at ($(far-key)+(4mm,1.3mm)$)
        {Far field: discard\\$\msf{SR},\msf{RS},\msf{RR}$};

\node[below=6mm of physical] (physical-geometry)
        {\includegraphics[width=34mm]
        {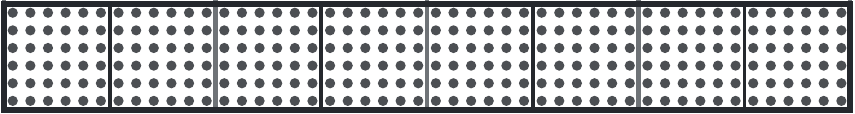}};
    \node[font=\scriptsize,anchor=south]
        at ($(physical-geometry.north west)!0.125!(physical-geometry.north east)$)
        {$\msf F$};
    \node[font=\scriptsize,anchor=south]
        at ($(physical-geometry.north west)!0.3125!(physical-geometry.north east)$)
        {$\msf N$};
    \node[font=\scriptsize,anchor=south]
        at ($(physical-geometry.north west)!0.4375!(physical-geometry.north east)$)
        {$\msf B$};
    \node[font=\scriptsize,anchor=south]
        at ($(physical-geometry.north west)!0.5625!(physical-geometry.north east)$)
        {$\msf N$};
    \node[font=\scriptsize,anchor=south]
        at ($(physical-geometry.north west)!0.8125!(physical-geometry.north east)$)
        {$\msf F$};
    \node[below=6mm of transformed] (skeleton-geometry)
        {\includegraphics[width=34mm]
        {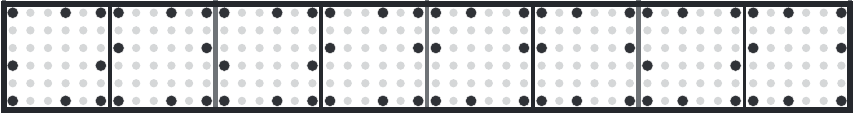}};
    \node[font=\scriptsize,anchor=north]
        at (physical-geometry.south) {Physical points};
    \node[font=\scriptsize,anchor=north]
        at (skeleton-geometry.south) {$\msf S$ coordinates (schematic)};
\end{tikzpicture}
 }
    \caption{The one-level coordinate split, regrouped into $\msf S/\msf R$
    coordinates.  Colors within the $\msf{SS}$ core indicate the next
    coarser partition.  The lower-right panel schematically represents
    coordinates in the retained orthonormal bases.}
    \label{fig:one-level-recursion}
\end{figure}
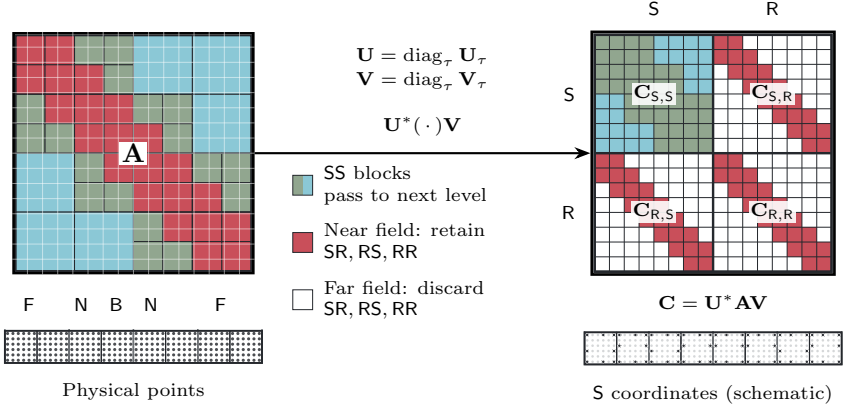

\subsection{One-Level Recovery}
\label{sec:one-level-randomized-recovery}

Recovery has two tasks: first find the box bases, then recover
the neighboring $\msf R$-containing coefficients, using only
matrix--vector products. Sketches formed with dense Gaussian test matrices mix nearby
and far-field interactions.

We describe linear transformations on the test matrix that separate
these interactions. Block nullification cancels the nearby contributions
to identify the far-field bases. Block extraction then recovers
the required neighboring coefficients in the resulting coordinates.
We first describe these operations separately, then show how both
can use the same forward--adjoint sketch pair.

\subsubsection{Block Nullification and Basis Recovery}
\label{sec:one-level-recover-bases}

Fix a box $\tau$ and use the decomposition
$\msf B,\msf N,\msf F$ above.  Randomized range finding uses a sketch of
$\mtx A_{\msf B,\msf F}$ to recover the target basis. 
However, a dense sketch also has contributions 
from the box itself and its near neighbors. 
Block nullification transforms the test matrix so that its
$\msf B\cup\msf N$ rows vanish.

Fix $p\geq2$.  Draw independent Gaussian test matrices
$\mtx\Omega,\mtx\Psi\in\field^{N\times s}$ and form
$\mtx Y=\mtx A\mtx\Omega$ and $\mtx Z=\mtx A^*\mtx\Psi$.  For the current box,
\[
    \mtx Y_{\msf B}
    =\mtx A_{\msf B,\msf B}\mtx\Omega_{\msf B}
    +\mtx A_{\msf B,\msf N}\mtx\Omega_{\msf N}
    +\mtx A_{\msf B,\msf F}\mtx\Omega_{\msf F}.
\]
To remove the $\msf B\cup\msf N$ contribution, define
\begin{equation}
\label{eq:one-level-nullification-test}
    \mtx G_\Omega
    :=\begin{bmatrix}\mtx\Omega_{\msf B}\\ \mtx\Omega_{\msf N}\end{bmatrix}
    \in\field^{n_\tau\times s},
    \qquad n_\tau:={\textstyle \sum}_{\sigma\in\near(\tau)}\ m_\sigma.
\end{equation}
If $s\geq n_\tau+k_\tau+p$, then the nullspace of $\mtx G_\Omega$ has
dimension at least $k_\tau+p$.  Choose \(\mtx K_\Omega\in\mathbb F^{s\times(k_\tau+p)}\) with orthonormal columns in \(\operatorname{null}(\mtx G_\Omega)\), using only $\mtx G_{\Omega}$.
Multiplying $\mtx Y_{\msf B}$ by
$\mtx K_\Omega$ gives
\begin{equation}
\label{eq:one-level-nullified-sketch}
    \mtx Y_{\msf B}\mtx K_\Omega
    =
    \mtx A_{\msf B,\msf F}\mtx \Omega_{\msf F}
    \mtx K_\Omega.
\end{equation}
The matrix $\mtx\Omega_{\msf F}$ is independent of $\mtx G_\Omega$, and
$\mtx K_\Omega$ has orthonormal columns.  Therefore
$\mtx\Omega_{\msf F}\mtx K_\Omega$ is standard Gaussian with
$k_\tau+p$ columns. Retain the top $k_\tau$ left singular vectors
of $\mtx Y_{\msf B}\mtx K_\Omega$ as $\mtx U_{\msf S_\tau}$. Applying the same construction to
$\mtx Z=\mtx A^*\mtx\Psi$, with the nearby rows of $\mtx\Psi$ nullified,
recovers the source basis $\mtx V_{\msf S_{\tau}}$.
The resulting basis
error is bounded by \cref{lem:gaussian-range-estimate}.
Applying this construction to every box uses only the two
dense sketches $\mtx Y$ and $\mtx Z$.  Thus
$s\geq\max_{\tau}(n_\tau+k_\tau+p)$
suffices for all boxes.

\begin{figure}[t]
    \centering
    \begin{tikzpicture}[
    every node/.style={align=center},
    process/.style={-{Stealth[length=2.2mm]},line width=0.8pt},
    page/.style={
        draw=black,
        fill=white,
        line width=0.9pt,
        inner sep=0.3mm,
    },
]
    \node[page] (ystackback)
        {\includegraphics[width=4.8mm,height=28mm]
        {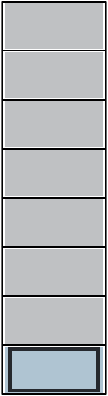}};
    \node[page,xshift=-2.1mm,yshift=-1.4mm]
        at (ystackback.center) (ystackmid)
        {\includegraphics[width=4.8mm,height=28mm]
        {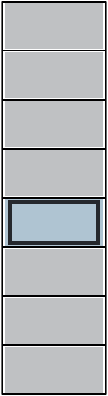}};
    \node[page,xshift=-4.2mm,yshift=-2.8mm]
        at (ystackback.center) (ystackfront)
        {\includegraphics[width=4.8mm,height=28mm]
        {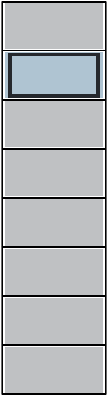}};

    \node[page,right=10.5mm of ystackfront] (A)
        {\includegraphics[width=28mm]
        {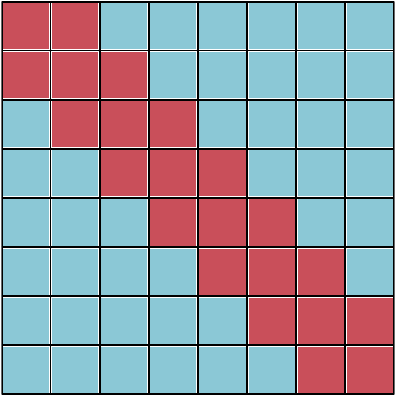}};

    \node[page,right=5mm of A,xshift=3.6mm,yshift=2.8mm] (ostackback)
        {\includegraphics[width=4.8mm,height=28mm]
        {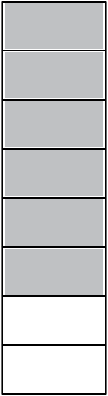}};
    \node[page,xshift=-2.1mm,yshift=-1.4mm]
        at (ostackback.center) (ostackmid)
        {\includegraphics[width=4.8mm,height=28mm]
        {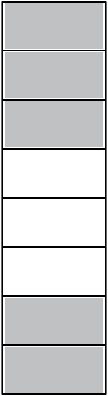}};
    \node[page,xshift=-4.2mm,yshift=-2.8mm]
        at (ostackback.center) (ostackfront)
        {\includegraphics[width=4.8mm,height=28mm]
        {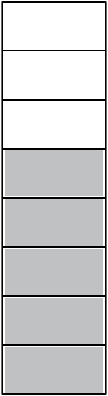}};

    \node at ($(ystackback.east)!0.5!(A.west)$) {$=$};
    \node[below=1mm of ystackfront,font=\footnotesize]
        {$\text{shaded: }\mtx Y_j\mtx K_j$\\[-0.5mm]
        \(\scriptstyle j=2,5,8\)};
    \node[below=1mm of ostackfront,font=\footnotesize]
        {$\mtx\Omega\mtx K_j$\\[-0.5mm]
        \(\scriptstyle j=2,5,8\)};

    \node[page,left=28mm of ystackfront] (sharedA)
        {\includegraphics[width=11.3mm]
        {figures/recovery_one_level_strip_physical_operator.pdf}};
    \node[page,left=5.5mm of sharedA] (sharedY)
        {\includegraphics[width=3.5mm]
        {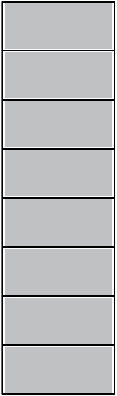}};
    \node[page,right=1.8mm of sharedA] (sharedOmega)
        {\includegraphics[width=3.5mm]
        {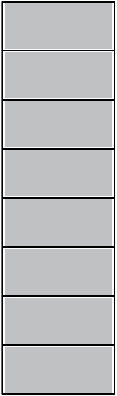}};
    \node at ($(sharedY.east)!0.5!(sharedA.west)$) {$=$};
    \node[below=1mm of sharedY,font=\scriptsize] {$\mtx Y$};
    \node[below=1mm of sharedA,font=\scriptsize] {$\mtx A$};
    \node[below=1mm of sharedOmega,font=\scriptsize] {$\mtx\Omega$};
    \draw[process] ($(sharedOmega.east)+(2mm,0)$) --
        node[above=1mm,font=\scriptsize,fill=white,
        inner xsep=2pt,inner ysep=1pt]
        {block\\nullification}
        ($(ystackfront.west)+(-2mm,0)$);

    \node[below=4mm of A,inner sep=0pt] (geom)
        {\includegraphics[width=28mm]
        {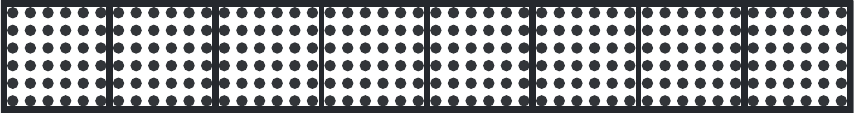}};
    \foreach \i in {1,...,8} {
        \pgfmathsetmacro{\taufraction}{(\i-0.5)/8}
        \node[font=\scriptsize\bfseries,anchor=south]
            at ($(geom.north west)!\taufraction!(geom.north east)
                +(0,0.35mm)$)
            {$\i$};
    }
\end{tikzpicture}
     \caption{Block nullification.  The same dense sketch
    $\mtx Y=\mtx A\mtx\Omega$ is used for each box $\tau_j$.
    The matrix $\mtx K_j$ annihilates the rows of $\mtx\Omega$ associated
    with $\tau_j$ and its near neighbors, so the shaded block
    $\mtx Y_j\mtx K_j$ samples only its far field.}
    \label{fig:one-level-block-nullification}
\end{figure}

\FloatBarrier
\subsubsection{Block Extraction}
\label{sec:one-level-recover-residual}

Once the target and source bases have been recovered, the remaining task is
to recover the $\msf R$-containing coefficients between neighboring boxes.
Let $\mtx U=\diag_\tau\mtx U_\tau$,
$\mtx V=\diag_\tau\mtx V_\tau$, and
$\mtx C=\mtx U^*\mtx A\mtx V$.
Fix a box $\tau$ and abbreviate
$\msf S_\tau,\msf R_\tau$ by $\msf S,\msf R$.

The recovered target basis makes $\mtx C_{\msf R,\msf F}$
small. Block extraction isolates the neighboring coefficients
$\mtx C_{\msf R,\msf B}$ and $\mtx C_{\msf R,\msf N}$
from the sketches, up to an error from the remaining far-field
contribution.
Assume that Gaussian test matrices and sketches
of $\mtx C$ are available in the transformed coordinates:
\[
    \widehat{\mtx Y}:=\mtx C\widehat{\mtx\Omega},
    \qquad
    \widehat{\mtx Z}:=\mtx C^*\widehat{\mtx\Psi}.
\]
The relevant $\msf R$ block row is
\[
    \widehat{\mtx Y}_{\msf R}
    =\mtx C_{\msf R,\msf B}\widehat{\mtx\Omega}_{\msf B}
    +\mtx C_{\msf R,\msf N}\widehat{\mtx\Omega}_{\msf N}
    +\mtx C_{\msf R,\msf F}\widehat{\mtx\Omega}_{\msf F}.
\]
Define
$\widehat{\mtx G}_\Omega
:=[\widehat{\mtx\Omega}_{\msf B}^*\ \widehat{\mtx\Omega}_{\msf N}^*]^*$.
With the stated number of columns,
$\widehat{\mtx G}_\Omega$ has full row rank with probability one.
Thus the $\msf B\cup\msf N$ rows of the effective test matrix
$\widehat{\mtx\Omega}\pinv{\widehat{\mtx G}_\Omega}$ form the identity,
while its $\msf F$ rows are nonzero.  Multiplying the equation for
$\widehat{\mtx Y}_{\msf R}$ on the right by
$\pinv{\widehat{\mtx G}_\Omega}$ gives
\begin{equation}
\label{eq:one-level-block-extraction}
    [\widetilde{\mtx C}_{\msf R,\msf B}\quad
     \widetilde{\mtx C}_{\msf R,\msf N}]
    :={}
    \widehat{\mtx Y}_{\msf R}
    \pinv{\widehat{\mtx G}_\Omega}
    =
    [\mtx C_{\msf R,\msf B}\quad
     \mtx C_{\msf R,\msf N}]
    +
    \mtx C_{\msf R,\msf F}
        \widehat{\mtx\Omega}_{\msf F}
    \pinv{\widehat{\mtx G}_\Omega}.
\end{equation}
The second term is the extraction error: contamination of the recovered
near-field coefficients by the small far-field residual.
The analogous adjoint calculation with $\widehat{\mtx Z}$ and
$\widehat{\mtx\Psi}$ defines $\widetilde{\mtx C}_{\msf B,\msf R}$ and
$\widetilde{\mtx C}_{\msf N,\msf R}$.
Both the forward and adjoint calculations estimate the
$\msf{RR}$ blocks. We use the forward estimates when
assembling the approximation.

\begin{figure}[t]
    \centering
    \resizebox{0.98\linewidth}{!}{\def\blockextractC{\begin{tikzpicture}[x=1.48mm,y=1.48mm]
        \fill[white] (0,0) rectangle (16,16);
        \foreach \i in {0,...,7} {
            \foreach \j in {0,...,7} {
                \pgfmathtruncatemacro{\near}{abs(\i-\j)<=1}
                \pgfmathtruncatemacro{\parentfar}{abs(floor(\i/2)-floor(\j/2))>1}
                \ifnum\parentfar=1
                    \def\corecolor{extractCyan}
                \else
                    \def\corecolor{extractGreen}
                \fi
                \fill[\corecolor]
                    ({2*\j},{15-2*\i}) rectangle ({2*\j+1},{16-2*\i});
                \ifnum\near=1
                    \fill[extractRed]
                        ({2*\j+1},{15-2*\i}) rectangle ({2*\j+2},{16-2*\i});
                    \fill[extractRed]
                        ({2*\j},{14-2*\i}) rectangle ({2*\j+1},{15-2*\i});
                    \fill[extractRed]
                        ({2*\j+1},{14-2*\i}) rectangle ({2*\j+2},{15-2*\i});
                \fi
            }
        }
        \foreach \g in {0,...,16} {
            \draw[extractGrid!48,line width=0.10pt] (\g,0) -- (\g,16);
            \draw[extractGrid!48,line width=0.10pt] (0,\g) -- (16,\g);
        }
        \foreach \g in {2,4,...,14} {
            \draw[extractGrid,line width=0.55pt] (\g,0) -- (\g,16);
            \draw[extractGrid,line width=0.55pt] (0,\g) -- (16,\g);
        }
        \draw[extractGrid,line width=0.85pt] (0,0) rectangle (16,16);
    \end{tikzpicture}}

\begin{tikzpicture}[
    every node/.style={align=center},
    process/.style={-{Stealth[length=2.2mm]},line width=0.8pt},
    page/.style={draw=black,fill=white,line width=0.9pt,inner sep=0.25mm},
]
    \definecolor{extractRed}{HTML}{C94F5A}
    \definecolor{extractGreen}{HTML}{91A68F}
    \definecolor{extractCyan}{HTML}{8BC8D6}
    \definecolor{extractGray}{HTML}{D9D9D9}
    \definecolor{extractGrid}{HTML}{25282D}

\node[page] (Yhat)
        {\includegraphics[width=4.1mm,height=24mm]
        {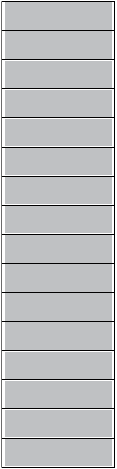}};
    \node[page,right=8mm of Yhat] (Cleft) {\blockextractC};
    \node[page,right=4mm of Cleft] (Omegahat)
        {\includegraphics[width=4.1mm,height=24mm]
        {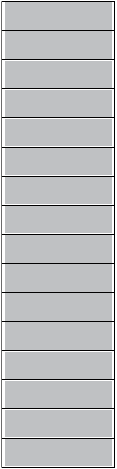}};
    \node at ($(Yhat.east)!0.5!(Cleft.west)$) {$=$};
    \node[below=0.8mm of Yhat,font=\scriptsize] {$\widehat{\mtx Y}$};
    \node[font=\small,fill=white,fill opacity=0.82,text opacity=1,
        inner sep=0.55mm] at (Cleft.center) {$\mtx C$};
    \node[below=0.8mm of Omegahat,font=\scriptsize] {$\widehat{\mtx\Omega}$};

    \draw[process] ($(Omegahat.east)+(1.5mm,0)$) --
        node[above=1.2mm,xshift=1mm,font=\scriptsize,fill=white,inner sep=1pt]
        {block extraction}
        ++(16mm,0) coordinate (extractend);

\node[page,right=4mm of extractend] (Yprime)
        {\begin{tikzpicture}[x=1.45mm,y=1.48mm]
            \fill[extractGray] (0,0) rectangle (6,16);
            \fill[extractRed] (0,6) rectangle (6,7);
            \foreach \g in {0,...,6}
                \draw[extractGrid,line width=0.14pt] (\g,0) -- (\g,16);
            \foreach \g in {0,...,16}
                \draw[extractGrid!60,line width=0.10pt] (0,\g) -- (6,\g);
            \foreach \g in {2,4,...,14}
                \draw[extractGrid,line width=0.55pt] (0,\g) -- (6,\g);
            \draw[extractGrid,line width=0.85pt] (0,0) rectangle (6,16);
        \end{tikzpicture}};
    \node[below=0.8mm of Yprime,font=\scriptsize]
        {$\widehat{\mtx Y}\pinv{\widehat{\mtx G}_{\Omega}}$};

    \node[right=6mm of Yprime,font=\large] (equals) {$=$};
    \node[page,right=6mm of equals] (Cright) {\blockextractC};
    \node[font=\small,fill=white,fill opacity=0.82,text opacity=1,
        inner sep=0.55mm] at (Cright.center) {$\mtx C$};

\node[page,right=5mm of Cright] (effective)
        {\begin{tikzpicture}[x=1.45mm,y=1.48mm]
            \fill[extractGray] (0,0) rectangle (6,4);
            \fill[white] (0,4) rectangle (6,10);
            \fill[extractGray] (0,10) rectangle (6,16);
            \foreach \xpos/\y in {1/9,3/7,5/5}
                \node[font=\small,fill=white,inner sep=0.15pt,
                    xshift=-0.85mm,yshift=0mm]
                    at (\xpos,\y) {$\mathbf I$};
            \foreach \y in {1,3,11,13,15}
                \node[font=\small,inner sep=0pt,
                    xshift=-0.90mm,yshift=0.15mm]
                    at (3,\y) {$\ast$};
            \foreach \g in {0,...,16}
                \draw[extractGrid!60,line width=0.10pt] (0,\g) -- (6,\g);
            \foreach \g in {2,4,...,14}
                \draw[extractGrid,line width=0.55pt] (0,\g) -- (6,\g);
            \draw[extractGrid,line width=0.9pt] (0,0) rectangle (6,16);
        \end{tikzpicture}};
    \node[below=0.8mm of effective,font=\scriptsize]
        {$\widehat{\mtx\Omega}\pinv{\widehat{\mtx G}_{\Omega}}$};

    \node[below=1.5mm of Cright,inner sep=0pt] (geom)
        {\includegraphics[width=25mm]
        {figures/recovery_one_level_strip_geometry.pdf}};
\end{tikzpicture}
 }
    \caption{Block extraction for one box, with coordinates displayed boxwise.
    From left to right, the three identity blocks of the effective test select
    the left neighbor, the box itself, and the right neighbor, while the rows
    marked $\ast$ produce the additive far-field term.  The selected
    $\msf R$ row of the product gives
    \eqref{eq:one-level-block-extraction}. }
    \label{fig:one-level-block-extraction}
\end{figure}

\subsubsection{Sketch Reuse and the One-Level Guarantee}
\label{sec:one-level-recovery-error}

Block extraction requires no additional operator applications.
After recovering the box bases by nullification, we obtain the
required sketches by rotating the original test matrices and
sketches into the recovered coordinates:
\begin{equation}
\label{eq:one-level-sketch-reuse}
\begin{aligned}
    \widehat{\mtx Y}&:=\mtx U^*\mtx Y
        =\mtx C\widehat{\mtx\Omega},
    &\widehat{\mtx\Omega}&:=\mtx V^*\mtx\Omega,\\
    \widehat{\mtx Z}&:=\mtx V^*\mtx Z
        =\mtx C^*\widehat{\mtx\Psi},
    &\widehat{\mtx\Psi}&:=\mtx U^*\mtx\Psi.
\end{aligned}
\end{equation}
The rotated test matrices remain Gaussian.  The source basis $\mtx V$ is
recovered from the adjoint sketch $\mtx Z=\mtx A^*\mtx\Psi$ and is therefore
independent of $\mtx\Omega$; similarly, $\mtx U$ is independent of
$\mtx\Psi$.

For a fixed box, the recovered target basis and the subsequent
row extraction both depend on $\mtx\Omega$. To control the
extraction error, we use the fact that nullification and extraction
involve orthogonal components of the test matrix.
Nullification uses
$k_\tau+p$ directions in $\operatorname{null}(\mtx G_\Omega)$ to sample the
far field, whereas extraction uses the orthogonal component in
$\range(\mtx G_\Omega^*)$.  Conditional on \(\mtx G_\Omega\), the corresponding projections of the independent far-field Gaussian matrix onto these orthogonal subspaces are independent.  The analogous statement holds for the adjoint sketch.

We compare the recovered approximation with the best one-level
approximation having the same box partition, rank bounds, and
neighbor relation. Let $\mathcal M$ denote the
class of matrices having the one-level structure described in
\Cref{sec:one-level-basis-split}.  
For each box $\tau$, allow target and source spaces of dimension at most $k_\tau$. For completed box bases $\overline{\mtx U},\overline{\mtx V}$,
assembled with columns ordered as $(\msf S,\msf R)$, write
\[
\overline{\mtx A}
=
\overline{\mtx U}
\begin{bmatrix}
    \overline{\mtx C}_{\msf S,\msf S}
    & \overline{\mtx C}_{\msf S,\msf R}\\
    \overline{\mtx C}_{\msf R,\msf S}
    & \overline{\mtx C}_{\msf R,\msf R}
\end{bmatrix}
\overline{\mtx V}^{*},
\]
where the coefficients containing an $\msf R$ index satisfy
\[
\begin{bmatrix}
    \mtx0 & \overline{\mtx C}_{\msf S,\msf R}\\
    \overline{\mtx C}_{\msf R,\msf S}
    & \overline{\mtx C}_{\msf R,\msf R}
\end{bmatrix}
=
\left[
\begin{array}{cc}
    \mtx0 & \overline{\mtx C}_{\msf S,\msf R}\\
    \overline{\mtx C}_{\msf R,\msf S}
    & \overline{\mtx C}_{\msf R,\msf R}
\end{array}
\right]_{\mathcal N}.
\]
Here $\mathcal N$ denotes the neighboring box pairs, and
$[\cdot]_{\mathcal N}$ retains their coefficient blocks
and sets all other blocks to zero.
There are no restrictions on the block $\overline{\mtx C}_{\msf S,\msf S}$.
Define the best approximation error in this class as
\begin{equation}
\label{eq:one-level-opt}
    \operatorname{opt}(\mtx A)
    :=
    \inf_{\overline{\mtx A}\in\mathcal M}
    \normfro{\mtx A-\overline{\mtx A}}.
\end{equation}

\begin{lemma}[One-level recovery estimate]
\label{lem:one-level-recovery}
Suppose
\(
    s\geq\max_\tau (n_\tau+k_\tau+p).
\)
Form the Gaussian sketches
\[
    \mtx Y=\mtx A\mtx\Omega,
    \qquad
    \mtx Z=\mtx A^*\mtx\Psi
\]
using $s$ applications of each operator. Using the same sketch pair, recover the box bases and neighboring
$\msf R$-containing coefficients. Let
$\widetilde{\mtx A}\in\mathcal M$ denote the approximation obtained
by completing these recovered coefficients with the exact
$\msf{SS}$ block in the recovered coordinates. Then
\begin{equation}
\label{eq:one-level-quasioptimality}
    \E\normfro{\mtx A-\widetilde{\mtx A}}
    \leq
    \sqrt{2}
    \left(\max_\tau\beta_\tau\right)
    \operatorname{opt}(\mtx A),
\end{equation}
where $\beta_\tau:= \alpha_{k_\tau,p} \sqrt{1+{n_\tau}/{(k_\tau+p-1)}}$ and $\alpha_{k,p}$ is the Gaussian range-recovery factor from
\cref{eq:appendix-range-alpha}.
\end{lemma}

After randomized recovery of the box bases and neighboring
$\msf R$-containing coefficients, the approximation in the lemma
is completed with the exact $\msf{SS}$ block in the recovered
coordinates. This isolates the error introduced at the current
level, and the surviving block forms the smaller problem passed
to the next coarser level.
The expected recovery error is within an explicit factor
of the best error in the prescribed one-level class. 
The lemma adapts the fresh-sketch estimate of
Amsel et al.~\cite[Appendix~C]{amsel2026quasioptimal} to the reused-sketch
setting.  The proof is given in Appendix~\ref{sec:app-one-level-recovery}.
 \section{Multilevel Approximation}
\label{sec:multilevel-prd}

\Cref{sec:one-level-prd} describes the approximation at a single scale.
We now extend the same coordinate split across an adaptive hierarchy.
The tree determines the box partition and neighbor relation at each level;
the repeated splits produce nested coordinate spaces and a telescoping
decomposition, from which the stored $\mathcal H^2$ approximation is assembled.

\subsection{Adaptive Tree and Level Partitions}
\label{subsec:adaptive-tree}

Starting from a root box on level $0$, boxes are partitioned into children until every
leaf contains at most $m_{\mathrm{leaf}}$ points.
Let $\mathcal L$ denote the set of leaves and let $L$ be the maximum leaf depth.
Because the refinement is adaptive, leaves may occur on different levels, so the
level-$\ell$ boxes alone do not necessarily partition the physical points.  
At level $\ell$, we define
\[
    \mathcal T_\ell:=\{\tau:\operatorname{lev}(\tau)=\ell\},\qquad
    \mathcal B_\ell:=\mathcal T_\ell\cup
    \{\lambda\in\mathcal L:\operatorname{lev}(\lambda)<\ell\}.
\]
Every point belongs either to a level-$\ell$ box or to a leaf whose branch
terminated at a coarser level, so $\mathcal B_\ell$ is a partition. Moreover, $\mathcal B_\ell$ refines $\mathcal B_{\ell-1}$: each box
is either an unchanged leaf or a child of a nonleaf in
$\mathcal T_{\ell-1}$.
Since $L$ is the maximum leaf depth, $\mathcal B_L=\mathcal L$.

For strong admissibility, let $\near_\ell(\tau)$ denote the boxes in
$\mathcal B_\ell$ that are adjacent to $\tau$, including
$\tau$ itself.  We assume adjacent leaves differ by at most one level, as in a
standard $2{:}1$ balanced adaptive tree \cite[Sec.~2.2]{malhotra2015pvfmm}.  Thus a box may have both
same-level neighbors and neighboring coarser leaves.  
Figure~\ref{fig:adaptive-tree} illustrates this partition.

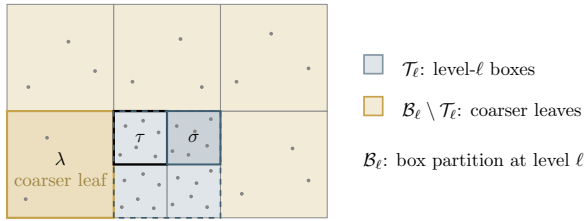
\begin{figure}[H]
    \centering
    \resizebox{0.60\linewidth}{!}{\begin{tikzpicture}[
    x=10mm,
    y=10mm,
    every node/.style={font=\normalsize},
    box/.style={draw=black!45,line width=0.55pt,fill=white},
]
    \definecolor{basisBlue}{HTML}{405F73}
    \definecolor{nearGold}{HTML}{C6A34A}

\path[box,fill=nearGold!22] (0,0) rectangle (2,2);
    \path[box,fill=nearGold!22] (0,2) rectangle (2,4);
    \path[box,fill=nearGold!22] (2,2) rectangle (4,4);
    \path[box,fill=nearGold!22] (4,0) rectangle (6,2);
    \path[box,fill=nearGold!22] (4,2) rectangle (6,4);

    \foreach \x in {2,3} {
        \foreach \y in {0,1} {
            \path[box,fill=basisBlue!16] (\x,\y) rectangle +(1,1);
        }
    }

\path[draw=nearGold,fill=nearGold!35,line width=1.15pt]
        (0,0) rectangle (2,2);
    \path[draw=black,fill=basisBlue!16,line width=1.25pt]
        (2,1) rectangle (3,2);
    \path[draw=basisBlue,fill=basisBlue!28,line width=1.05pt]
        (3,1) rectangle (4,2);

\draw[basisBlue,dashed,line width=0.9pt]
        (2,0) rectangle (4,2);

\foreach \x/\y in {
        0.35/0.35,0.75/1.50,1.55/0.70,
        0.40/2.45,1.15/3.30,1.65/2.75,
        2.35/2.55,3.25/3.35,3.65/2.45,
        4.40/0.45,5.35/1.30,5.65/0.70,
        4.30/2.55,4.95/3.45,5.65/2.80,
        2.15/0.20,2.45/0.35,2.75/0.18,3.10/0.25,3.45/0.40,3.75/0.20,
        2.20/0.70,2.55/0.82,2.85/0.62,3.20/0.75,3.55/0.65,3.82/0.88,
        2.12/1.18,2.42/1.32,2.78/1.15,3.10/1.25,3.40/1.18,3.75/1.35,
        2.22/1.72,2.55/1.82,2.84/1.60,3.18/1.72,3.52/1.62,3.82/1.85}
        \fill[black!48] (\x,\y) circle (0.035);

    \node[font=\normalsize,fill=nearGold!35,inner sep=0.5pt] at (1,1.12) {$\lambda$};
    \node[text=nearGold!80!black,fill=nearGold!35,inner sep=0.5pt]
        at (1,0.72) {coarser leaf};
    \node[font=\normalsize,fill=basisBlue!16,inner sep=0.5pt] at (2.5,1.5) {$\tau$};
    \node[font=\normalsize,text=black,fill=basisBlue!28,inner sep=0.5pt]
        at (3.5,1.5) {$\sigma$};
    \node[anchor=west] at (6.55,2.85)
        {\tikz\draw[basisBlue!65,fill=basisBlue!16,line width=0.9pt]
            (0,0) rectangle (0.35,0.35);\quad $\mathcal T_\ell$: level-$\ell$ boxes};
    \node[anchor=west] at (6.55,2.05)
        {\tikz\draw[nearGold,fill=nearGold!22,line width=0.9pt]
            (0,0) rectangle (0.35,0.35);\quad $\mathcal B_\ell\setminus\mathcal T_\ell$: coarser leaves};
    \node[anchor=west,align=left] at (6.55,1.05)
        {$\mathcal B_\ell$: box partition at level $\ell$};
\end{tikzpicture}
 }
\caption{An adaptive partition $\mathcal B_{\ell}$ at level $\ell$.
Boxes in $\mathcal T_\ell$ are processed at this level, while coarser leaves
remain unchanged.  Highlighted are $\tau$, its same-level neighbor $\sigma$,
and a neighboring coarser leaf $\lambda$.}
\label{fig:adaptive-tree}
\end{figure}

\subsection{Nested Coordinates and Telescoping Decomposition}
\label{subsec:nested-coordinates}

The one-level split leaves a smaller matrix on the surviving
coordinates. To represent this remaining matrix at coarser scales,
we use nested bases: each parent basis is expressed in the surviving
bases of its children. The local factors $\mtx U_{\msf S_\tau}$
and $\mtx V_{\msf S_\tau}$ are bases in physical coordinates
at leaves and transfer matrices at nonleaves.

The maps $\mathcal U_\tau$ and $\mathcal V_\tau$ express the
coordinates entering the local split in physical space.
At a leaf, these maps are identities; at a nonleaf, they collect
the surviving bases of the children:
\[
    \mathcal U_{\tau}
    :=
    \diag_{\nu\in\operatorname{ch}(\tau)}
    \mathcal U_{\msf S_\nu},
    \qquad
    \mathcal V_{\tau}
    :=
    \diag_{\nu\in\operatorname{ch}(\tau)}
    \mathcal V_{\msf S_\nu}.
\]
The corresponding surviving bases in physical coordinates are
\[
    \mathcal U_{\msf S_\tau}
    :=
    \mathcal U_\tau\mtx U_{\msf S_\tau},
    \qquad
    \mathcal V_{\msf S_\tau}
    :=
    \mathcal V_\tau\mtx V_{\msf S_\tau}.
\]

Let $\msf S_\ell$ index the retained coordinates on
$\mathcal B_\ell$. The corresponding maps into physical space,
$\mathcal U_{\msf S_\ell}$ and $\mathcal V_{\msf S_\ell}$,
are block diagonal, with surviving box bases on
$\mathcal T_\ell$ and identity blocks on coarser leaves.
The corresponding coefficient matrix is
\begin{equation}
\label{eq:level-core}
    \mtx C_{\msf S_\ell,\msf S_\ell}
    :=
    \mathcal U_{\msf S_\ell}^*
    \mtx A
    \mathcal V_{\msf S_\ell}.
\end{equation}
For the initial physical coordinates, set
$\mathcal U_{\msf S_{L+1}}=\mathcal V_{\msf S_{L+1}}=\mtx I$.
Successive coefficient matrices act on different coordinate spaces.
To compare them, we express them in physical coordinates:
\begin{equation}
\label{eq:levelwise-approximation}
    \mtx W_\ell
    :=
    \mathcal U_{\msf S_\ell}
    \mtx C_{\msf S_\ell,\msf S_\ell}
    \mathcal V_{\msf S_\ell}^*,
    \qquad
    \mtx W_{L+1}:=\mtx A.
\end{equation}
The contribution associated with level $\ell$ is the successive difference
\begin{equation}
\label{eq:levelwise-residual}
    \mtx D_\ell
    :=
    \mtx W_{\ell+1}-\mtx W_\ell.
\end{equation}
Let $\ell_0$ denote the coarsest transformed level.
Summing these differences gives a telescoping decomposition
\begin{equation}
\label{eq:multilevel-telescope}
    \mtx A
    =
    \mtx W_{\ell_0}
    +
    \sum_{\ell=\ell_0}^{L}\mtx D_\ell.
\end{equation}
For strong admissibility, we take $\ell_0=2$ and, for simplicity,
assume $\mathcal B_2=\mathcal T_2$.
For weak admissibility, we take $\ell_0=1$.

\subsection{Levelwise coefficients and the stored representation}
\label{subsec:levelwise-coefficients}

We express each level contribution in completed box
coordinates to identify the coefficients retained in the
approximation. 
To compare adjacent levels, complete the box basis or transfer matrix for
$\tau\in\mathcal T_\ell$ to a unitary transform,
\begin{equation}
\label{eq:completed-box-transforms}
\begin{gathered}
    \mtx U_\tau=[\mtx U_{\msf S_\tau}\ \mtx U_{\msf R_\tau}],
    \qquad
    [\mathcal U_{\msf S_\tau}\ \mathcal U_{\msf R_\tau}]
    :=\mathcal U_\tau\mtx U_\tau,\\
    \mtx V_\tau=[\mtx V_{\msf S_\tau}\ \mtx V_{\msf R_\tau}],
    \qquad
    [\mathcal V_{\msf S_\tau}\ \mathcal V_{\msf R_\tau}]
    :=\mathcal V_\tau\mtx V_\tau.
\end{gathered}
\end{equation}

For $\tau,\sigma\in\mathcal T_\ell$,
\cref{eq:level-core,eq:levelwise-approximation} give
\begin{equation}
\label{eq:boxwise-projected-component}
    (\mtx W_\ell)_{\tau\sigma}
    =\mathcal U_{\msf S_\tau}
      \mtx C_{\msf S_\tau,\msf S_\sigma}
      \mathcal V_{\msf S_\sigma}^*,
    \qquad
    \mtx C_{\msf S_\tau,\msf S_\sigma}
    =\mathcal U_{\msf S_\tau}^*\mtx A_{\tau\sigma}
      \mathcal V_{\msf S_\sigma}.
\end{equation}
In the completed coordinates, $(\mtx W_\ell)_{\tau\sigma}$
has only an $\msf{SS}$ block:
\[
    \begin{bmatrix}
        \mathcal U_{\msf S_\tau} & \mathcal U_{\msf R_\tau}
    \end{bmatrix}^{*}
    (\mtx W_\ell)_{\tau\sigma}
    \begin{bmatrix}
        \mathcal V_{\msf S_\sigma} & \mathcal V_{\msf R_\sigma}
    \end{bmatrix}
    =
    \begin{bmatrix}
        \mtx C_{\msf S_\tau,\msf S_\sigma} & \mtx0\\
        \mtx0 & \mtx0
    \end{bmatrix}.
\]
To identify $\mtx D_\ell$, we express
$(\mtx W_{\ell+1})_{\tau\sigma}$ on the finer level in the same coordinates.
By nestedness, its $\msf{SS}$ block agrees with that of
$(\mtx W_\ell)_{\tau\sigma}$:
\[
    \begin{bmatrix}
        \mathcal U_{\msf S_\tau} & \mathcal U_{\msf R_\tau}
    \end{bmatrix}^{*}
    (\mtx W_{\ell+1})_{\tau\sigma}
    \begin{bmatrix}
        \mathcal V_{\msf S_\sigma} & \mathcal V_{\msf R_\sigma}
    \end{bmatrix}
    =
    \begin{bmatrix}
        \mtx C_{\msf S_\tau,\msf S_\sigma} &
        \mtx C_{\msf S_\tau,\msf R_\sigma}\\
        \mtx C_{\msf R_\tau,\msf S_\sigma} &
        \mtx C_{\msf R_\tau,\msf R_\sigma}
    \end{bmatrix}.
\]
Subtracting cancels the common $\msf{SS}$ block.
The remaining blocks are small for well-separated boxes
because the local bases are chosen to capture the far-field
block rows and columns of the incoming coefficient matrix,
as in \Cref{sec:one-level-prd}:
\begin{equation}
\label{eq:transformed-level-difference}
    \begin{aligned}
    \begin{bmatrix}
        \mathcal U_{\msf S_\tau} & \mathcal U_{\msf R_\tau}
    \end{bmatrix}^{*}
    (\mtx D_\ell)_{\tau\sigma}
    \begin{bmatrix}
        \mathcal V_{\msf S_\sigma} & \mathcal V_{\msf R_\sigma}
    \end{bmatrix}
    &=
    \begin{bmatrix}
        \mtx 0 & \mtx C_{\msf S_\tau,\msf R_\sigma}\\
        \mtx C_{\msf R_\tau,\msf S_\sigma} &
        \mtx C_{\msf R_\tau,\msf R_\sigma}
    \end{bmatrix}\\
    &\approx \mtx 0\qquad \text{when $\tau$ and $\sigma$ are well-separated.}
    \end{aligned}
\end{equation}
For neighboring box pairs these coefficients are retained. For well-separated pairs they are discarded, using the approximation in \eqref{eq:transformed-level-difference}. 

If \(\lambda\in\near_\ell(\tau)\cap(\mathcal B_\ell\setminus\mathcal T_\ell)\), then \(\lambda\) remains in physical coordinates, and only the \(\tau\) side is split:
\begin{equation}
\label{eq:adaptive-one-sided-coefficients}
    \begin{bmatrix}
        \mathcal U_{\msf S_\tau} & \mathcal U_{\msf R_\tau}
    \end{bmatrix}^{*}
    (\mtx D_\ell)_{\tau\lambda}
    =\begin{bmatrix}\mtx0\\ \mtx C_{\msf R_\tau,\lambda}\end{bmatrix},
    \qquad
    (\mtx D_\ell)_{\lambda\tau}
    \begin{bmatrix}
        \mathcal V_{\msf S_\tau} & \mathcal V_{\msf R_\tau}
    \end{bmatrix}
    =\begin{bmatrix}\mtx0 & \mtx C_{\lambda,\msf R_\tau}\end{bmatrix}.
\end{equation}

To express the approximation over the full partition, combine
the local transforms on $\mathcal T_\ell$ with identities on
coarser leaves. These transforms split $\msf S_{\ell+1}$ into
surviving coordinates $\msf S_\ell$ and complementary
coordinates $\msf R_\ell$. Let $\mathcal U_{\msf R_\ell}$ and
$\mathcal V_{\msf R_\ell}$ collect the corresponding
complementary bases in physical coordinates.
Retaining the neighboring coefficient blocks and summing
as in \eqref{eq:multilevel-telescope} gives
\begin{equation}
\label{eq:levelwise-expansion}
    \mtx A
    \approx
    \mathcal U_{\msf S_{\ell_0}}
    \mtx C_{\msf S_{\ell_0},\msf S_{\ell_0}}
    \mathcal V_{\msf S_{\ell_0}}^*
    +\sum_{\ell=\ell_0}^{L}
    \begin{bmatrix}
        \mathcal U_{\msf S_\ell} & \mathcal U_{\msf R_\ell}
    \end{bmatrix}
    \begin{bmatrix}
        \mtx0 & \mtx C_{\msf S_\ell,\msf R_\ell}\\
        \mtx C_{\msf R_\ell,\msf S_\ell} &
        \mtx C_{\msf R_\ell,\msf R_\ell}
    \end{bmatrix}_{\near_\ell}
    \begin{bmatrix}
        \mathcal V_{\msf S_\ell} & \mathcal V_{\msf R_\ell}
    \end{bmatrix}^{*}.
\end{equation}
Here the restriction to $\near_\ell$ retains only neighboring
box-pair blocks in each coefficient submatrix and sets all other
blocks to zero. See \Cref{fig:levelwise-expansions}.

\begin{figure}[!htb]
    \centering
    \resizebox{0.97\linewidth}{!}{
        \begin{tikzpicture}[
    every node/.style={align=center},
    process/.style={-{Stealth[length=2.0mm]},line width=0.75pt},
    page/.style={
        draw=black,
        fill=white,
        line width=0.95pt,
        inner sep=0.25mm,
    },
]
    \node[minimum width=32mm,minimum height=34mm,inner sep=0pt]
        (geom-frame) {};
    \node at (geom-frame.center) (geom)
        {\includegraphics[height=30mm]
        {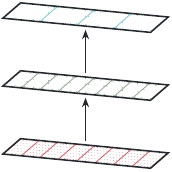}};
    \node[font=\small,anchor=south east]
        at ($(geom-frame.north west)+(-2mm,0)$) {(a)};
    \node[font=\small,anchor=east] at ($(geom.west)+(-0.8mm,11mm)$)
        {$\msf S_2$};
    \node[font=\small,anchor=east] at ($(geom.west)+(-0.8mm,0)$)
        {$(\msf S_2,\msf R_2)$};
    \node[font=\small,anchor=east] at ($(geom.west)+(-0.8mm,-11mm)$)
        {$(\msf S_3,\msf R_3)$};

    \node[page,anchor=west] (A) at ($(geom-frame.east)+(5mm,11mm)$)
        {\includegraphics[width=28mm]
        {figures/recovery_multilevel_strip_physical_partition.pdf}};
    \node[font=\small,fill=white,inner sep=0.8pt]
        at (A.center) {$\mtx A$};
    \node[right=1.2mm of A,font=\small] (approx) {$\approx$};

    \node[right=1.2mm of approx,font=\small] (U3)
        {$[\mathcal U_{\msf S_3}\ \mathcal U_{\msf R_3}]$};
    \node[page,right=0.5mm of U3] (C3)
        {\includegraphics[width=25mm]
        {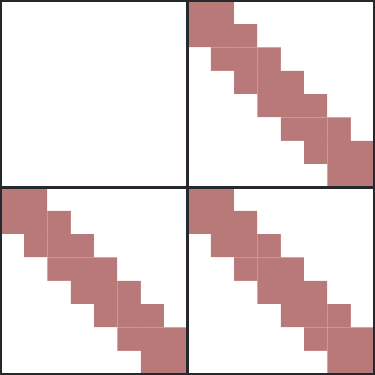}};
    \node[right=0.5mm of C3,font=\small] (V3)
        {$[\mathcal V_{\msf S_3}\ \mathcal V_{\msf R_3}]^*$};
    \node[font=\small] at
        ($(C3.north west)!0.5!(C3.center)$) {$\mtx 0$};

    \node[font=\small] (plus32)
        at ($(approx.west)+(0,-27mm)$) {$+$};
    \node[right=2mm of plus32,font=\small] (U2)
        {$[\mathcal U_{\msf S_2}\ \mathcal U_{\msf R_2}]$};
    \node[page,right=0.5mm of U2] (C2)
        {\includegraphics[width=12.5mm]
        {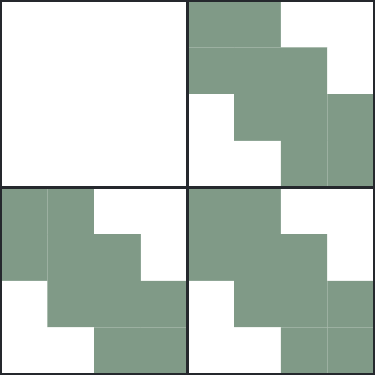}};
    \node[right=0.5mm of C2,font=\small] (V2)
        {$[\mathcal V_{\msf S_2}\ \mathcal V_{\msf R_2}]^*$};
    \node[font=\small] at
        ($(C2.north west)!0.5!(C2.center)$) {$\mtx 0$};

    \node[right=2mm of V2,font=\small] (plus2S) {$+$};
    \node[right=2mm of plus2S,font=\small] (U2S)
        {$\mathcal U_{\msf S_2}$};
    \node[page,right=0.5mm of U2S] (C2SS)
        {\includegraphics[width=6.25mm]
        {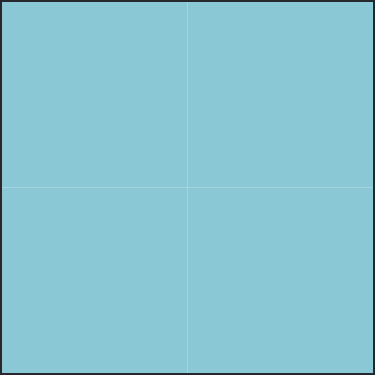}};
    \node[right=0.5mm of C2SS,font=\small] (V2S)
        {$\mathcal V_{\msf S_2}^*$};
    \node[below=0.9mm of C2SS,font=\small]
        {$\mtx C_{\msf S_2,\msf S_2}$};
\end{tikzpicture}
     }
    \vspace{5mm}
    \resizebox{0.97\linewidth}{!}{
        \begin{tikzpicture}[
    every node/.style={align=center},
    process/.style={-{Stealth[length=2.0mm]},line width=0.75pt},
    page/.style={
        draw=black,
        fill=white,
        line width=0.95pt,
        inner sep=0.25mm,
    },
]
    \node[minimum width=32mm,minimum height=34mm,inner sep=0pt]
        (geom-frame) {};
    \node at (geom-frame.center) (geom)
        {\includegraphics[height=34mm]
        {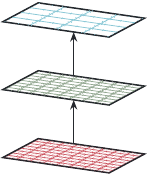}};
    \node[font=\small,anchor=south east]
        at ($(geom-frame.north west)+(-2mm,0)$) {(b)};
    \node[font=\small,anchor=east] at ($(geom.west)+(-0.8mm,12.5mm)$)
        {$\msf S_2$};
    \node[font=\small,anchor=east] at ($(geom.west)+(-0.8mm,0)$)
        {$(\msf S_2,\msf R_2)$};
    \node[font=\small,anchor=east] at ($(geom.west)+(-0.8mm,-12.5mm)$)
        {$(\msf S_3,\msf R_3)$};

    \node[page,anchor=west] (A) at ($(geom-frame.east)+(5mm,11mm)$)
        {\includegraphics[width=28mm]
        {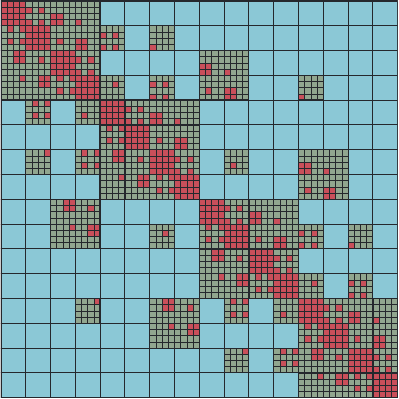}};
    \node[font=\small,fill=white,inner sep=0.8pt]
        at (A.center) {$\mtx A$};
    \node[right=1.2mm of A,font=\small] (approx) {$\approx$};
    \node[right=1.2mm of approx,font=\small] (U3)
        {$[\mathcal U_{\msf S_3}\ \mathcal U_{\msf R_3}]$};
    \node[page,right=0.5mm of U3] (C3)
        {\includegraphics[width=25mm]
        {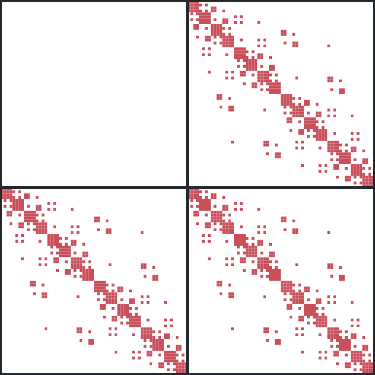}};
    \node[right=0.5mm of C3,font=\small] (V3)
        {$[\mathcal V_{\msf S_3}\ \mathcal V_{\msf R_3}]^*$};
    \node[font=\small] at
        ($(C3.north west)!0.5!(C3.center)$) {$\mtx 0$};

    \node[font=\small] (plus32)
        at ($(approx.west)+(0,-27mm)$) {$+$};
    \node[right=2mm of plus32,font=\small] (U2)
        {$[\mathcal U_{\msf S_2}\ \mathcal U_{\msf R_2}]$};
    \node[page,right=0.5mm of U2] (C2)
        {\includegraphics[width=12.5mm]
        {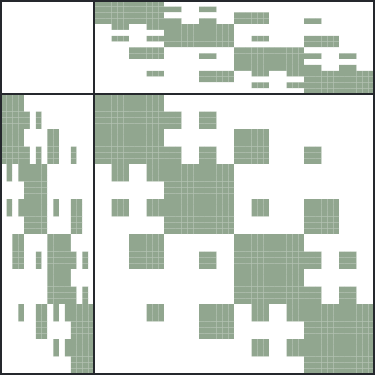}};
    \node[right=0.5mm of C2,font=\small] (V2)
        {$[\mathcal V_{\msf S_2}\ \mathcal V_{\msf R_2}]^*$};
    \node[font=\small] at
        ($(C2.north west)!0.25!(C2.center)$) {$\mtx 0$};

    \node[right=2mm of V2,font=\small] (plus2S) {$+$};
    \node[right=2mm of plus2S,font=\small] (U2S)
        {$\mathcal U_{\msf S_2}$};
    \node[page,right=0.5mm of U2S] (C2SS)
        {\includegraphics[width=6.25mm]
        {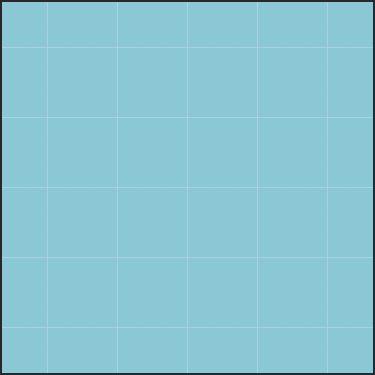}};
    \node[right=0.5mm of C2SS,font=\small] (V2S)
        {$\mathcal V_{\msf S_2}^*$};
    \node[below=0.9mm of C2SS,font=\small]
        {$\mtx C_{\msf S_2,\msf S_2}$};

\end{tikzpicture}
     }
    \caption{The multilevel approximation for \(L=3\) on a uniform binary tree and quadtree. Red and green show the neighboring \(\msf R\)-containing coefficient blocks retained at levels 3 and 2; cyan shows the terminal \(\msf S_{\ell_0}\msf S_{\ell_0}\) coefficient matrix. Each level matrix is ordered as \((\msf S_\ell,\msf R_\ell)\), with a zero \(\msf S\msf S\) block because that interaction is represented at coarser levels.}
    \label{fig:levelwise-expansions}
\end{figure}
Only the local leaf bases and transfer matrices, the retained neighboring
coefficient blocks, and the terminal coefficient matrix are stored.
The calligraphic maps in
\eqref{eq:levelwise-expansion} are applied recursively and never formed
explicitly. \suppref{app:deterministic-matvec} gives the corresponding
matrix--vector multiplication algorithm.

The following algebraic equivalence holds without assuming
nestedness or orthogonality of the box bases. The conventional form stores near interactions at the leaves and far interactions at their first well-separated level.

\begin{proposition}[Equivalence with the conventional hierarchical form]
\label{prop:interaction-list-equivalence}
For fixed box bases, the levelwise approximation
\eqref{eq:levelwise-expansion} is an exact algebraic regrouping of the
corresponding conventional adaptive hierarchical approximation.
\end{proposition}

The nested orthonormal bases also give the cross-level
orthogonality used in the multilevel error analysis.

\begin{proposition}[Orthogonality across levels]
\label{lem:orthogonality-levels}
The terminal component $\mtx W_{\ell_0}$ and the exact level contributions
$\mtx D_\ell$ are mutually orthogonal in the Frobenius inner product.
The same cross-level orthogonality holds for levelwise retained or recovered
contributions having the same incoming coordinate spaces and zero
$\msf{SS}$ block.
\end{proposition}

The equivalence proof is given in \suppref{app:interaction-list-equivalence}.
The orthogonality proof is given in Appendix~\ref{app:multilevel-proof}.
 \section{Randomized Multilevel Recovery}
\label{sec:algorithms}

To recover the representation in \Cref{sec:multilevel-prd},
we apply the one-level procedure of \Cref{sec:one-level-prd}
to the incoming coefficient matrix
$\mtx C_{\msf S_{\ell+1},\msf S_{\ell+1}}$
at levels $\ell=L,\ldots,\ell_0$.
At each level, the procedure recovers the local box transforms
and neighboring $\msf R_\ell$-containing coefficients,
leaving the surviving matrix for the next coarser level.

In this section, we present two closely related algorithms
that differ in how they obtain sketches of these
coefficient matrices.
Both reuse the same sketch pair for basis recovery and coefficient
extraction within each level.
Resketched recovery forms fresh sketches at every level using
new Gaussian tests and products with $\mtx A$ and $\mtx A^*$.
Carried recovery instead forms one forward--adjoint sketch pair
in physical coordinates and updates it through the hierarchy.
Pseudocode for the one-level routine and both multilevel variants
is given in~\suppref{app:randomized-algorithms}.

\subsection{Resketched Recovery}
\label{subsec:resketched-recovery}

At level $\ell$, resketched recovery draws fresh Gaussian test
matrices $\widehat{\mtx\Omega}_{\msf S_{\ell+1}}$ and
$\widehat{\mtx\Psi}_{\msf S_{\ell+1}}$ in the incoming
coefficient coordinates. The forward sketch is obtained by
mapping the test to physical coordinates, applying $\mtx A$,
and mapping the result back:
\begin{equation}
\label{eq:incoming-matrix-sketches}
    \widehat{\mtx Y}_{\msf S_{\ell+1}}
    :=
    \mathcal U_{\msf S_{\ell+1}}^*\mtx A
    \bigl(
        \mathcal V_{\msf S_{\ell+1}}
        \widehat{\mtx\Omega}_{\msf S_{\ell+1}}
    \bigr)
    =
    \mtx C_{\msf S_{\ell+1},\msf S_{\ell+1}}
    \widehat{\mtx\Omega}_{\msf S_{\ell+1}}.
\end{equation}
The adjoint sketch is obtained analogously using $\mtx A^*$,
with the target and source coordinate maps interchanged.
These sketches sample the exact incoming coefficient matrix.
Conditional on the bases recovered at finer levels, the new
test matrices are independent standard Gaussians in the
incoming coordinates. Consequently, the one-level estimate
applies at every level.

After level $\ell_0$ has been processed, the terminal coefficient
matrix $\mtx C_{\msf S_{\ell_0},\msf S_{\ell_0}}$ is recovered
exactly using one forward application of $\mtx A$ with
$|\msf S_{\ell_0}|$ right-hand sides.

\subsection{Carried Recovery}
\label{subsec:carried-recovery}

Carried recovery draws one pair of dense Gaussian test matrices
$\mtx\Omega,\mtx\Psi$ in physical coordinates and forms
\[
    \mtx Y=\mtx A\mtx\Omega,
    \qquad
    \mtx Z=\mtx A^*\mtx\Psi.
\]
These are the only operator applications used to sample the
hierarchy. At each level, the one-level procedure acts on the
tests and sketches in the incoming $\msf S_{\ell+1}$ coordinates.
To obtain the sketches for the next coarser level, subtract
the recovered neighboring contribution and retain the surviving
rows. For each box $\tau\in\mathcal B_\ell$, the forward update is
\begin{equation}
\label{eq:carried-forward-update}
    \widehat{\mtx Y}_{\msf S_\tau}
    \gets
    \widehat{\mtx Y}_{\msf S_\tau}
    -
    \sum_{\sigma\in\near_\ell(\tau)}
    \widetilde{\mtx C}_{\msf S_\tau,\msf R_\sigma}
    \widehat{\mtx\Omega}_{\msf R_\sigma}.
\end{equation}
The adjoint update is analogous, with target and source coordinates interchanged.

The subtraction removes only the recovered neighboring contribution.
To describe what remains, let $\widetilde{\mtx D}_j$ denote the
physical-coordinate contribution assembled from the recovered
$\msf{SR}$, $\msf{RS}$, and $\msf{RR}$ coefficients at level $j$, and define
\[
    \mtx E_j:=\mtx D_j-\widetilde{\mtx D}_j.
\]
The error $\mtx E_j$ contains both the discarded far-field contribution and the
error in the recovered neighboring coefficients.

After levels $j=\ell+1,\ldots,L$ have been processed, the forward test
matrix and sketch entering level $\ell$ satisfy
\begin{equation}
\label{eq:carried-current-sketches}
    \widehat{\mtx\Omega}_{\msf S_{\ell+1}}
    =
    \mathcal V_{\msf S_{\ell+1}}^*\mtx\Omega,
    \qquad
    \widehat{\mtx Y}_{\msf S_{\ell+1}}
    =
    \mtx C_{\msf S_{\ell+1},\msf S_{\ell+1}}
    \widehat{\mtx\Omega}_{\msf S_{\ell+1}}
    +
    \mathcal U_{\msf S_{\ell+1}}^*
    \left(
        {\textstyle\sum}_{j=\ell+1}^{L}\mtx E_j
    \right)\mtx\Omega .
\end{equation}
The corresponding adjoint identities are analogous.  Therefore carried recovery
samples the exact incoming coefficient matrix together with accumulated unrecovered
errors from finer levels.  

After level $\ell_0$ has been processed, the terminal coefficient matrix is
estimated from the surviving forward sketch by
\begin{equation}
\label{eq:carried-terminal-recovery}
    \widetilde{\mtx C}_{\msf S_{\ell_0},\msf S_{\ell_0}}
    :=
    \widehat{\mtx Y}_{\msf S_{\ell_0}}
    \bigl(
        \widehat{\mtx\Omega}_{\msf S_{\ell_0}}
    \bigr)^\dagger.
\end{equation}
Provided $\widehat{\mtx\Omega}_{\msf S_{\ell_0}}$ has full row
rank, this estimate is exact when the accumulated carried
error is zero; otherwise, that error perturbs the estimate.

The incoming coordinate maps, accumulated errors, and carried
tests depend on the same original Gaussian draws.
Consequently, the conditional Gaussian argument used for
resketched recovery does not directly apply. \Cref{fig:multilevel-test-support} illustrates why a
finer-level recovery error is absent from a fresh resketched sample but
remains in a carried sketch.

\begin{remark}[Carried test matrices]
At coarser levels, the carried test blocks are obtained by
applying $\mathcal V_\sigma^*$ and $\mathcal U_\sigma^*$
to the original Gaussian test blocks $\mtx\Omega_\sigma$
and $\mtx\Psi_\sigma$, respectively.
These maps have orthonormal rows but depend on the original
tests, so the usual Gaussian invariance argument does not apply.
We therefore check the numerical row rank of each local test
matrix when computing its pseudoinverse. All checks passed
in our experiments. If a rank deficiency is detected, one can
resketch the affected level and resume carried recovery.
\label{remark:carried_rank}
\end{remark}

\begin{figure}[H]
    \centering
    \resizebox{0.98\linewidth}{!}{
        \begin{tikzpicture}[
    every node/.style={align=center},
    page/.style={
        draw=black,
        fill=white,
        line width=0.9pt,
        inner sep=0.22mm,
    },
]
    \definecolor{levelThreeRed}{HTML}{C94F5A}
    \definecolor{levelTwoGreen}{HTML}{91A68F}
    \definecolor{terminalBlue}{HTML}{8BC8D6}
    \definecolor{errorRed}{HTML}{9D3540}

\node[minimum width=40mm,minimum height=65mm,inner sep=0pt]
        (geom-frame) {};
    \node at (geom-frame.center) (geom)
        {\includegraphics[height=56mm]
        {recovery_multilevel_square_geometry_stack.pdf}};
    \node[font=\small,fill=white,inner sep=0.5pt,anchor=west]
        at ($(geom.center)+(6mm,-9mm)$)
        {geometric\\ hierarchy};
    \node[font=\small,anchor=east]
        at ($(geom.west)+(-1mm,18mm)$) {$\msf S_2$};
    \node[font=\small,anchor=east]
        at ($(geom.west)+(-1mm,0)$) {$(\msf S_2,\msf R_2)$};
    \node[font=\small,anchor=east]
        at ($(geom.west)+(-1mm,-18mm)$) {$(\msf S_3,\msf R_3)$};

\coordinate (right-top) at ($(geom-frame.east)+(15mm,17mm)$);
    \node[page,anchor=west] (Yresk) at (right-top)
        {\includegraphics[width=6mm,height=24mm]
        {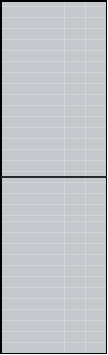}};
    \draw[draw=levelTwoGreen,line width=1.1pt]
        (Yresk.north west) rectangle
        ($(Yresk.north east)!0.5!(Yresk.south east)$);
    \node[right=2.5mm of Yresk,font=\normalsize] (equals-resk) {$=$};
    \node[page,right=2.5mm of equals-resk] (Cresk)
        {\includegraphics[width=24mm,height=24mm]
        {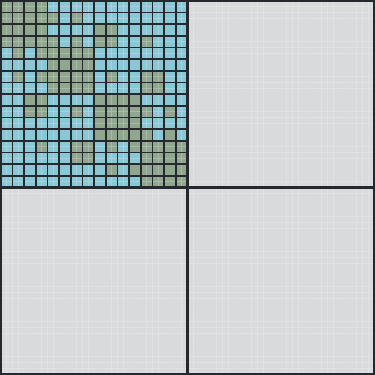}};
    \begin{scope}
        \clip
            (Cresk.north) rectangle (Cresk.east)
            (Cresk.west) rectangle (Cresk.south)
            (Cresk.center) rectangle (Cresk.south east);
        \node[inner sep=0pt,outer sep=0pt,opacity=0.34,
            blend mode=multiply] at (Cresk.center)
            {\includegraphics[width=24mm,height=24mm]
            {figures/recovery_multilevel_square_native_C3.pdf}};
    \end{scope}
    \node[font=\tiny,fill=white,fill opacity=0.88,
        text opacity=1,inner sep=0.25mm]
        at ($(Cresk.north west)!0.5!(Cresk.center)$)
        {$\mtx C_{\msf S_3,\msf S_3}$};
    \node[right=2.5mm of Cresk,font=\normalsize] (times-resk) {$\times$};
    \node[page,right=2.5mm of times-resk] (Oresk)
        {\includegraphics[width=6mm,height=24mm]
        {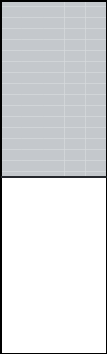}};
    \draw[draw=levelTwoGreen,line width=1.1pt]
        (Oresk.north west) rectangle
        ($(Oresk.north east)!0.5!(Oresk.south east)$);
    \node[font=\tiny,fill=white,fill opacity=0.88,
        text opacity=1,inner sep=0.2mm]
        at ($(Oresk.center)+(0,6mm)$)
        {$\widehat{\mtx\Omega}_{\msf S_3}$};
    \node[font=\scriptsize,fill=white,inner sep=0.35pt]
        at ($(Oresk.center)+(0,-6mm)$) {$\mtx 0$};
    \node[above=1.4mm of Cresk,font=\small]
        {(a) resketched recovery};
    \node[font=\small,anchor=east]
        at ($(Yresk.north west)!0.5!(Yresk.west)+(-1mm,0)$)
        {$\widehat{\mtx Y}_{\msf S_3}$};
    \node[font=\tiny,anchor=west]
        at ($(Oresk.east)+(0.7mm,6mm)$) {$\msf S_3$};
    \node[font=\tiny,anchor=west]
        at ($(Oresk.east)+(0.7mm,-6mm)$) {$\msf R_3$};

\node[page,anchor=west] (Ycar)
        at ($(Yresk.west)+(0,-36mm)$)
        {\includegraphics[width=6mm,height=24mm]
        {figures/recovery_level3_resketched_sketch.pdf}};
    \begin{scope}
        \clip ($(Ycar.north west)+(0.3mm,-0.3mm)$) rectangle
            ($(Ycar.east)+(-0.3mm,0.3mm)$);
        \node[inner sep=0pt,outer sep=0pt,opacity=0.72,
            blend mode=multiply]
            at ($(Ycar.north)!0.5!(Ycar.center)$)
            {\includegraphics[viewport=90 90 180 180,clip,
                width=5.4mm,height=11.4mm]
            {figures/recovery_multilevel_square_native_C3.pdf}};
    \end{scope}
    \draw[draw=levelTwoGreen,line width=1.1pt]
        (Ycar.north west) rectangle
        ($(Ycar.north east)!0.5!(Ycar.south east)$);
    \node[right=2.5mm of Ycar,font=\normalsize] (equals-car) {$=$};
    \node[page,right=2.5mm of equals-car] (Ccar)
        {\includegraphics[width=24mm,height=24mm]
        {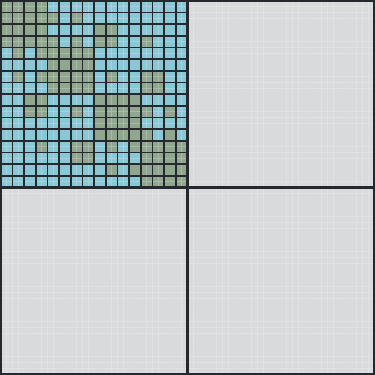}};
    \begin{scope}
        \clip
            (Ccar.north) rectangle (Ccar.east)
            (Ccar.west) rectangle (Ccar.south)
            (Ccar.center) rectangle (Ccar.south east);
        \node[inner sep=0pt,outer sep=0pt,opacity=0.34,
            blend mode=multiply] at (Ccar.center)
            {\includegraphics[width=24mm,height=24mm]
            {figures/recovery_multilevel_square_native_C3.pdf}};
    \end{scope}
    \node[font=\tiny,fill=white,fill opacity=0.88,
        text opacity=1,inner sep=0.25mm]
        at ($(Ccar.north west)!0.5!(Ccar.center)$)
        {$\mtx C_{\msf S_3,\msf S_3}$};
    \node[right=2.5mm of Ccar,font=\normalsize] (times-car) {$\times$};
    \node[page,right=2.5mm of times-car] (Ocar)
        {\includegraphics[width=6mm,height=24mm]
        {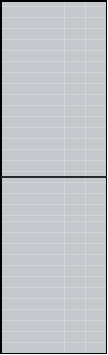}};
    \node[font=\tiny,fill=white,fill opacity=0.88,
        text opacity=1,inner sep=0.2mm]
        at ($(Ocar.center)+(0,6mm)$)
        {$\widehat{\mtx\Omega}_{\msf S_3}$};
    \node[font=\tiny,fill=white,fill opacity=0.88,
        text opacity=1,inner sep=0.2mm]
        at ($(Ocar.center)+(0,-6mm)$)
        {$\widehat{\mtx\Omega}_{\msf R_3}$};
    \node[above=1.4mm of Ccar,font=\small]
        {(b) carried recovery};
    \node[font=\small,anchor=east]
        at ($(Ycar.north west)!0.5!(Ycar.west)+(-1mm,0)$)
        {$\widehat{\mtx Y}_{\msf S_3}$};
    \node[font=\tiny,anchor=west]
        at ($(Ocar.east)+(0.7mm,6mm)$) {$\msf S_3$};
    \node[font=\tiny,anchor=west]
        at ($(Ocar.east)+(0.7mm,-6mm)$) {$\msf R_3$};

\end{tikzpicture}
     }
    \caption{After level-$3$ recovery, we sample the remaining
    matrix $\mtx C_{\msf S_3,\msf S_3}$ for recovery at level $2$.
    The resketched test has components only on the surviving
    coordinates $\msf S_3$. The carried test also has components
    on $\msf R_3$, allowing unrecovered level-$3$ errors
    (shown in faded red) to enter the sketch.}
    \label{fig:multilevel-test-support}
\end{figure}

\subsection{Query Counts and Complexity}
\label{subsec:recovery-complexity}

We count each right-hand-side column in an application of
$\mtx A$ or $\mtx A^*$ as one query.
At level $\ell$, all boxes share the same test columns, so
the sample width is determined by the largest local requirement.
Let $m_\sigma$ denote the number of incoming coordinates
for a box $\sigma$ at level $\ell$, and set
\[
    n_\tau:={\textstyle\sum}_{\sigma\in\near_\ell(\tau)}m_\sigma,
    \qquad
    s_\ell:=\max_{\tau\in\mathcal T_\ell}(n_\tau+k_\tau+p).
\]
This choice leaves at least $k_\tau+p$ null directions for
each processed box at level $\ell$.

Resketched recovery draws new tests at every level, while
carried recovery uses one initial sketch pair wide enough
for all levels.
Set $d_{\ell_0}:=|\msf S_{\ell_0}|$. 
With $s_\ell$ columns at level $\ell$, the forward-query
count for resketched recovery is
\begin{equation}
\label{eq:resketched-query-count}
    s_{\mathrm{resketch}}
    := \sum_{\ell=\ell_0}^{L}s_\ell+d_{\ell_0}.
\end{equation}
The final $d_{\ell_0}$-column application is needed only for $\mtx A$.
The adjoint-query count is $s_{\mathrm{resketch}}-d_{\ell_0}$.

Provided the local row-rank checks in \Cref{remark:carried_rank} pass, 
carried recovery uses the same query count for $\mtx A$
and $\mtx A^*$:
\begin{equation}
\label{eq:rank-query-count}
    s_{\mathrm{carried}}
    :=\max\left\{d_{\ell_0}+p,
       \max_{\ell_0\leq\ell\leq L}s_\ell\right\}.
\end{equation}
The term $d_{\ell_0}+p$ provides oversampling for recovery
of the terminal coefficient matrix.

To bound the sample widths, let $k:=\max_{\ell,\tau\in\mathcal T_\ell}k_\tau$.
Assume there are
at most $c_{\mathrm{leaf}}k$ coordinates in a leaf.
A nonleaf receives at most $k$ surviving coordinates from
each child, so at most $bk$ if it has at most $b$ children.
Thus every box $\sigma\in\mathcal B_\ell$, including a
coarser leaf, satisfies
$m_\sigma\le\max\{c_{\mathrm{leaf}},b\}k$.
Let $\nu$ bound $|\near_\ell(\tau)|$ for every processed box
$\tau\in\mathcal T_\ell$, including $\tau$ itself.
These neighbors belong to $\mathcal B_\ell$: they are
level-$\ell$ boxes or coarser leaves. Therefore $n_\tau\le\nu\max\{c_{\mathrm{leaf}},b\}k$, and
\begin{equation}
    s_\ell\le
    c_{\rm geom}\ k+p, \qquad c_{\rm geom}:= \nu\max\{c_{\mathrm{leaf}},b\}+1.
\end{equation}

For fixed geometric
constants and $\ell_0$, $d_{\ell_0}=O(k)$. 
Combining \cref{eq:resketched-query-count,eq:rank-query-count} with
$L=O(\log(N/k))$ and $p=O(k)$ gives the total query complexity
\[
\begin{aligned}
\text{carried:}\quad
    &O(k+p)\\
\text{resketched:}\quad
    &O\bigl(L(k+p)\bigr)=O\bigl(k\log(N/k)\bigr).
\end{aligned}
\]
If the tree has $O(N/k)$ boxes, the representation uses $O(Nk)$ storage and applies in
$O(Nk)$ work per right-hand side.  Excluding applications of $\mtx A$ and
$\mtx A^*$, the postprocessing costs are
\[
\begin{aligned}
\text{carried:}\quad
    &O\bigl(N(k+p)^2\bigr)
      =O(Nk^2),\\
\text{resketched:}\quad
    &O\bigl(N(k+p)^2+Nk(k+p)L\bigr)
      =O\bigl(Nk^2\log(N/k)\bigr).
\end{aligned}
\]
The second
term in resketched recovery is the cost of mapping fresh
test matrices to physical coordinates and the
resulting sketches back at every level.

\begin{remark}[Geometry-dependent query counts]
\label{rem:geometry-query-counts}
For a quadtree ($d=2$) or octree ($d=3$), assume leaf size
at most $2^d k$ and $2{:}1$ balance between adjacent leaves.
Because the number of near neighbors in $\mathcal B_\ell$ is bounded by $3^d$,
one may take $c_{\rm {geom}}=6^d+1$.
Actual neighborhood sizes can give smaller sample widths,
particularly for lower-dimensional point sets. The wiggly-torus experiment in
\cref{subsec:exp-wtorus} uses approximately $120k+p$ samples rather than the
three-dimensional bound $217k+p$.  For weak admissibility on a binary tree with leaf size
at most $2k$, one may take $c_{\mathrm{geom}}=3$,
without a balance assumption.
\end{remark}

\subsection{Resketched Recovery Guarantee}
\label{subsec:sample-planning-guarantee}

For a prescribed tree $\mathcal T$, rank bounds $(k_\tau)$,
and neighbor relations $\mathcal N=(\near_\ell)$, let
$\mathcal H^2$ denote the class of matrices represented exactly
by the right-hand side of \eqref{eq:levelwise-expansion},
with nested orthonormal box bases of ranks at most $k_\tau$
and arbitrary coefficients.
The best approximation error in this class is
\begin{equation}
\label{eq:projected-residual-comparison-class}
    \mathrm{OPT}
    :=\inf_{\overline{\mtx A}
        \in\mathcal H^2}
        \normfro{\mtx A-\overline{\mtx A}}.
\end{equation}
For $p\geq2$, collect the one-level recovery factors into
\begin{equation}
\label{eq:multilevel-local-beta}
    \beta
    :=\max_{\ell_0\leq\ell\leq L}
      \max_{\tau\in\mathcal T_\ell}
      \alpha_{k_\tau,p}
      \sqrt{1+{n_\tau}/{(k_\tau+p-1)}}.
\end{equation}
Here $\alpha_{k_\tau,p}$ is defined in
\cref{eq:appendix-range-alpha}. By \Cref{lem:one-level-recovery},
$\beta$ bounds the one-level factor for every box and level.

\begin{theorem}[Quasi-optimal resketched recovery]
\label{thm:resketched-quasioptimality}
Suppose $p\geq2$.
The resketched algorithm of \Cref{subsec:resketched-recovery} produces an approximation
$\widetilde{\mtx A}$ satisfying
\begin{equation}
\label{eq:resketched-error-bound}
    \E\normfro{\mtx A-\widetilde{\mtx A}}
    \leq
    \sqrt{2(L-\ell_0+1)}\,\beta\,\mathrm{OPT}.
\end{equation}
The numbers of queries to $\mtx A$ and $\mtx A^*$ are given in
\Cref{subsec:recovery-complexity}.
\end{theorem}

The proof is in \Cref{app:multilevel-proof}.

\begin{corollary}[Weak admissibility]
\label{cor:weak-admissibility-recovery}
Consider a binary tree.  Take $\ell_0=1$ and
$\near_\ell(\tau)=\{\tau\}$, and assume $L=O(\log(N/k))$.  In the
uniform-rank case, leaf dimensions at most $2k$ give
\[
    \E\normfro{\mtx A-\widetilde{\mtx A}}
    \leq
    \sqrt{2L}\,\alpha_{k,p}
    \sqrt{1+{2k}/{(k+p-1)}}\,\mathrm{OPT}.
\]
Here $\mathrm{OPT}$ uses the corresponding weak-admissibility representation.
The query counts are at most $L(3k+p)+2k$ for $\mtx A$ and $L(3k+p)$ for
$\mtx A^*$.
With $p=O(k)$, both counts are $O(k\log(N/k))$.
\end{corollary}

\begin{remark}[Uniform bound on $\beta$]
Suppose the rank is at most $k$. Set $p-1=\gamma k$ with fixed $\gamma>0$. 
Then $\alpha_{k,p}$ is bounded
independently of $k$, and
\[
    \beta\leq
    \alpha_{k,p}\sqrt{1+(c_{\rm geom}-1)/{(1+\gamma)}}.
\]
For fixed $c_{\mathrm{geom}}$ and $\gamma$, $\beta$ is
bounded independently of $N$ and $k$.
For trees of logarithmic depth, the factor multiplying $\mathrm{OPT}$
in \cref{eq:resketched-error-bound} is therefore
$O(\sqrt{\log(N/k)})$.
\end{remark}

 \section{Numerical Experiments}
\label{sec:experiments}

We compare carried and resketched recovery on three discretized
operators to assess how reusing sketches across levels affects
accuracy and reduces operator queries.
The three test problems differ in how products with $\mtx A$ and
$\mtx A^*$ are obtained. The first is an explicit Laplace kernel matrix, for which
inexpensive operator applications permit controlled studies of oversampling
and admissibility.  The second is a matrix-free Dirichlet-to-Neumann map,
where each operator application requires two interior volume solves.  The
third is an oscillatory Helmholtz boundary-integral operator on a curved
surface, applied through a fast multipole method.

For carried recovery, we report its
reconstruction time $T_{\mathrm{rec}}$, compressed application time per
right-hand side $T_{\mathrm{apply}}$, and stored-representation size $M$.
Reconstruction timings exclude sketch generation and
applications of $\mtx A$ and $\mtx A^*$. 
All timings were obtained on Cardinal at the Ohio
Supercomputer Center using Intel Xeon CPU Max 9470 nodes
(2.0\,GHz), with numerical-library thread pools limited to one thread.

For both recovery methods, we estimate relative Frobenius error using
\[
    \operatorname{relerr}_F
    :=
    {\|(\mtx A-\widetilde{\mtx A})\mtx G\|_F}\ /\ 
         {\|\mtx A\mtx G\|_F},
\]
where $\mtx G$ has 20 Gaussian columns independent of the recovery sketches.
Accuracy results use a single randomized realization with
prescribed random seeds.

In the numerical comparisons, resketched recovery uses
$s_{\mathrm{carried}}$ columns at every processed level.  Its forward-query
count is therefore
\[
    s_{\mathrm {resketch}} = (L-\ell_0+1)s_{\mathrm{carried}}+d_{\ell_0}.
\]
In \cref{fig:biegrid-main,fig:dtn-main,fig:wtorus-main},
colored annotations give storage in GB in the application-time
panels and forward-query counts in the error panels
(upright for carried, italic for resketched recovery).
The reconstruction panels list the maximum carried query count
over problem sizes for each rank.
All query counts exclude the independent validation probes.
Each reported $\beta$ is the maximum over the corresponding
ranks and problem sizes.

\subsection{Laplace Kernel on a Square Grid}
\label{subsec:exp-biegrid}

For points $x_i$ on a uniform grid in the unit square,
the Laplace kernel matrix has off-diagonal entries
$$\mtx A_{ij} = -{(2 \pi N)}^{-1} \log \|x_i - x_j\|.$$
The diagonal entries
are local cell integrals evaluated using QUADPACK
\cite{piessens1983quadpack}. Operator products are evaluated using
FFTs. The strong-admissibility experiments use $k=10,20,30$,
$p=4k$, and a maximum leaf size of $4k$.

As $N$ grows from approximately 20 thousand to 1.28 million,
carried reconstruction time, compressed application time,
and storage scale approximately linearly
(\cref{fig:biegrid-main}).
At $k=20$, carried recovery uses 820 forward queries
throughout, while its error increases from $2.38\times10^{-7}$
to $4.65\times10^{-7}$. On the largest grid, resketched
recovery attains an error of $1.72\times10^{-7}$ using
6060 forward queries. The query savings from carried
recovery grow as the hierarchy deepens.

\begin{figure}[t]
\centering
\includegraphics[width=\linewidth]{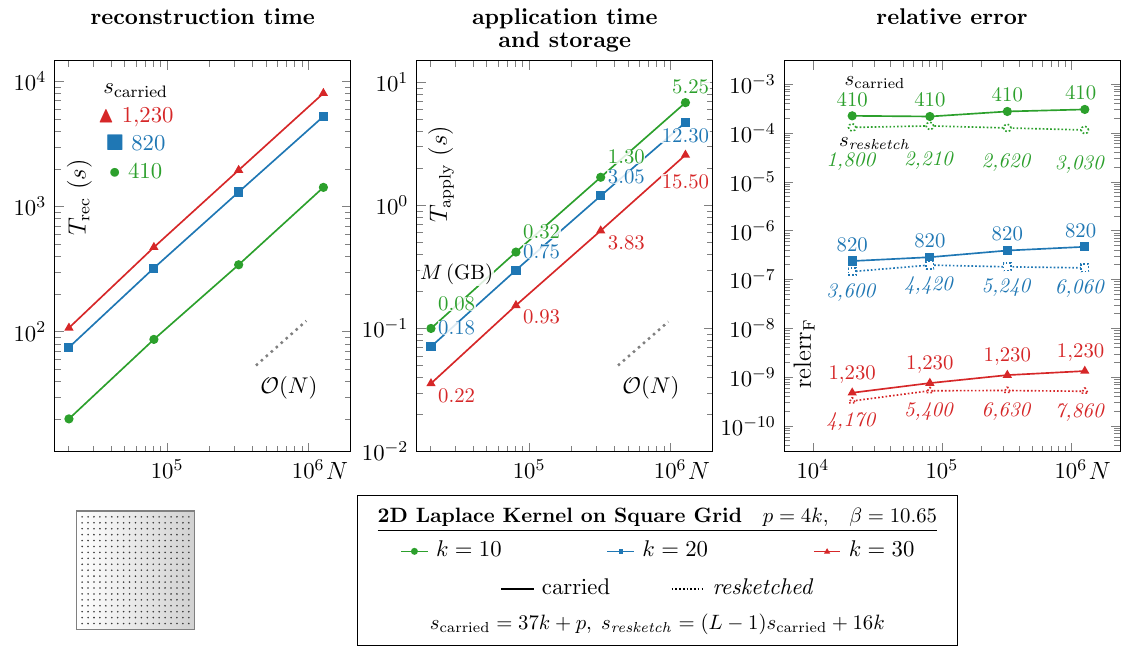}
\caption{Laplace kernel recovery with strong admissibility,
 for $k=10,20,30$ and $p=4k$. The first two panels show carried reconstruction and
compressed application timings; the third compares carried (solid) and resketched
(dotted) errors.}
\label{fig:biegrid-main}
\end{figure}

\begin{samepage}
\subsubsection{Effect of Oversampling}

We next vary the oversampling parameter $p$ at fixed rank
$k=20$ and maximum leaf size $4k$
(\cref{fig:biegrid-beta-sweep}), beginning with resketched recovery.
Increasing $p$
from 10 to 160 reduces $\beta$ from 31.94 to 7.59,
a factor of 4.2. At each fixed problem size, the approximation
class and $\mathrm{OPT}$ remain unchanged, so the theoretical
upper bound on the expected error decreases by the same factor.
The observed resketched error decreases by factors of 3.3--4.1
across the tested sizes, 
consistent with the bound in \Cref{thm:resketched-quasioptimality}.
For carried recovery, larger $p$ reduces the observed error
growth with $N$ and brings its accuracy closer to that of
resketched recovery.

\end{samepage}

\begin{figure}[!htb]
\centering
\includegraphics[width=0.96\linewidth]{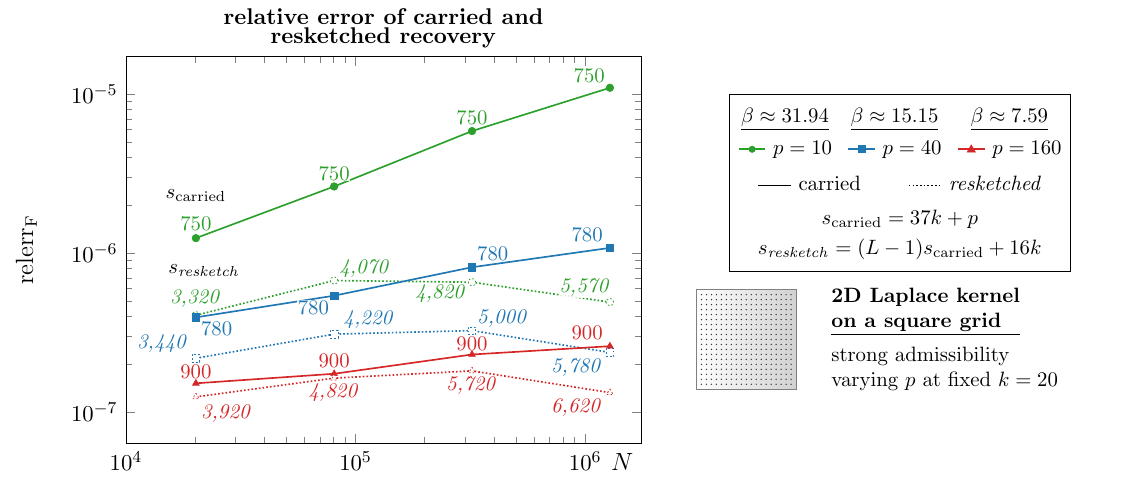}
\caption{Effect of oversampling on Laplace kernel recovery at $k=20$. 
Colors distinguish
\(p=10,40,160\); solid and dotted curves denote carried and resketched recovery.}
\label{fig:biegrid-beta-sweep}
\end{figure}

\subsubsection{Comparison to Weak Admissibility}

We compare weak and strong admissibility at two fixed sampling
budgets, $s=410$ and $s=820$, using the same quadtree within
each comparison. For these trees, the carried sample counts
are $s=5k+p$ for weak admissibility and $s=37k+p$ for strong
admissibility. With the same oversampling parameter $p$,
weak admissibility therefore permits a much larger
rank at each budget.

Weakly admissible bases must also approximate interactions
with neighboring boxes, so fixed-rank recovery can lose
accuracy as the grid is refined, as observed at the larger
sampling budget. Despite its larger geometric sampling
constant, strongly admissible recovery gives smaller errors
at both budgets (\cref{fig:biegrid-weak-strong}), with errors
that remain nearly independent of problem size.

\begin{figure}[!htb]
\centering
\includegraphics[width=\linewidth]{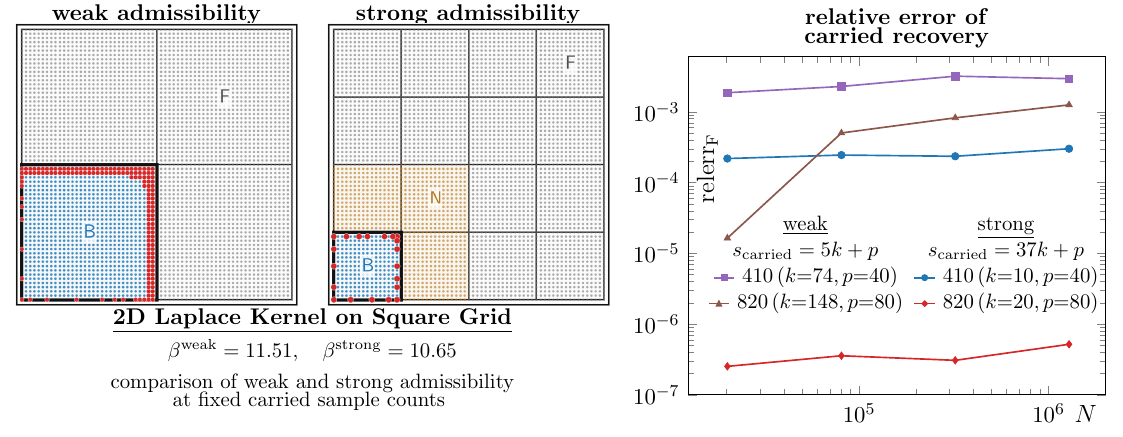}
\vspace{-0.20em}
\caption{Laplace kernel weak and strong carried recovery at matched
query budgets. Left: far fields and representative points;
right: recovery errors. Here $\mathsf B$, $\mathsf N$, and
$\mathsf F$ denote the box, its near neighbors, and its far field.}
\label{fig:biegrid-weak-strong}
\end{figure}

\subsection{Dirichlet-to-Neumann (DtN) Map}
\label{subsec:exp-dtn}

We consider a diffusion--reaction problem in two subdomains
with different diffusion coefficients. After discretization,
eliminating the interior unknowns leaves a system on their
shared interface. We recover the resulting interface operator
from matrix--vector products, each of which requires an
interior solve in both subdomains.

The domain is the cube $\Omega=[-1,1]^3$, divided by the curved interface
\[
    \Gamma
    =
    \{(x,y,z)\in\Omega:
      x=\gamma(y,z)\},
\]
where $\gamma(y,z)= 0.16\ \sin(1.35\pi y)\cos(1.15\pi z)$.
The subdomains $\Omega_1$ and $\Omega_2$ lie on either side
of $\Gamma$. We discretize the diffusion--reaction equation
\[
    -\nabla\cdot(a\nabla u)+0.25u=f,
    \qquad
    a|_{\Omega_1}=1,\quad a|_{\Omega_2}=4,
\]
using quadratic DOLFINx elements on a conforming mapped
tetrahedral mesh.

Let $I$ collect the interior unknowns of both subdomains, and define
$\mtx K_{II}=\operatorname{diag}(\mtx K_{11},\mtx K_{22})$.
Block elimination gives
\[
\begin{bmatrix}
    \mtx K_{II} & \mtx K_{I\Gamma}\\
    \mtx K_{\Gamma I} & \mtx K_{\Gamma\Gamma}
\end{bmatrix}
=
\begin{bmatrix}
    \mtx I & \mtx 0\\
    \mtx K_{\Gamma I}\mtx K_{II}^{-1} & \mtx I
\end{bmatrix}
\begin{bmatrix}
    \mtx K_{II} & \mtx 0\\
    \mtx 0 & \mtx S_\Gamma
\end{bmatrix}
\begin{bmatrix}
    \mtx I & \mtx K_{II}^{-1}\mtx K_{I\Gamma}\\
    \mtx 0 & \mtx I
\end{bmatrix}
\]
where
\[
    \mtx A := \mtx S_\Gamma
    =\mtx K_{\Gamma\Gamma}
     -\mtx K_{\Gamma I}\mtx K_{II}^{-1}\mtx K_{I\Gamma}.
\]
The discrete DtN map $\mtx S_\Gamma$ maps interface
Dirichlet data on $\Gamma$ to the sum of the outward fluxes from the two
subdomains.
The map is defined for $f=0$ and homogeneous outer Dirichlet
conditions. Body forcing and nonzero outer boundary data
enter the reduced right-hand side.
For this elliptic problem,
interactions between well-separated portions of the interface
are numerically low rank.

The dense interface operator can be applied through sparse
matrix products and interior solves, without forming its entries.
Interface Schur
complements of this type arise in structured multifrontal and slab-based
sparse direct solvers
\cite{xia2010superfast,xia2013randomized,yesypenko2024slablu}.
The largest DtN matrix has 1.28 million interface unknowns
and arises from a volume discretization with 42.4 million
unknowns. The sparse interior factorization was distributed
across four compute nodes.
Once constructed, the compressed operator can be applied
without further volume solves.

The randomized recovery experiments use ranks \(k=5,10,15\),
oversampling \(p=5k\), and maximum leaf size \(6k\).  Across four matrix dimensions from 20 thousand to
1.28 million unknowns, the error decreases with rank and remains nearly
independent of size.  On the largest interface at $k=15$, carried recovery uses
1175 forward queries, compared with 7290 for resketched
recovery, with errors of $2.40\times10^{-6}$ and
$1.45\times10^{-6}$, respectively (\cref{fig:dtn-main}).

\begin{figure}[t]
\centering
\begin{subfigure}[t]{0.68\linewidth}
\vspace{0pt}
\centering
\includegraphics[width=\linewidth]{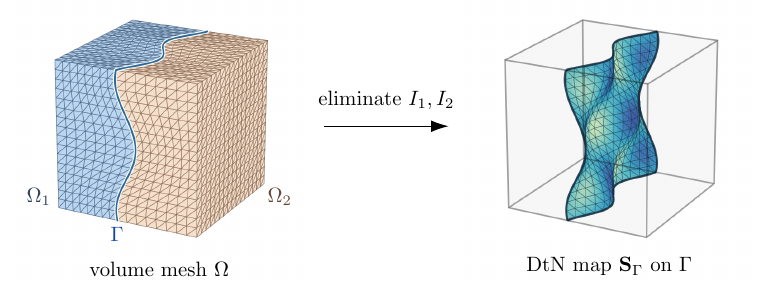}
\caption{Volume mesh and the resulting DtN map.}
\label{fig:dtn-elimination-geometry-panel}
\end{subfigure}\hfill
\begin{subfigure}[t]{0.28\linewidth}
\vspace{2.55em}
\centering
\scriptsize
\renewcommand{\arraystretch}{1.08}
\renewcommand{\tabcolsep}{2.6pt}
\centering
\begin{tabular}{@{}rrr@{}}
\toprule
$N_\Gamma$ & $N_\Omega$ & PDE err. \\
\midrule
20K & 675K & $3.49\times 10^{-6}$ \\
80K & 2.64M & $4.66\times 10^{-7}$ \\
320K & 10.6M & $6.67\times 10^{-8}$ \\
1.28M & 42.4M & $1.18\times 10^{-8}$ \\
\bottomrule
\end{tabular} \vspace{1.05em}
\caption{Validation.}
\label{fig:dtn-volume-validation}
\end{subfigure}
\caption{DtN map construction and validation:
(a) volume elimination forming $\mtx S_\Gamma$;
(b) relative $L^2(\Omega)$ error of the full-volume
manufactured-solution solve.}
\label{fig:dtn-elimination-geometry}
\end{figure}

\begin{figure}[t]
\centering
\includegraphics[width=\linewidth]{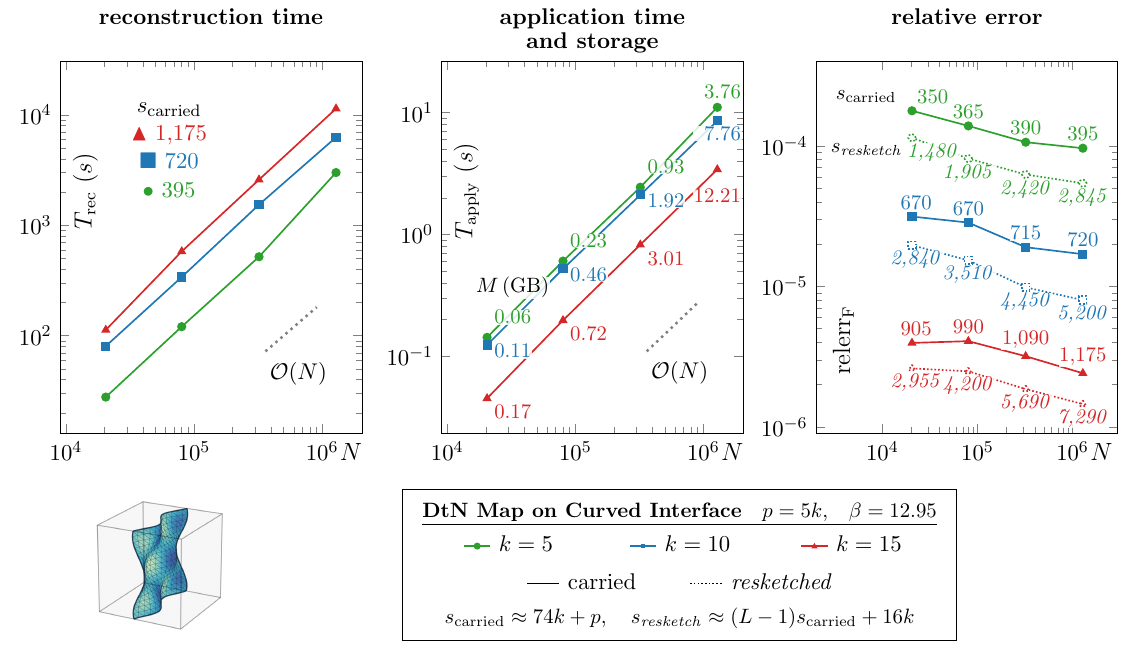}
\caption{DtN map recovery with $k=5,10,15$ and $p=5k$.
Timings and storage use carried recovery; errors compare
carried (solid) and resketched (dotted) recovery.}
\label{fig:dtn-main}
\end{figure}

The volume discretization was validated using the manufactured solution
\[
    u_\xi(x,y,z) =
    {\left(x-\gamma(y,z)\right)}\ /\ {a_\xi},
    \qquad \xi\in\{1,2\},
\]
with corresponding forcing and boundary data. 
This choice preserves continuity of the solution and normal
flux across $\Gamma$ despite the jump in diffusion coefficient.
Multigrid solves yield the
relative $L^2(\Omega)$ errors in \Cref{fig:dtn-elimination-geometry}(b),
with observed convergence orders of approximately 2.5--2.9, which is
consistent with the expected rate for quadratic elements.

\subsection{Helmholtz Boundary Integral Equation (BIE) on a Wiggly Torus}
\label{subsec:exp-wtorus}

\begin{figure}[!t]
\centering
\begin{subfigure}[t]{0.64\linewidth}
\vspace{0pt}
\centering
\begin{minipage}[t]{0.49\linewidth}
\vspace{0pt}
\centering
\IfFileExists{figures/wtorus_sigma_ring6_kappa25p1.png}{\begin{tikzpicture}
      \node[inner sep=0pt] (img) {\includegraphics[width=\linewidth]{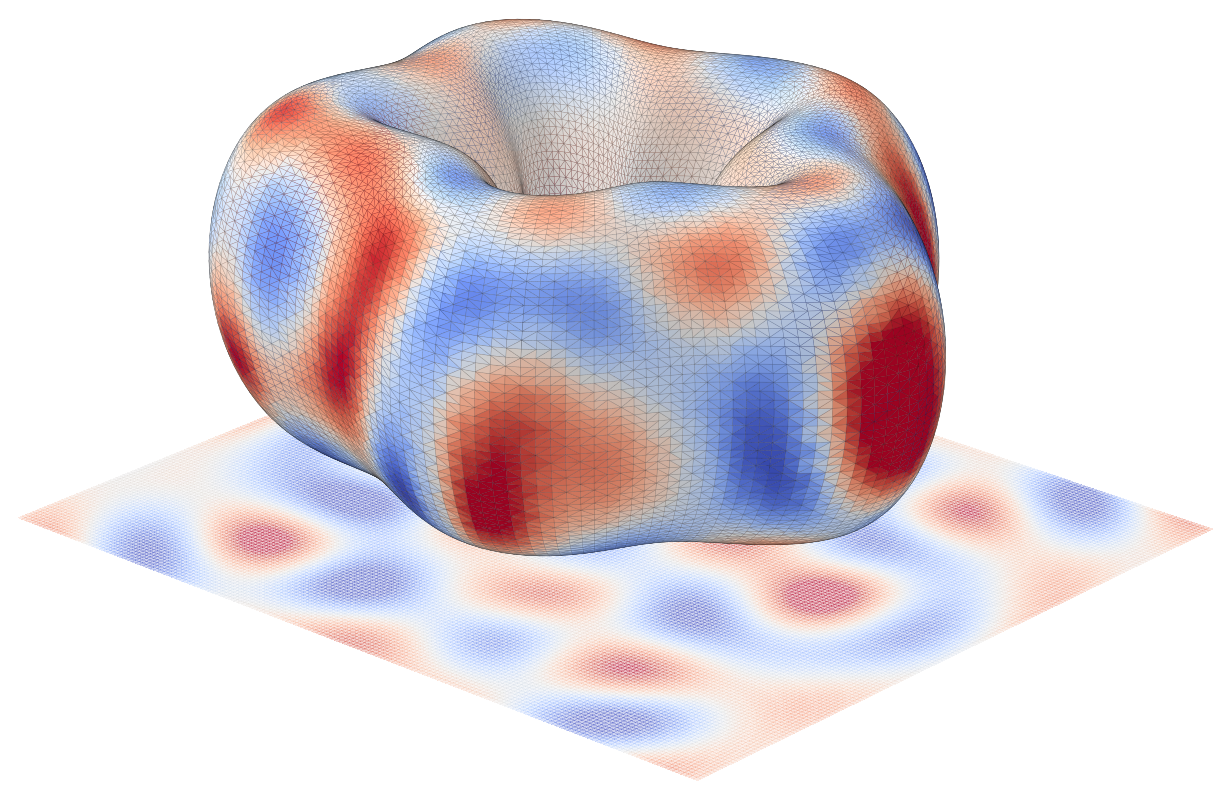}};
      \node[anchor=south west,font=\small,fill=white,fill opacity=0.76,text opacity=1,inner sep=1.0pt]
        at ([xshift=2.0pt,yshift=2.0pt]img.south west) {\(\kappa = 25.1\)};
    \end{tikzpicture}}{\fbox{\parbox[c][0.20\textheight][c]{0.90\linewidth}{\centering
    ring6 \(\kappa = 25.1\) density figure pending}}}
\end{minipage}\hfill
\begin{minipage}[t]{0.49\linewidth}
\vspace{0pt}
\centering
\IfFileExists{figures/wtorus_sigma_ring6_kappa50p3.png}{\begin{tikzpicture}
      \node[inner sep=0pt] (img) {\includegraphics[width=\linewidth]{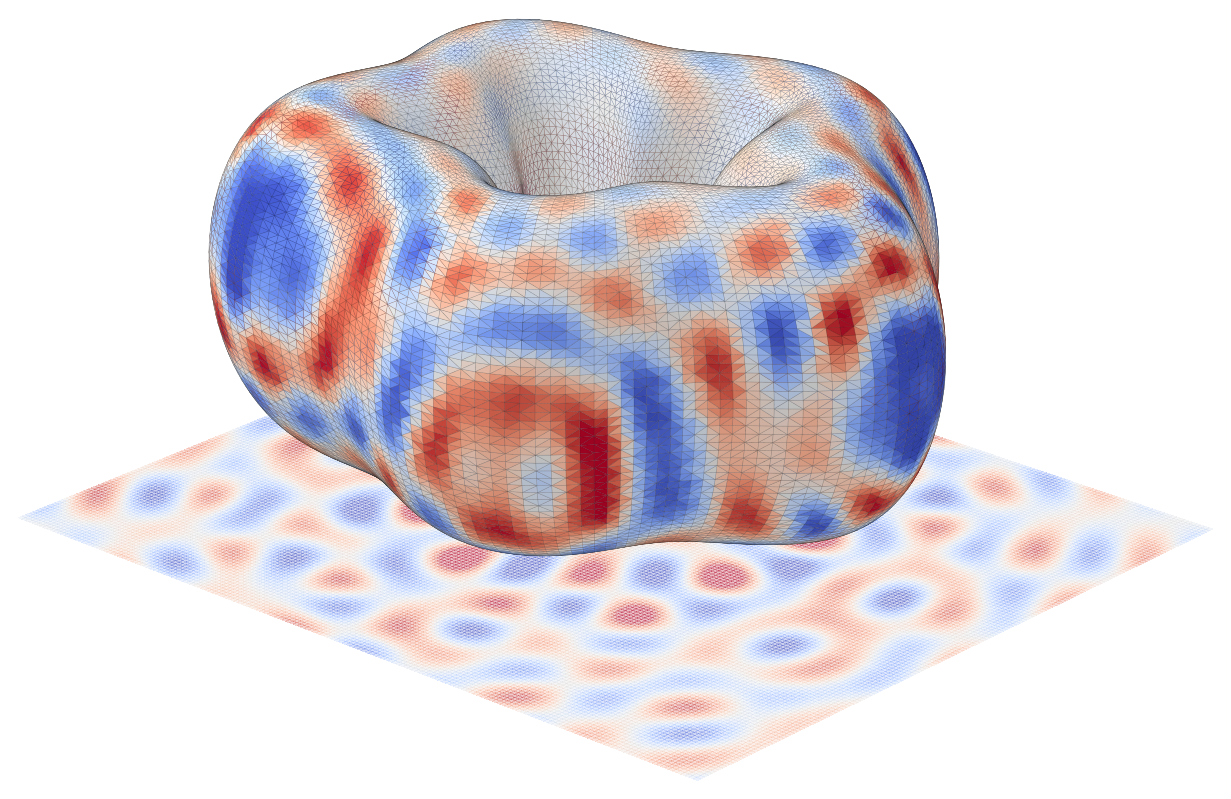}};
      \node[anchor=south west,font=\small,fill=white,fill opacity=0.76,text opacity=1,inner sep=1.0pt]
        at ([xshift=2.0pt,yshift=2.0pt]img.south west) {\(\kappa = 50.3\)};
    \end{tikzpicture}}{\fbox{\parbox[c][0.20\textheight][c]{0.90\linewidth}{\centering
    ring6 \(\kappa = 50.3\) density figure pending}}}
\end{minipage}
\caption{Surface-density renderings.}
\end{subfigure}\hfill
\begin{subfigure}[t]{0.32\linewidth}
\vspace{1.8em}
\centering
\scriptsize
\renewcommand{\arraystretch}{1.08}
\renewcommand{\tabcolsep}{2.6pt}
\centering
\begin{tabular}{@{}r r r@{}}
\toprule
$N$ & ${\kappa=25.1}$ & ${\kappa=50.3}$ \\
\midrule
20K & $8.78\times 10^{-4}$ & $3.55\times 10^{-3}$ \\
80K & $9.72\times 10^{-5}$ & $4.33\times 10^{-5}$ \\
320K & $1.29\times 10^{-5}$ & $4.70\times 10^{-6}$ \\
1.28M & $1.77\times 10^{-6}$ & $6.15\times 10^{-7}$ \\
\bottomrule
\end{tabular} \caption{Validation.}
\end{subfigure}
\caption{Helmholtz BIE on a wiggly torus. The wavenumbers $\kappa=25.1$ and $50.3$ correspond to
approximately four and eight wavelengths across the
largest bounding box dimension, respectively. (a) Six-source illustration: surface colors show
$\operatorname{Re}\sigma$, and the plane shows the real part
of the layer potential evaluated outside $\Omega$,
with separate color scales.
(b) Relative target-point errors for two-source
manufactured solutions.}
\label{fig:wtorus-geometry}
\end{figure}

\begin{figure}[!t]
\centering
\vspace{-3pt}
\includegraphics[width=\linewidth]{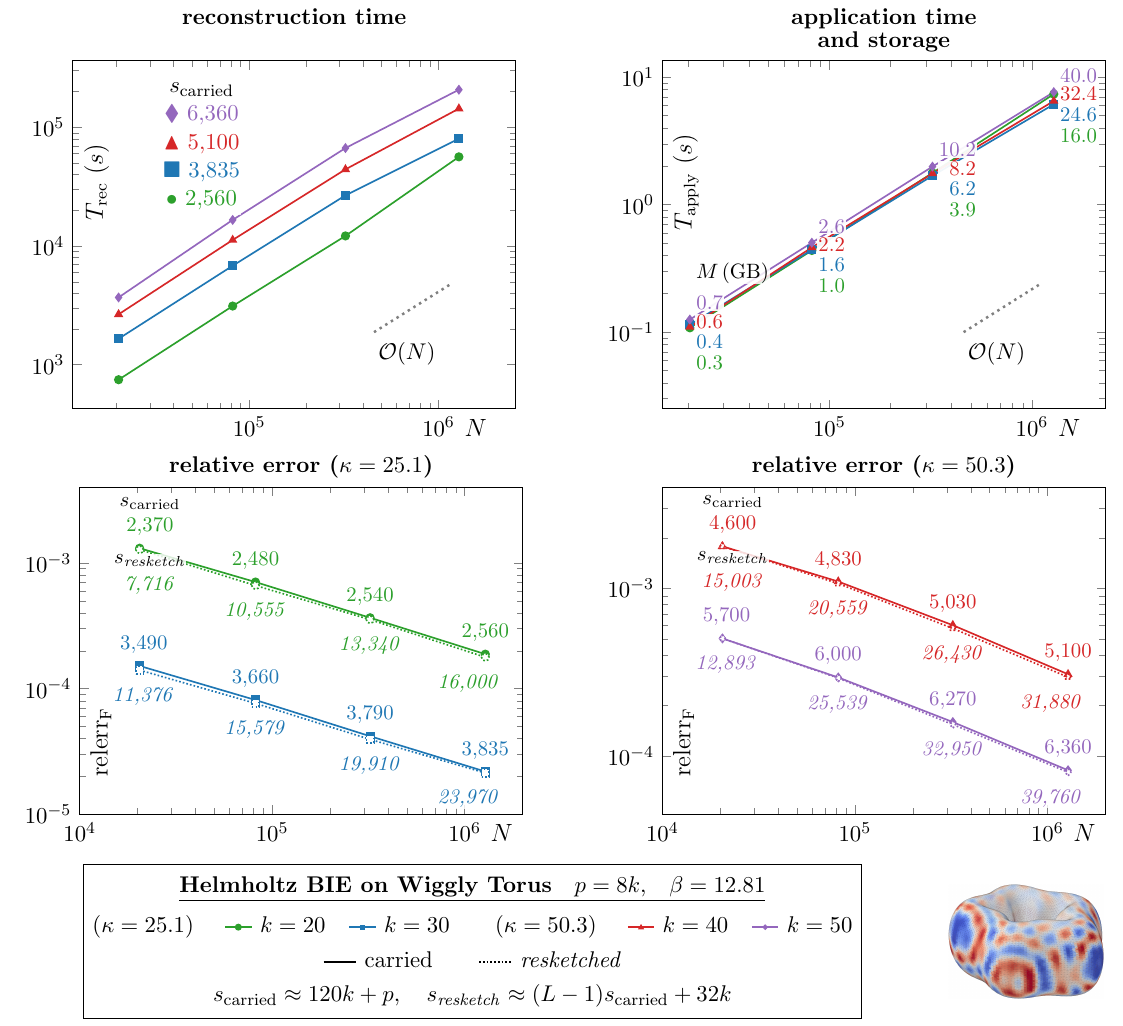}
\caption{Helmholtz BIE recovery.  The cases are
\(\kappa=25.1\) with ranks \(k=20,30\) and \(\kappa=50.3\) with ranks \(k=40,50\),
with \(p=8k\).}
\label{fig:wtorus-main}
\end{figure}

The final example concerns time-harmonic wave propagation in
a bounded domain $\Omega\subset\mathbb R^3$, whose boundary
$\Gamma$ is a five-oscillation torus scaled to a unit bounding
box. We consider the Helmholtz Dirichlet problem
\[
    (\Delta+\kappa^2)u=0 \quad\text{in }\Omega,
    \qquad
    u=g \quad\text{on }\Gamma,
\]
where $\kappa$ is the wavenumber and $g$ is prescribed boundary
data. With
$G_\kappa(x,y)=e^{i\kappa\|x-y\|}/(4\pi\|x-y\|)$,
we represent the solution by
\begin{equation}
    u(x)=\int_\Gamma
    \left[
        \frac{\partial G_\kappa(x,y)}{\partial n_y}
        +i\kappa G_\kappa(x,y)
    \right]\sigma(y)\,dS_y,
    \qquad x\in\Omega,
    \label{eq:wtorus-potential}
\end{equation}
where $n_y$ is the outward normal and $\sigma$ is an unknown
surface density. Taking the interior boundary trace and
imposing $u=g$ gives
\begin{equation}
    \left(
        -\frac12 I
        +i\kappa\mathcal S_\kappa
        +\mathcal D_\kappa
    \right)\sigma=g,
    \label{eq:wtorus-bie}
\end{equation}
where $\mathcal S_\kappa$ and $\mathcal D_\kappa$ are the
single- and double-layer boundary operators.
Once $\sigma$ is computed, evaluating
\eqref{eq:wtorus-potential} recovers $u$ at interior points.
The matrix $\mtx A$ to be recovered is the dense matrix
obtained by corrected Nystr\"om discretization of the boundary
operator in \eqref{eq:wtorus-bie} on curved triangular patches.
It maps discrete surface densities to Dirichlet boundary data.
Products with $\mtx A$ and $\mtx A^*$ are evaluated using a
Helmholtz FMM with local quadrature corrections
\cite{sushnikova2023fmmlu}.

We use ranks \(k=20,30\) for \(\kappa = 25.1\) and ranks \(k=40,50\)
for \(\kappa = 50.3\).  All recovery runs use oversampling \(p=8k\) and maximum
leaf size \(6k\).
Recovery samples are converted to complex single precision;
held-out samples remain in complex double precision. The larger ranks used at the higher frequency are consistent
with the frequency dependence of nondirectional hierarchical
compression for Helmholtz kernels \cite{engquist2007fast}.  On the largest problem, with 1.28 million unknowns, carried
and resketched errors differ by at most 6\%, while resketched
recovery uses approximately 6.25 times as many forward queries.

For validation, we prescribe $g$ as the boundary trace of
a Helmholtz field generated by two exterior point sources.
We solve the original FMM-based system using unrestarted,
unpreconditioned GMRES and compare the resulting potential
with the analytic field at interior target points. 
The relative target-point errors for both
frequencies appear in \cref{fig:wtorus-geometry}.
On the finer meshes, the target-point errors exhibit
approximately third-order convergence under mesh refinement.

\FloatBarrier
 \section{Conclusion}

We developed two randomized algorithms for recovering explicit,
strongly admissible $\mathcal H^2$ approximations using only
products with a matrix and its adjoint. Carried recovery requires
one initial forward--adjoint sketch pair and updates it throughout
the hierarchy. The samples can be acquired in a single batch and require
no hierarchy-specific probing. Under the stated
geometric and rank conditions, the required sample count is
independent of tree depth. Resketched recovery draws fresh
Gaussian tests at each level and satisfies a global quasi-optimal
bound on the expected Frobenius error relative to the prescribed
approximation class.

Experiments on Laplace, Dirichlet-to-Neumann, and Helmholtz
operators with up to $1.28$ million unknowns show approximately
linear storage and compressed application cost at fixed rank.
With sufficient oversampling,
carried recovery attains accuracy comparable to resketched
recovery while using many fewer operator queries. 
Developing a global error analysis for carried recovery is a
 direction for future work. 
\appendix

\section{Proofs and Supporting Estimates}
\label{app:proofs_and_supporting_estimates}
We first establish a fixed-rank Gaussian range estimate and
the one-level recovery bound with sketch reuse. We then combine
the one-level bound with orthogonality across levels to prove
the multilevel guarantee.

\subsection{Fixed-Rank One-Sided Sketch Range Estimate}
\label{app:randomized-estimates}

The following proof adapts the one-sided sketch argument of
\cite[Theorem~4.2]{amsel2026quasioptimal}, building on the
Gaussian range-finding analysis of
\cite[Sections~9--10]{halko2011finding}.

\begin{lemma}
\label{lem:gaussian-range-estimate}
Let $\field\in\{\R,\C\}$, let $\mtx A\in\field^{m\times n}$, and let
$1\leq k\leq m$ and $p\geq2$.  Draw a standard Gaussian matrix
$\mtx\Omega\in\field^{n\times(k+p)}$, and let $\mtx Q$ contain the top $k$
left singular vectors of $\mtx A\mtx\Omega$.  Define
\begin{equation}
\label{eq:appendix-range-alpha}
    \alpha_{k,p}^2
    :=
    1+
    \left(
        \sqrt{
            p\min\left\{
                {e^2(k+p)}/{(p^2-1)},
                {k}/{(p-1)}
            \right\}
        }
        +\sqrt{{k}/{(p-1)}}
    \right)^2.
\end{equation}
Then
\begin{equation}
\label{eq:appendix-gaussian-range}
    \E\normfro{(\mtx I-\mtx Q\mtx Q^*)\mtx A}^2
    \leq
    \alpha_{k,p}^2
    \min_{\rank(\mtx M)\leq k}\normfro{\mtx A-\mtx M}^2.
\end{equation}
\end{lemma}

\begin{proof}
If $n\leq k$, then
$\operatorname{range}(\mtx A\mtx\Omega)=\operatorname{range}(\mtx A)$ with
probability one, and
both sides of \cref{eq:appendix-gaussian-range} vanish.
Now assume $k<n$.  Let $\mtx A_1=\mtx U_1\mtx\Sigma_1\mtx V_1^*$ be a
best rank-$k$ approximation to $\mtx A$, and let
$\mtx A_2=\mtx U_2\mtx\Sigma_2\mtx V_2^*=\mtx A-\mtx A_1$.  Then
\begin{equation}
\label{eq:appendix-range-split}
    \begin{aligned}
        \normfro{(\mtx I-\mtx Q\mtx Q^*)\mtx A}^2
        &=
        \normfro{(\mtx I-\mtx Q\mtx Q^*)\mtx A_1}^2
        +\normfro{(\mtx I-\mtx Q\mtx Q^*)\mtx A_2}^2\\
        &\leq
        \normfro{(\mtx I-\mtx Q\mtx Q^*)\mtx A_1}^2
        +\normfro{\mtx\Sigma_2}^2.
    \end{aligned}
\end{equation}

It remains to bound the first term in
\cref{eq:appendix-range-split}, which measures the part of
$\mtx A_1$ missed by the recovered basis. Set $\mtx\Omega_i=\mtx V_i^*\mtx\Omega$ for $i=1,2$.  These are independent
standard Gaussian matrices, and
$\mtx\Omega_1\mtx\Omega_1^\dagger=\mtx I$ with probability one.
Right-multiplying the identity
$\mtx A\mtx\Omega
=\mtx U_1\mtx\Sigma_1\mtx\Omega_1
+\mtx U_2\mtx\Sigma_2\mtx\Omega_2$ by
$\mtx\Omega_1^\dagger\mtx V_1^*$ and rearranging gives
\[
        \mtx A_1=
        \mtx A\mtx\Omega\mtx\Omega_1^\dagger\mtx V_1^*
        -\mtx U_2\mtx\Sigma_2\mtx\Omega_2
        \mtx\Omega_1^\dagger\mtx V_1^*.
\]
Therefore
\begin{equation}
\label{eq:appendix-range-two-terms}
    \normfro{(\mtx I-\mtx Q\mtx Q^*)\mtx A_1}
    \leq
    \normfro{(\mtx I-\mtx Q\mtx Q^*)\mtx A\mtx\Omega}
    \normtwo{\mtx\Omega_1^\dagger}
    +\normfro{\mtx\Sigma_2\mtx\Omega_2\mtx\Omega_1^\dagger}.
\end{equation}

Let $\mtx\Pi=\mtx\Omega_1^\dagger\mtx\Omega_1$.  This is the orthogonal
projector onto $\operatorname{range}(\mtx\Omega_1^*)$ and has rank $k$.
Since $\mtx A\mtx\Omega\mtx\Pi$ has rank at most $k$,
the truncated SVD $\mtx Q\mtx Q^*\mtx A\mtx\Omega$
has no larger Frobenius approximation error than this competitor.
Additionally, $\mtx\Omega_1(\mtx I-\mtx\Pi)=\mtx 0$, and the competitor's
residual contains only trailing components. This gives
\begin{equation}
\label{eq:appendix-rank-p-competitor}
    \normfro{(\mtx I-\mtx Q\mtx Q^*)\mtx A\mtx\Omega}
    \leq
    \normfro{\mtx A\mtx\Omega(\mtx I-\mtx\Pi)}
    =
    \normfro{\mtx\Sigma_2\mtx\Omega_2(\mtx I-\mtx\Pi)}.
\end{equation}
Since $\mtx\Omega_2$ is independent of $\mtx\Omega_1$, we can apply the fixed-matrix
Gaussian identity \cite[Proposition~A.1]{halko2011finding}. Additionally, the matrix $(\mtx I-\mtx\Pi)$ has rank $p$.
\[
    \begin{aligned}
        \E_{\mtx\Omega_2}
        \normfro{\mtx\Sigma_2\mtx\Omega_2(\mtx I-\mtx\Pi)}^2
        &=
        \normfro{\mtx\Sigma_2}^2\normfro{\mtx I-\mtx\Pi}^2
        =p\normfro{\mtx\Sigma_2}^2,\\
        \E_{\mtx\Omega_2}
        \normfro{\mtx\Sigma_2\mtx\Omega_2\mtx\Omega_1^\dagger}^2
        &=
        \normfro{\mtx\Sigma_2}^2
        \normfro{\mtx\Omega_1^\dagger}^2.
    \end{aligned}
\]
The pseudoinverse of a Gaussian matrix has the following bounds
\[
 \E\normfro{\mtx\Omega_1^\dagger}^2\leq k/(p-1), \qquad  \E\normtwo{\mtx\Omega_1^\dagger}^2
    \leq {e^2(k+p)}/{(p^2-1)},
\]
where the Frobenius norm bound is in \cite[Proposition~A.6]{halko2011finding}, and
the spectral norm bound is in \cite[Lemma~3.1]{nakatsukasa2020fast} for $k\geq2$.
For $k=1$, the spectral norm bound is true as well, since the spectral and Frobenius norms coincide.
Using Minkowski's inequality in $L^2$, applied to
\cref{eq:appendix-range-two-terms,eq:appendix-rank-p-competitor} and combined
with \cref{eq:appendix-range-split}, gives the result.
\end{proof} \subsection{One-Level Recovery with Sketch Reuse}
\label{sec:app-one-level-recovery}

The preceding estimate controls the far-field residual left by basis recovery. We now account for the additional error from extracting nearby coefficients with the same sketches. The one-level recovery estimate also holds when only
a subset $\mathcal T$ of boxes is compressed. In this case, for $\tau\notin\mathcal T$, we choose $\mtx U_\tau=\mtx V_\tau=\mtx I$ and set $\msf R_\tau$ to the empty set. The maxima in both the sample-size condition and the error bound
are then taken over $\tau\in\mathcal T$.

\begin{proof}[Proof of \Cref{lem:one-level-recovery}]
To compare recovery with an arbitrary admissible approximation,
fix $\overline{\mtx A}\in\mathcal M$ and, for each box $\tau$, set
\[
\eta_{\tau,\mathrm{row}}
:=\normfro{\mtx A_{\msf B,\msf F}
           -\overline{\mtx A}_{\msf B,\msf F}},\qquad
\eta_{\tau,\mathrm{col}}
:=\normfro{\mtx A_{\msf F,\msf B}
           -\overline{\mtx A}_{\msf F,\msf B}}.
\]
The corresponding far-field blocks of $\overline{\mtx A}$
have rank at most $k_\tau$.

Using unitary invariance, we work with transformed coordinates
\(\mtx C=\mtx U^*\mtx A\mtx V\).
Fix a box $\tau$ and consider
$
    [\mtx C_{\msf R,\msf B}\quad
     \mtx C_{\msf R,\msf N}\quad
     \mtx C_{\msf R,\msf F}].
$
The first two blocks are recovered by block extraction, while
\(\mtx C_{\msf R,\msf F}\) is discarded. 
Let $\mtx G:=\mtx\Omega_{\msf B\cup\msf N}$ and
$\widehat{\mtx G}:=\widehat{\mtx\Omega}_{\msf B\cup\msf N}$.
Let $\mtx H_1$ contain the $k_\tau+p$ orthonormal nullspace
vectors of $\mtx G$ used in block nullification, and let
$\mtx H_2$ have orthonormal columns spanning
$\range(\mtx G^*)$.
Conditional on $\mtx G$ and $\mtx V$,
\[
    \mtx\Omega_{\msf F,1}
    :=\mtx\Omega_{\msf F}\mtx H_1,
    \qquad
    \mtx\Omega_{\msf F,2}
    :=\mtx\Omega_{\msf F}\mtx H_2
\]
are independent standard Gaussian matrices.
Block nullification uses $\mtx\Omega_{\msf F,1}$ as its
test matrix. By \cref{lem:gaussian-range-estimate},
the discarded far-field block satisfies
\begin{equation}
\E\!\left[
  \normfro{\mtx C_{\msf R,\msf F}}^2
  \mid \mtx G,\mtx V
\right]\leq \alpha_{k_\tau,p}^2 \eta_{\tau,\rm {row}}^2.
\label{eq:error_tau_row} 
\end{equation}

It remains to account for the error introduced by block extraction. By \cref{eq:one-level-block-extraction},
\[ [\widetilde{\mtx C}_{\msf R,\msf B}\quad \widetilde{\mtx C}_{\msf R,\msf N}] - [\mtx C_{\msf R,\msf B}\quad \mtx C_{\msf R,\msf N}] = \mtx C_{\msf R,\msf F} \widehat{\mtx\Omega}_{\msf F} \pinv{\widehat{\mtx G}}. \] 
Since $\range(\pinv{\mtx G})=\range(\mtx G^*)=\range(\mtx H_2)$, \[ \widehat{\mtx\Omega}_{\msf F}\pinv{\widehat{\mtx G}} = \mtx V_{\msf F}^* \mtx\Omega_{\msf F}\mtx H_2 \mtx H_2^*\pinv{\mtx G} \mtx V_{\msf B\cup\msf N}. \]
With $\mtx G$, $\mtx V$, and $\mtx\Omega_{\msf F,1}$ fixed,
$\mtx C_{\msf R,\msf F}$ is fixed, and the remaining randomness in the
extraction term enters only through the independent Gaussian matrix
$\mtx\Omega_{\msf F,2}$. We take expectations first over $\mtx \Omega_{\msf F,2}$, then apply the conditional range estimate and take expectations over $\mtx G, \mtx V$. The  identities
\cite[Propositions~A.1 and A.6]{halko2011finding} give
 \[ \E\!\left[ \normfro{\mtx C_{\msf R,\msf F}}^2 + \normfro{ \mtx C_{\msf R,\msf F} \widehat{\mtx\Omega}_{\msf F} \pinv{\widehat{\mtx G}} }^2 \right] \leq \beta_\tau^2\eta_{\tau,\mathrm{row}}^2. \] 
 The adjoint argument gives the bound $\beta_\tau^2\eta_{\tau,\mathrm{col}}^2$ for the residual block column.

Using the forward estimate on overlapping $\msf{RR}$ blocks, every error
block is controlled by at least one of the row or column estimates.  Summing
these estimates counts each corresponding competitor block at most twice; therefore
\[ \E\normfro{ \mtx A- \widetilde{\mtx A}}^2 \leq 2 \left( \max_{\tau}\beta_\tau^2 \right) \normfro{ \mtx A -\overline{\mtx A} }^2. \] Taking the infimum over all \(\overline{\mtx A} \in \mathcal M\), then taking the square root, and applying Jensen's inequality gives the result. 
\end{proof} \subsection{Multilevel Recovery and Orthogonality}
\label{app:multilevel-proof}

We first establish orthogonality of the exact level contributions
and the levelwise recovery errors. We then combine this identity
with the one-level estimate to prove
\Cref{thm:resketched-quasioptimality}.

\begin{proof}[Proof of \Cref{lem:orthogonality-levels}]
For distinct levels $\ell, \ell'$, assume $ \ell' < \ell$.
By nestedness, the coarser contribution $\mtx D_{\ell'}$ is supported
entirely in the $\msf S_{\ell}, \msf S_{\ell}$ block when expressed
in level-$\ell$ coordinates as
\[
\begin{aligned}
\mtx D_{\ell'}
&=
\begin{bmatrix}
\mathcal U_{\msf S_\ell} & \mathcal U_{\msf R_\ell}
\end{bmatrix}
\begin{bmatrix}
\mtx B & \mtx0\\
\mtx0 & \mtx0
\end{bmatrix}
\begin{bmatrix}
\mathcal V_{\msf S_\ell} & \mathcal V_{\msf R_\ell}
\end{bmatrix}^{*},
\\
\mtx D_{\ell}
&=
\begin{bmatrix}
\mathcal U_{\msf S_\ell} & \mathcal U_{\msf R_\ell}
\end{bmatrix}
\begin{bmatrix}
\mtx0 & \mtx C_{\msf S_\ell,\msf R_\ell}\\
\mtx C_{\msf R_\ell,\msf S_\ell} & \mtx C_{\msf R_\ell,\msf R_\ell}
\end{bmatrix}
\begin{bmatrix}
\mathcal V_{\msf S_\ell} & \mathcal V_{\msf R_\ell}
\end{bmatrix}^{*}.
\end{aligned}
\]
The coefficient matrices have disjoint block support, and
the common basis matrices have orthonormal columns.
Hence $\langle\mtx D_\ell,\mtx D_{\ell'}\rangle_F=0$.
The terminal component \(\mtx W_{\ell_0}\) likewise occupies only the \(\msf S \msf S\) block at every level and is therefore orthogonal to each level contribution.
The same argument applies to the recovered contributions
and their errors.
\end{proof}

\begin{proof}[Proof of \Cref{thm:resketched-quasioptimality}]

Fix a level $\ell$, and condition on all test matrices drawn at finer levels.
Then the current core matrix $\mtx C:=\mtx C_{\msf S_{\ell+1},\msf S_{\ell+1}}$
 is fixed.
Resketched recovery uses
$\widehat{\mtx Y}=\mtx C\widehat{\mtx\Omega}$ and
$\widehat{\mtx Z}=\mtx C^*\widehat{\mtx\Psi}$
with fresh independent Gaussian test matrices.
The one-level procedure produces completed transforms
$\mtx U,\mtx V$ and neighboring $\msf{SR}$, $\msf{RS}$,
and $\msf{RR}$ coefficients, which assemble the
physical-coordinate contribution $\widetilde{\mtx D}_\ell$.
Completing the recovered coefficients with the exact
$\msf{SS}$ block, as in \Cref{lem:one-level-recovery}, gives
\[
\mtx E_\ell = \mtx D_{\ell} - \widetilde{\mtx D}_{\ell}
=
\mathcal U_{\msf S_{\ell+1}}
\left(
\mtx C-\mtx U
\begin{bmatrix}
\mtx C_{\msf S_\ell,\msf S_\ell}
&
\widetilde{\mtx C}_{\msf S_\ell,\msf R_\ell}
\\
\widetilde{\mtx C}_{\msf R_\ell,\msf S_\ell}
&
\widetilde{\mtx C}_{\msf R_\ell,\msf R_\ell}
\end{bmatrix}
\mtx V^*
\right)
\mathcal V_{\msf S_{\ell+1}}^*,
\]
where the recovered blocks vanish on nonneighboring pairs.
Let $\operatorname{opt}_\ell:=\operatorname{opt}(\mtx C)$
as defined in \eqref{eq:one-level-opt} for the current
partition, with prescribed ranks on $\mathcal T_{\ell}$ and coarser leaves in $\mathcal B_{\ell} \setminus \mathcal T_{\ell}$ left uncompressed.
Since the surviving coordinate maps have orthonormal columns, the squared estimate in the proof of \Cref{lem:one-level-recovery} gives
\[
\mathbb E\ \normfro{\mtx E_\ell}^2
\le 2\beta^2\operatorname{opt}_\ell^2.
\]

For any $\overline{\mtx A}\in
\mathcal H^2$, set
$\overline{\mtx C}:=
\mathcal U_{\msf S_{\ell+1}}^*
\overline{\mtx A}
\mathcal V_{\msf S_{\ell+1}}$.
Because these transforms act boxwise, they preserve neighbor
support and rank bounds on each far-field block
row and column, so $\overline{\mtx C}$
is an admissible one-level competitor.
Therefore,
\[
\operatorname{opt}_\ell
\le \normfro{\mtx C-\overline{\mtx C}}
\le \normfro{\mtx A-\overline{\mtx A}},
\]
and taking the infimum gives
$\operatorname{opt}_\ell\le\mathrm{OPT}$.
Orthogonality holds for every realization, and the terminal
block $\mtx C_{\msf S_{\ell_0},\msf S_{\ell_0}}$
is recovered exactly. Taking expectations over all test
matrices therefore gives
\[
\mathbb E\normfro{\mtx A-\widetilde{\mtx A}}^2
=
\textstyle{\sum}_{\ell=\ell_0}^{L}\ \mathbb E\normfro{\mtx E_\ell}^2
\le 2(L-\ell_0+1)\beta^2\mathrm{OPT}^2.
\]
Applying Jensen's inequality proves the claim.
\end{proof} 
\section{Supplementary Proofs and Algorithms}
\label{app:pseudocode}
This appendix provides the algebraic equivalence proof and detailed
algorithms. We first regroup the telescoping representation into the
conventional interaction-list form, then describe matrix--vector
multiplication with nested bases and give pseudocode for one-level,
resketched, and carried recovery.

\subsection{Equivalence with the Conventional Hierarchical Form}
\label{app:interaction-list-equivalence}

\begin{proof}[Proof of \Cref{prop:interaction-list-equivalence}]
Fix the box bases of bounded rank, and let $\widetilde{\mtx A}$ denote the
approximation in \eqref{eq:levelwise-expansion}.
The regrouping below uses only the near-neighbor sparsity
pattern of the retained level contributions and the identity
$\mtx D_\ell=\mtx W_{\ell+1}-\mtx W_\ell$. 
No orthogonality or nestedness of the box bases is used
in this regrouping.

Because the coordinate maps act boxwise, restriction to
neighboring coefficient blocks commutes with their embedding
in physical coordinates. Thus
\begin{equation}
\label{eq:h-matrix-telescoping-form}
    \widetilde{\mtx A}
    =
    \mtx W_{\ell_0}
    +\displaystyle{\sum}_{\ell=\ell_0}^{L}\ 
        \blockrestrict{\mtx D_\ell}{\mathcal N_\ell}.
\end{equation}
Here $\mathcal N_\ell$ denotes the full near-neighbor mask
on $\mathcal B_\ell$, identified with the corresponding
physical matrix entries. Pairs of coarser leaves contribute
zero, since $\mtx W_{\ell+1}$ and $\mtx W_\ell$ both equal
$\mtx A$ on those blocks.

To compare this telescoping representation with the conventional
representation, we identify the level at which each pair
first becomes well separated. The near masks are nested, 
$
    \mathcal N_L
    \subseteq
    \mathcal N_{L-1}
    \subseteq\cdots\subseteq
    \mathcal N_{\ell_0},
$
since descendants of a well-separated pair remain well separated.  
For $\ell>\ell_0$, $\mathcal I_\ell$ consists of the pairs
that are still near on the preceding level but become
well separated on level $\ell$,
\begin{equation}
\label{eq:interaction_list}
    \mathcal I_\ell
    :=
    \mathcal N_{\ell-1}\setminus\mathcal N_\ell\quad \text{for}\ \ell_0<\ell\leq L
    ,\quad \text{and}\quad
    \mathcal I_{\ell_0}
    :=
    (\mathcal N_{\ell_0})^c.
\end{equation}
The masks $\mathcal I_{\ell_0},\ldots,\mathcal I_L,\mathcal N_L$ 
form a disjoint partition of the physical matrix entries.

With the same box bases, the conventional approximation
retains $\mtx A$ on neighboring leaf pairs and uses
$\mtx W_\ell$ on pairs first separated at level $\ell$.
Therefore
\begin{equation}
\label{eq:h-matrix-interaction-form}
{\mtx A}_{\mathcal H}:=
\displaystyle{\sum}_{\ell=\ell_0}^{L}\ 
\blockrestrict{\mtx W_\ell}{\mathcal I_\ell} + \blockrestrict{\mtx A}{\mathcal N_L}.
\end{equation}
On the 2:1-balanced trees used here, this is the standard U-, V-, W-, and
X-list organization \cite[Sec.~3.3.1]{ying2004kernel}: the U-list contains
the retained leaf-neighbor interactions, while the V-, W-, and X-lists
assign each well-separated pair to the first level at which it becomes
admissible.

To show that $\mtx A_{\mathcal H}=\widetilde{\mtx A}$,
we regroup this interaction-list expression into the
telescoping form above. Expanding each interaction mask as a difference of near masks
and collecting terms restricted to the same mask gives
\[
\begin{aligned}
    {\mtx A}_{\mathcal H}
    &=
    \mtx W_{\ell_0}
    +
    \blockrestrict{
        \mtx W_{\ell_0+1}-\mtx W_{\ell_0}}
        {\mathcal N_{\ell_0}}
    +\cdots 
    +
    \blockrestrict{
        \mtx W_L-\mtx W_{L-1}}
        {\mathcal N_{L-1}}
    +
    \blockrestrict{
        \mtx A-\mtx W_L}
        {\mathcal N_L}.
\end{aligned}
\]
Using
$\mtx D_\ell=\mtx W_{\ell+1}-\mtx W_\ell$ and 
$\mtx W_{L+1}=\mtx A$,
we conclude that ${\mtx A}_{\mathcal H}= \widetilde{\mtx A}$.
\end{proof}

\begin{algorithm}[!htbp]
\caption{Boxwise matrix--vector multiplication}
\label{alg:adaptive-tree-matvec}
\small
\begin{algorithmic}[1]
\Require Vector $\mtx q$; tree; box bases
    $\mtx U_\tau=[\mtx U_{\msf S_\tau}\ \mtx U_{\msf R_\tau}]$ and
    $\mtx V_\tau=[\mtx V_{\msf S_\tau}\ \mtx V_{\msf R_\tau}]$; stored
    ordered-pair blocks $\mtx C_{\rho\sigma}$
\Ensure Vector $\mtx u=\widetilde{\mtx A}\mtx q$ produced by the stored
    representation
\State For every leaf $\lambda$, set $\mtx q_\lambda$ to its entries in
    $\mtx q$ and $\mtx u_\lambda=\mtx0$; for every processed box $\tau$, set
    $\widehat{\mtx u}_{\msf S_\tau}
     =\widehat{\mtx u}_{\msf R_\tau}=\mtx0$
\vspace{0.4ex}
\For{$\ell=L,L-1,\ldots,\ell_0$} \Comment{upward pass}
    \ForAll{$\tau\in\mathcal T_\ell$}
        \State $\displaystyle
            \begin{bmatrix}\widehat{\mtx q}_{\msf S_\tau}\\
                            \widehat{\mtx q}_{\msf R_\tau}\end{bmatrix}
            \gets \mtx V_\tau^*\mtx q_\tau$ if $\tau\in\mathcal L$, and
            $\mtx V_\tau^*\widehat{\mtx q}_\tau$ otherwise
    \EndFor
    \ForAll{nonleaf $\rho\in\mathcal T_{\ell-1}$}
        \State $\displaystyle \widehat{\mtx q}_\rho\gets
            [\widehat{\mtx q}_{\msf S_\tau}]_{
            \tau\in\operatorname{ch}(\rho)}$
    \EndFor
\EndFor
\vspace{0.4ex}
\ForAll{stored ordered same-level pair $(\rho,\sigma)$}
    \Comment{interaction pass}
    \State $\displaystyle
        \begin{bmatrix}\widehat{\mtx u}_{\msf S_\rho}\\
                        \widehat{\mtx u}_{\msf R_\rho}\end{bmatrix}
        \mathrel{+}=
        \begin{bmatrix}
            \mtx0 & \mtx C_{\msf S_\rho,\msf R_\sigma}\\
            \mtx C_{\msf R_\rho,\msf S_\sigma} &
            \mtx C_{\msf R_\rho,\msf R_\sigma}
        \end{bmatrix}
        \begin{bmatrix}\widehat{\mtx q}_{\msf S_\sigma}\\
                        \widehat{\mtx q}_{\msf R_\sigma}\end{bmatrix}$
\EndFor
\vspace{0.25ex}
\ForAll{stored ordered coarse pair $(\rho,\lambda)$, with $\lambda$ a
    coarser leaf}
    \State $\displaystyle
        \begin{bmatrix}\widehat{\mtx u}_{\msf S_\rho}\\
                        \widehat{\mtx u}_{\msf R_\rho}\end{bmatrix}
        \mathrel{+}=
        \begin{bmatrix}\mtx0\\ \mtx C_{\msf R_\rho,\lambda}\end{bmatrix}
        \mtx q_\lambda$
\EndFor
\vspace{0.25ex}
\ForAll{stored ordered fine pair $(\lambda,\rho)$, with $\lambda$ a
    coarser leaf}
    \State $\mtx u_\lambda\mathrel{+}=
        \begin{bmatrix}\mtx0 & \mtx C_{\lambda,\msf R_\rho}\end{bmatrix}
        \begin{bmatrix}\widehat{\mtx q}_{\msf S_\rho}\\
                        \widehat{\mtx q}_{\msf R_\rho}\end{bmatrix}$
\EndFor
\vspace{0.25ex}
\ForAll{$\rho,\sigma\in\mathcal T_{\ell_0}$}
    \Comment{coarsest coefficient matrix}
    \State $\widehat{\mtx u}_{\msf S_\rho}\mathrel{+}=
        \mtx C_{\msf S_\rho,\msf S_\sigma}
        \widehat{\mtx q}_{\msf S_\sigma}$
\EndFor
\vspace{0.4ex}
\For{$\ell=\ell_0,\ldots,L$} \Comment{downward pass}
    \ForAll{$\tau\in\mathcal T_\ell$}
        \State $\widehat{\mtx u}_{\msf S_\tau}\mathrel{+}=
            (\widehat{\mtx u}_{\operatorname{par}(\tau)})_\tau$
        \State $\displaystyle \mtx u_\tau\mathrel{+}=\mtx U_\tau
            \begin{bmatrix}\widehat{\mtx u}_{\msf S_\tau}\\
                            \widehat{\mtx u}_{\msf R_\tau}\end{bmatrix}$
            if $\tau\in\mathcal L$, and
            $\displaystyle \widehat{\mtx u}_\tau\gets\mtx U_\tau
            \begin{bmatrix}\widehat{\mtx u}_{\msf S_\tau}\\
                            \widehat{\mtx u}_{\msf R_\tau}\end{bmatrix}$ otherwise
    \EndFor
\EndFor
\State \Return $\mtx u=[\mtx u_\lambda]_{\lambda\in\mathcal L}$
\end{algorithmic}
\end{algorithm}

\subsection{Matrix--Vector Multiplication with Nested Bases}
\label{app:deterministic-matvec}

For each retained ordered pair $(\rho,\sigma)$, let
$\mtx C_{\rho\sigma}$ denote its coefficient block in the completed
coordinates of \cref{eq:levelwise-expansion}. These are the stored coefficients of the telescoping representation.  The interaction pass in
\Cref{alg:adaptive-tree-matvec} lists the same-level (colleague), coarse, and
fine pair types explicitly.  Here
$\ell_0$ is the coarsest transformed level: $\ell_0=2$ for strong
admissibility and $\ell_0=1$ for weak admissibility.  For simplicity, assume
$\mathcal B_{\ell_0}=\mathcal T_{\ell_0}$, so no coarser leaves occur in the
coarsest partition.  For weak admissibility, there are no coarse or fine
neighbor interactions.  The upward pass gathers each child's surviving coefficients
at its parent. The downward pass distributes the parent
contribution to its children's surviving coordinates.
At level $\ell_0$, this contribution is zero.

\begin{remark}[Matvec storage]
Only the coarse and fine interactions require both physical and transformed
box vectors, because the coarser leaf remains in physical coordinates.
Same-level interactions and the coarsest coefficient matrix use transformed
coordinates throughout.  To avoid storing both versions of the charge and
potential vectors, the local basis factors can instead be absorbed into the
coarse and fine interaction blocks.
\end{remark}
 \vspace{-0.5em}

\begin{algorithm}[!htbp]
\caption{One-level boxwise recovery}
\label{alg:recover-boxwise}
\small
\begin{algorithmic}[1]
\Require Boxes $\mathcal B$ with dimensions $m_\tau$ and neighbor lists
    $\near(\tau)\ni\tau$; compressed subset $\mathcal T$ with ranks
    $1\leq k_\tau\leq m_\tau$
\Require $p\geq2$ and
    $s\geq\max_{\tau\in\mathcal T}(n_\tau+k_\tau+p)$, where
    $n_\tau:=\sum_{\sigma\in\near(\tau)}m_\sigma$
\Require Test matrices $\mtx\Omega_\tau,\mtx\Psi_\tau$ and sketches
    $\mtx Y_\tau,\mtx Z_\tau$, all with $s$ columns
\vspace{0.35ex}
\Statex For exact sketches and $\tau\in\mathcal B$,
    $\mtx Y_\tau=\sum_{\sigma\in\mathcal B}
    \mtx C_{\tau,\sigma}\mtx\Omega_\sigma$ and
    $\mtx Z_\tau=\sum_{\sigma\in\mathcal B}
    \mtx C_{\sigma,\tau}^*\mtx\Psi_\sigma$.
\Statex For $\tau\in\mathcal T$, write
    $\msf B=\{\tau\}$, $\msf N=\near(\tau)\setminus\{\tau\}$,
    $\msf S=\msf S_\tau$, and $\msf R=\msf R_\tau$.
\vspace{0.45ex}
\ForAll{$\tau\in\mathcal T$} \Comment{parallel block nullification}
    \State $\displaystyle \mtx G_\Omega\gets
        \mtx\Omega_{\msf B \cup \msf N}$,
        $\mtx K_\Omega\gets\operatorname{null}_{k_\tau+p}(\mtx G_\Omega)$;
        define $\mtx G_\Psi,\mtx K_\Psi$ analogously
    \State $\mtx U_{\msf S}\gets
        \operatorname{orth}_{k_\tau}(\mtx Y_{\msf B}\mtx K_\Omega)$,
        $\mtx V_{\msf S}\gets
        \operatorname{orth}_{k_\tau}(\mtx Z_{\msf B}\mtx K_\Psi)$
        \label{algline:recover-boxwise-nullified-sketches}
    \State Complete to unitary transforms
        $\mtx U_\tau=[\mtx U_{\msf S}\ \mtx U_{\msf R}]$,
        $\mtx V_\tau=[\mtx V_{\msf S}\ \mtx V_{\msf R}]$
\EndFor
\vspace{0.45ex}
\State Set $\mtx U_\tau=\mtx V_\tau=\mtx I$,
    $\msf S_\tau=\tau$, and $\msf R_\tau=\varnothing$ for
    $\tau\in\mathcal B\setminus\mathcal T$
\State $\mtx U\gets\diag_{\tau\in\mathcal B}\mtx U_\tau$ and
    $\mtx V\gets\diag_{\tau\in\mathcal B}\mtx V_\tau$
\ForAll{$\tau\in\mathcal B$}
    \Comment{rotate test and sketch matrices in place}
    \State $\displaystyle
        \left(
        \begin{bmatrix}\widehat{\mtx Y}_{\msf S_\tau}\\
                        \widehat{\mtx Y}_{\msf R_\tau}\end{bmatrix},
        \begin{bmatrix}\widehat{\mtx\Psi}_{\msf S_\tau}\\
                        \widehat{\mtx\Psi}_{\msf R_\tau}\end{bmatrix}
        \right)
        \gets\mtx U_\tau^*(\mtx Y_\tau,\mtx\Psi_\tau)$
    \State Apply the analogous overwrite with $\mtx V_\tau^*$ to
        $(\mtx Z_\tau,\mtx\Omega_\tau)$
\EndFor
\vspace{0.45ex}
\ForAll{$\tau\in\mathcal T$} \Comment{parallel block extraction}
    \State $\displaystyle \widehat{\mtx G}_\Omega\gets
        \widehat{\mtx\Omega}_{\msf B \cup \msf N}$ and
        $\displaystyle \widehat{\mtx G}_\Psi\gets
        \widehat{\mtx\Psi}_{\msf B \cup \msf N}$
    \State $\displaystyle
        [\,\widetilde{\mtx C}_{\msf R,\msf B}\quad
            \widetilde{\mtx C}_{\msf R,\msf N}\,]
        \gets\widehat{\mtx Y}_{\msf R}
        \widehat{\mtx G}_\Omega^\dagger$
    \State $\displaystyle
        [\,\widetilde{\mtx C}_{\msf B,\msf R}^*\quad
            \widetilde{\mtx C}_{\msf N,\msf R}^*\,]
        \gets\widehat{\mtx Z}_{\msf R}
        \widehat{\mtx G}_\Psi^\dagger$
\EndFor
\vspace{0.35ex}
\State Store these extracted subblocks, using the $\widehat{\mtx Y}$ estimate
    on overlapping $\msf{RR}$ blocks
\State \Return recovered box transforms and neighboring coefficient blocks.
\end{algorithmic}
\end{algorithm}

\subsection{Randomized Recovery Algorithms}
\label{app:randomized-algorithms}

This section gives the one-level procedure and the two multilevel variants in
\Cref{sec:one-level-prd,sec:algorithms}.  Each multilevel method supplies the
one-level routine with boxwise test matrices and sketches in the incoming
coordinates.  In Algorithm~\ref{alg:recover-boxwise}, inputs are written
without hats; hats denote their values after the recovered box transforms are
applied.  We write \texttt{single\_level\_recovery} for this routine.  At the
finest level, its inputs are the physical test matrices and sketches.  Carried
test matrices are reused across levels, while resketched tests are redrawn
at each level.
Here $\operatorname{null}_r(\mtx G)$ denotes a matrix with
$r$ orthonormal columns in the nullspace of $\mtx G$, and
$\operatorname{orth}_r(\mtx M)$ denotes the leading $r$ left singular vectors
of $\mtx M$; $k_\tau+p$ null vectors suffice.
The one-level routine stores the recovered row and column subblocks in
$\widetilde{\mtx C}$.

\FloatBarrier

For the multilevel routines, a $\tau$ subscript denotes the
coordinates entering the current box transform; $\msf S_\tau$ and
$\msf R_\tau$ denote its output blocks.  Algorithm~\ref{alg:recover-boxwise}
overwrites the hatted arrays on $\mathcal T_\ell$ with these rotated blocks;
boxes not compressed at that level retain their $\tau$ coordinates.  The box maps
$\mathcal U_\tau,\mathcal V_\tau$ are defined in
\Cref{subsec:nested-coordinates}.  They are applied recursively, and only the local factors
$\mtx U_\tau,\mtx V_\tau$ are stored.  At level $\ell$, the one-level routine uses
$(\mathcal B,\mathcal T,\near)
=(\mathcal B_\ell,\mathcal T_\ell,\near_\ell)$. The $\msf R_\ell$ coefficient blocks refer to the recovered
neighboring $\msf{SR}$, $\msf{RS}$, and $\msf{RR}$ blocks
at level $\ell$.

\begin{algorithm}[!htbp]
\caption{Resketched multilevel recovery}
\label{alg:resketched-recovery}
\small
\begin{algorithmic}[1]
\Require $\mtx A$, $\mtx A^*$; current box partitions
    $\mathcal B_L,\ldots,\mathcal B_{\ell_0}$; compressed sets
    $\mathcal T_L,\ldots,\mathcal T_{\ell_0}$; neighbor lists; ranks
    $k_\tau$; oversampling $p$; levelwise sample counts
    $s_\ell\geq\max_{\tau\in\mathcal T_\ell}(n_\tau+k_\tau+p)$
\vspace{0.45ex}
\For{$\ell=L,L-1,\ldots,\ell_0$}
    \State For $\tau\in\mathcal B_\ell$, draw Gaussian
        $\widehat{\mtx\Omega}_\tau,\widehat{\mtx\Psi}_\tau$ with
        $s_\ell$ columns
    \State For $\tau\in\mathcal B_\ell$, set
        $\mtx\Omega_\tau\gets\mathcal V_\tau
        \widehat{\mtx\Omega}_\tau$ and
        $\mtx\Psi_\tau\gets\mathcal U_\tau
        \widehat{\mtx\Psi}_\tau$
    \State Form sketches
        $\mtx Y_\tau=\sum_{\sigma\in\mathcal B_\ell}
        \mtx A_{\tau\sigma}\mtx\Omega_\sigma$ and
        $\mtx Z_\tau=\sum_{\sigma\in\mathcal B_\ell}
        \mtx A_{\sigma\tau}^*\mtx\Psi_\sigma$
    \State For $\tau\in\mathcal B_\ell$, set
        $\widehat{\mtx Y}_\tau\gets\mathcal U_\tau^*\mtx Y_\tau$ and
        $\widehat{\mtx Z}_\tau\gets\mathcal V_\tau^*\mtx Z_\tau$
    \State Call $\texttt{single\_level\_recovery}
        (\mathcal B_\ell,\mathcal T_\ell,\near_\ell,
        \widehat{\mtx Y}_\tau,\widehat{\mtx Z}_\tau,
        \widehat{\mtx\Omega}_\tau,\widehat{\mtx\Psi}_\tau)$
    \State Store $\{\mtx U_\tau,\mtx V_\tau\}_{\tau\in\mathcal T_\ell}$ and the retained $\msf R_{\ell}$ coefficients of $\widetilde{\mtx C}$
\EndFor
\vspace{0.45ex}
\State Compute and store
    ${\mtx C}_{\msf S_{\ell_0},\msf S_{\ell_0}}\gets
    \mathcal U_{\msf S_{\ell_0}}^*
    (\mtx A\mathcal V_{\msf S_{\ell_0}})$
\vspace{0.35ex}
\State \Return recovered factors $\mtx U_\tau,\mtx V_\tau$ for
    $\tau\in\mathcal T_\ell$ at each level and stored coefficient blocks
    $\widetilde{\mtx C}$
\end{algorithmic}
\end{algorithm}
 
\FloatBarrier

Carried recovery subtracts the recovered neighboring
$\msf{SR}$ and $\msf{RS}$ contributions from the surviving
forward and adjoint sketch rows, respectively, and passes
these rows to the next level. See
Algorithm~\ref{alg:carried-recovery}.

\begin{algorithm}[!htbp]
\caption{Carried-sketch multilevel recovery}
\label{alg:carried-recovery}
\small
\begin{algorithmic}[1]
\Require $\mtx A$, $\mtx A^*$; current box partitions
    $\mathcal B_L,\ldots,\mathcal B_{\ell_0-1}$; compressed sets
    $\mathcal T_L,\ldots,\mathcal T_{\ell_0}$; neighbor lists; ranks $k_\tau$;
    oversampling $p$; sample count $s_{\mathrm{carried}}$ from
    \cref{eq:rank-query-count}
\vspace{0.45ex}
\State For $\tau\in\mathcal L$, draw Gaussian
    $\mtx\Omega_\tau,\mtx\Psi_\tau$ with $s_{\mathrm{carried}}$ columns
\Statex \hspace{\algorithmicindent}$\displaystyle
    \mtx Y_\tau=\sum_{\sigma\in\mathcal L}
    \mtx A_{\tau\sigma}\mtx\Omega_\sigma,
    \qquad
    \mtx Z_\tau=\sum_{\sigma\in\mathcal L}
    \mtx A_{\sigma\tau}^*\mtx\Psi_\sigma,
    \qquad \tau\in\mathcal L.$
\State For $\tau\in\mathcal L$, initialize
    $(\widehat{\mtx Y}_\tau,\widehat{\mtx Z}_\tau,
      \widehat{\mtx\Omega}_\tau,\widehat{\mtx\Psi}_\tau)
    \gets(\mtx Y_\tau,\mtx Z_\tau,\mtx\Omega_\tau,\mtx\Psi_\tau)$
\vspace{0.45ex}
\For{$\ell=L,L-1,\ldots,\ell_0$}
    \State Call $\texttt{single\_level\_recovery}$ with
        $(\mathcal B_\ell,\mathcal T_\ell,\near_\ell,
        \widehat{\mtx Y}_\tau,\widehat{\mtx Z}_\tau,
        \widehat{\mtx\Omega}_\tau,\widehat{\mtx\Psi}_\tau)$
    \State Store
        $\{\mtx U_\tau,\mtx V_\tau\}_{\tau\in\mathcal T_\ell}$, $\msf R_{\ell}$ coefficients of
        $\widetilde{\mtx C}$
    \Statex \hspace{\algorithmicindent}The call overwrites the hatted arrays on
        $\mathcal T_\ell$ by their $\msf S_\tau/\msf R_\tau$ row blocks
    \vspace{0.25ex}
    \ForAll{$\tau\in\mathcal B_\ell$}
        \State $\displaystyle
            \widehat{\mtx Y}_{\msf S_\tau}\mathrel{-}=
            \sum_{\rho\in\near_\ell(\tau)}
            \widetilde{\mtx C}_{\msf S_\tau,\msf R_\rho}
            \widehat{\mtx\Omega}_{\msf R_\rho}$
        \State $\displaystyle
            \widehat{\mtx Z}_{\msf S_\tau}\mathrel{-}=
            \sum_{\rho\in\near_\ell(\tau)}
            \widetilde{\mtx C}_{\msf R_\rho,\msf S_\tau}^{*}
            \widehat{\mtx\Psi}_{\msf R_\rho}$
    \EndFor
    \vspace{0.25ex}
    \ForAll{nonleaf $\tau\in\mathcal B_{\ell-1}$}
        \State $(\widehat{\mtx Y}_\tau,\widehat{\mtx Z}_\tau,
            \widehat{\mtx\Omega}_\tau,\widehat{\mtx\Psi}_\tau)
            \gets[(\widehat{\mtx Y}_{\msf S_\nu},
            \widehat{\mtx Z}_{\msf S_\nu},
            \widehat{\mtx\Omega}_{\msf S_\nu},
            \widehat{\mtx\Psi}_{\msf S_\nu})]
            _{\nu\in\operatorname{ch}(\tau)}$
    \EndFor
\EndFor
\vspace{0.45ex}
\State Collect the surviving rows and store
    $\widetilde{\mtx C}_{\msf S_{\ell_0},\msf S_{\ell_0}}
    \gets\widehat{\mtx Y}_{\msf S_{\ell_0}}
    (\widehat{\mtx\Omega}_{\msf S_{\ell_0}})^\dagger$
\vspace{0.35ex}
\State \Return recovered factors $\mtx U_\tau,\mtx V_\tau$ for
    $\tau\in\mathcal T_\ell$ at each level and stored coefficient blocks
    $\widetilde{\mtx C}$
\end{algorithmic}
\end{algorithm}
  
\pagebreak{}

\bibliographystyle{siamplain}
\bibliography{refs}

\begin{thebibliography}{10}

\bibitem{amsel2026fixedsparsity}
{\sc N.~Amsel, T.~Chen, F.~Duman~Keles, D.~Halikias, C.~Musco, and C.~Musco},
  {\em Fixed-sparsity matrix approximation from matrix-vector products}, SIAM
  Journal on Matrix Analysis and Applications, 47 (2026), pp.~483--511,
  \url{https://doi.org/10.1137/25M1742710}.

\bibitem{amsel2026quasioptimal}
{\sc N.~Amsel, T.~Chen, F.~Duman~Keles, D.~Halikias, C.~Musco, C.~Musco, and
  D.~Persson}, {\em Quasi-optimal hierarchically semi-separable matrix
  approximation}, SIAM Journal on Matrix Analysis and Applications, 47 (2026),
  pp.~586--621, \url{https://doi.org/10.1137/25M176622X}.

\bibitem{ballani2016hierarchical}
{\sc J.~Ballani and D.~Kressner}, {\em Matrices with hierarchical low-rank
  structures}, in Exploiting Hidden Structure in Matrix Computations:
  Algorithms and Applications, M.~Benzi and V.~Simoncini, eds., vol.~2173 of
  Lecture Notes in Mathematics, Springer, Cham, 2016, pp.~161--209,
  \url{https://doi.org/10.1007/978-3-319-49887-4_3}.

\bibitem{borm2010efficient}
{\sc S.~B{\"o}rm}, {\em Efficient Numerical Methods for Non-local Operators:
  {$\mathcal{H}^2$}-Matrix Compression, Algorithms and Analysis}, vol.~14 of
  EMS Tracts in Mathematics, European Mathematical Society, 2010.

\bibitem{borm2003introduction}
{\sc S.~B{\"o}rm, L.~Grasedyck, and W.~Hackbusch}, {\em Introduction to
  hierarchical matrices with applications}, Engineering Analysis with Boundary
  Elements, 27 (2003), pp.~405--422,
  \url{https://doi.org/10.1016/S0955-7997(02)00152-2}.

\bibitem{boukaram2025adaptive}
{\sc W.~H. Boukaram, Y.~Liu, P.~Ghysels, and X.~S. Li}, {\em Adaptive sketching
  based construction of {$\mathcal{H}^2$} matrices on {GPUs}}, in 2025 IEEE
  International Parallel and Distributed Processing Symposium Workshops
  (IPDPSW), 2025, pp.~433--442,
  \url{https://doi.org/10.1109/IPDPSW66978.2025.00069}.

\bibitem{boukaram2019randomized}
{\sc W.~H. Boukaram, G.~Turkiyyah, and D.~E. Keyes}, {\em Randomized {GPU}
  algorithms for the construction of hierarchical matrices from matrix-vector
  operations}, SIAM Journal on Scientific Computing, 41 (2019), pp.~C339--C366,
  \url{https://doi.org/10.1137/18M1210101}.

\bibitem{boulle2023elliptic}
{\sc N.~Boull{\'e}, D.~Halikias, and A.~Townsend}, {\em Elliptic {PDE} learning
  is provably data-efficient}, Proceedings of the National Academy of Sciences,
  120 (2023), p.~e2303904120, \url{https://doi.org/10.1073/pnas.2303904120}.

\bibitem{casulli2024functions}
{\sc A.~A. Casulli, D.~Kressner, and L.~Robol}, {\em Computing functions of
  symmetric hierarchically semiseparable matrices}, SIAM Journal on Matrix
  Analysis and Applications, 45 (2024), pp.~2314--2338,
  \url{https://doi.org/10.1137/24M1642354}.

\bibitem{chen2025nearoptimal}
{\sc T.~Chen, F.~Duman~Keles, D.~Halikias, C.~Musco, C.~Musco, and D.~Persson},
  {\em Near-optimal hierarchical matrix approximation from matrix-vector
  products}, in Proceedings of the 2025 Annual ACM-SIAM Symposium on Discrete
  Algorithms (SODA), 2025, pp.~2656--2692,
  \url{https://doi.org/10.1137/1.9781611978322.87}.

\bibitem{engquist2007fast}
{\sc B.~Engquist and L.~Ying}, {\em Fast directional multilevel algorithms for
  oscillatory kernels}, SIAM Journal on Scientific Computing, 29 (2007),
  pp.~1710--1737, \url{https://doi.org/10.1137/07068583X}.

\bibitem{gorman2019matrixfree}
{\sc C.~Gorman, G.~Ch{\'a}vez, P.~Ghysels, T.~Mary, F.-H. Rouet, and X.~S. Li},
  {\em Robust and accurate stopping criteria for adaptive randomized sampling
  in matrix-free hierarchically semiseparable construction}, SIAM Journal on
  Scientific Computing, 41 (2019), pp.~S61--S85,
  \url{https://doi.org/10.1137/18M1194961}.

\bibitem{greengard1987fast}
{\sc L.~Greengard and V.~Rokhlin}, {\em A fast algorithm for particle
  simulations}, Journal of Computational Physics, 73 (1987), pp.~325--348,
  \url{https://doi.org/10.1016/0021-9991(87)90140-9}.

\bibitem{greengard1997new}
{\sc L.~Greengard and V.~Rokhlin}, {\em A new version of the fast multipole
  method for the {L}aplace equation in three dimensions}, Acta Numerica, 6
  (1997), pp.~229--269, \url{https://doi.org/10.1017/S0962492900002725}.

\bibitem{hackbusch1999sparse}
{\sc W.~Hackbusch}, {\em A sparse matrix arithmetic based on
  {$\mathcal{H}$}-matrices. {P}art {I}: Introduction to
  {$\mathcal{H}$}-matrices}, Computing, 62 (1999), pp.~89--108,
  \url{https://doi.org/10.1007/s006070050015}.

\bibitem{hackbusch2015hierarchical}
{\sc W.~Hackbusch}, {\em Hierarchical Matrices: Algorithms and Analysis},
  vol.~49 of Springer Series in Computational Mathematics, Springer, 2015,
  \url{https://doi.org/10.1007/978-3-662-47324-5}.

\bibitem{hackbusch2002data}
{\sc W.~Hackbusch and S.~B{\"o}rm}, {\em Data-sparse approximation by adaptive
  {$\mathcal{H}^2$}-matrices}, Computing, 69 (2002), pp.~1--35,
  \url{https://doi.org/10.1007/s00607-002-1450-4}.

\bibitem{halko2011finding}
{\sc N.~Halko, P.-G. Martinsson, and J.~A. Tropp}, {\em Finding structure with
  randomness: Probabilistic algorithms for constructing approximate matrix
  decompositions}, SIAM Review, 53 (2011), pp.~217--288,
  \url{https://doi.org/10.1137/090771806}.

\bibitem{jiang2025dmk}
{\sc S.~Jiang and L.~Greengard}, {\em A dual-space multilevel kernel-splitting
  framework for discrete and continuous convolution}, Communications on Pure
  and Applied Mathematics, 78 (2025), pp.~1086--1143,
  \url{https://doi.org/10.1002/cpa.22240}.
\newblock arXiv:2308.00292.

\bibitem{levitt2024linear}
{\sc J.~Levitt and P.-G. Martinsson}, {\em Linear-complexity black-box
  randomized compression of rank-structured matrices}, SIAM Journal on
  Scientific Computing, 46 (2024), pp.~A1747--A1763,
  \url{https://doi.org/10.1137/22M1528574}.

\bibitem{levitt2024coloring}
{\sc J.~Levitt and P.-G. Martinsson}, {\em Randomized compression of
  rank-structured matrices accelerated with graph coloring}, Journal of
  Computational and Applied Mathematics, 451 (2024), p.~116044,
  \url{https://doi.org/10.1016/j.cam.2024.116044}.

\bibitem{lin2011fast}
{\sc L.~Lin, J.~Lu, and L.~Ying}, {\em Fast construction of hierarchical matrix
  representation from matrix-vector multiplication}, Journal of Computational
  Physics, 230 (2011), pp.~4071--4087,
  \url{https://doi.org/10.1016/j.jcp.2011.02.033}.

\bibitem{liu2016parallel}
{\sc X.~Liu, J.~Xia, and M.~V. de~Hoop}, {\em Parallel randomized and
  matrix-free direct solvers for large structured dense linear systems}, SIAM
  Journal on Scientific Computing, 38 (2016), pp.~S508--S538,
  \url{https://doi.org/10.1137/15M1023774}.

\bibitem{malhotra2015pvfmm}
{\sc D.~Malhotra and G.~Biros}, {\em {PVFMM}: A parallel kernel independent
  {FMM} for particle and volume potentials}, Communications in Computational
  Physics, 18 (2015), pp.~808--830,
  \url{https://doi.org/10.4208/CICP.020215.150515SW}.

\bibitem{martinsson2011fast}
{\sc P.-G. Martinsson}, {\em A fast randomized algorithm for computing a
  hierarchically semiseparable representation of a matrix}, SIAM Journal on
  Matrix Analysis and Applications, 32 (2011), pp.~1251--1274,
  \url{https://doi.org/10.1137/100786617}.

\bibitem{martinsson2016compressing}
{\sc P.-G. Martinsson}, {\em Compressing rank-structured matrices via
  randomized sampling}, SIAM Journal on Scientific Computing, 38 (2016),
  pp.~A1959--A1986, \url{https://doi.org/10.1137/15M1016679}.

\bibitem{minden2017recursive}
{\sc V.~Minden, K.~L. Ho, A.~Damle, and L.~Ying}, {\em A recursive
  skeletonization factorization based on strong admissibility}, Multiscale
  Modeling \& Simulation, 15 (2017), pp.~768--796,
  \url{https://doi.org/10.1137/16M1095949}.

\bibitem{nakatsukasa2020fast}
{\sc Y.~Nakatsukasa}, {\em Fast and stable randomized low-rank matrix
  approximation}, arXiv preprint arXiv:2009.11392,  (2020).

\bibitem{piessens1983quadpack}
{\sc R.~Piessens, E.~de~Doncker-Kapenga, C.~W. {\"U}berhuber, and D.~K.
  Kahaner}, {\em {QUADPACK}: A Subroutine Package for Automatic Integration},
  vol.~1 of Springer Series in Computational Mathematics, Springer, Berlin,
  Heidelberg, 1983, \url{https://doi.org/10.1007/978-3-642-61786-7}.

\bibitem{schafer2024sparse}
{\sc F.~Sch{\"a}fer and H.~Owhadi}, {\em Sparse recovery of elliptic solvers
  from matrix-vector products}, SIAM Journal on Scientific Computing, 46
  (2024), pp.~A998--A1025, \url{https://doi.org/10.1137/22M154226X}.

\bibitem{sushnikova2023fmmlu}
{\sc D.~Sushnikova, L.~Greengard, M.~O'Neil, and M.~Rachh}, {\em {FMM-LU}: A
  fast direct solver for multiscale boundary integral equations in three
  dimensions}, Multiscale Modeling \& Simulation, 21 (2023), pp.~1570--1601,
  \url{https://doi.org/10.1137/22M1514040}.

\bibitem{xia2013randomized}
{\sc J.~Xia}, {\em Randomized sparse direct solvers}, SIAM Journal on Matrix
  Analysis and Applications, 34 (2013), pp.~197--227,
  \url{https://doi.org/10.1137/12087116X}.

\bibitem{xia2010superfast}
{\sc J.~Xia, S.~Chandrasekaran, M.~Gu, and X.~S. Li}, {\em Superfast
  multifrontal method for large structured linear systems of equations}, SIAM
  Journal on Matrix Analysis and Applications, 31 (2010), pp.~1382--1411.

\bibitem{yesypenko2025skelfmm}
{\sc A.~Yesypenko, C.~Chen, and P.-G. Martinsson}, {\em A simplified fast
  multipole method based on strong recursive skeletonization}, Journal of
  Computational Physics, 524 (2025), p.~113707,
  \url{https://doi.org/10.1016/j.jcp.2024.113707}.

\bibitem{yesypenko2024slablu}
{\sc A.~Yesypenko and P.-G. Martinsson}, {\em {SlabLU}: A two-level sparse
  direct solver for elliptic {PDE}s}, Advances in Computational Mathematics, 50
  (2024), p.~90, \url{https://doi.org/10.1007/s10444-024-10176-x}.

\bibitem{yesypenko2026randomized}
{\sc A.~Yesypenko and P.-G. Martinsson}, {\em Randomized strong recursive
  skeletonization: Simultaneous compression and {LU} factorization of
  hierarchical matrices using matrix-vector products}, Journal of Scientific
  Computing, 106 (2026), p.~63,
  \url{https://doi.org/10.1007/s10915-025-03157-9}.

\bibitem{ying2004kernel}
{\sc L.~Ying, G.~Biros, and D.~Zorin}, {\em A kernel-independent adaptive fast
  multipole algorithm in two and three dimensions}, Journal of Computational
  Physics, 196 (2004), pp.~591--626,
  \url{https://doi.org/10.1016/j.jcp.2003.11.021}.

\end{thebibliography}

\end{document}